\documentclass[a4paper,12pt, twoside, pdftex]{amsart}

\usepackage{fullpage}
\usepackage{amsmath, amsthm, amssymb}
\usepackage{txfonts}
\usepackage{amsfonts}
\usepackage{prettyref}
\usepackage{mathrsfs}
\usepackage[all]{xy}
\usepackage[utf8]{inputenc}

\newtheorem{dummy}{dummy}[section]
\newtheorem{lemma}[dummy]{Lemma}
\newtheorem{theorem}[dummy]{Theorem}
\newtheorem{innercustomthm}{Theorem}
\newenvironment{customthm}[1]
{\renewcommand\theinnercustomthm{#1}\innercustomthm}
  {\endinnercustomthm}

\newtheorem{corollary}[dummy]{Corollary}
\newtheorem{proposition}[dummy]{Proposition}
\theoremstyle{definition}
\newtheorem{definition}[dummy]{Definition}
\newtheorem*{definition*}{Definition}
\newtheorem{example}[dummy]{Example}

\newtheorem{remark}[dummy]{Remark}

\numberwithin{equation}{section}

\usepackage{tikz}
\usetikzlibrary{backgrounds}

\newrefformat{th}{Theorem~\ref{#1}}
\newrefformat{cr}{Corollary~\ref{#1}}
\newrefformat{lm}{Lemma~\ref{#1}}
\newrefformat{dl}{Definition-Lemma~\ref{#1}}
\newrefformat{df}{Definition~\ref{#1}}
\newrefformat{cl}{Claim~\ref{#1}}
\newrefformat{sl}{Sublemma~\ref{#1}}
\newrefformat{pr}{Proposition~\ref{#1}}
\newrefformat{cj}{Conjecture~\ref{#1}}
\newrefformat{st}{Step~\ref{#1}}
\newrefformat{sc}{Section~\ref{#1}}
\newrefformat{df}{Definition~\ref{#1}}
\newrefformat{rm}{Remark~\ref{#1}}
\newrefformat{q}{Question~\ref{#1}}
\newrefformat{pb}{Problem~\ref{#1}}
\newrefformat{cd}{Condition~\ref{#1}}
\newrefformat{eg}{Example~\ref{#1}}
\newrefformat{he}{Heore~\ref{#1}}
\newrefformat{fg}{Figure~\ref{#1}}
\newrefformat{tb}{Table~\ref{#1}}
\newrefformat{as}{Assumption~\ref{#1}}

\newcommand{\pref}{\prettyref}

\newcommand{\Cl}{\operatorname{Cl}}
\newcommand{\Coh}{\operatorname{Coh}}
\newcommand{\colim}{\operatorname{colim}}
\newcommand{\Coker}{\operatorname{Coker}}
\newcommand{\Cone}{\operatorname{Cone}}
\newcommand{\Conv}{\operatorname{Conv}}
\newcommand{\Core}{\operatorname{Core}}
\newcommand{\Crit}{\operatorname{Crit}}
\newcommand{\del}{\partial}

\newcommand{\DG}{\operatorname{DG}}

\newcommand{\ddiv}{\operatorname{div}}
\newcommand{\Edge}{\operatorname{Edge}}
\newcommand{\End}{\operatorname{End}}

\newcommand{\Fuk}{\operatorname{Fuk}}
\newcommand{\GL}{\operatorname{GL}}
\newcommand{\Ham}{\operatorname{Ham}}

\newcommand{\Hom}{\operatorname{Hom}}
\newcommand{\id}{\operatorname{id}}
\newcommand{\ii}{\sqrt{-1}}

\newcommand{\IndCoh}{\operatorname{IndCoh}}

\newcommand{\Ker}{\operatorname{Ker}}
\newcommand{\Lie}{\operatorname{Lie}}
\newcommand{\Log}{\operatorname{Log}}
\newcommand{\MF}{\operatorname{MF}}

\newcommand{\Mod}{\operatorname{Mod}}

\newcommand{\mSh}{\operatorname{\mu Sh}}

\newcommand{\Open}{\operatorname{Open}}
\newcommand{\rank}{\operatorname{rank}}
\newcommand{\ret}{\operatorname{ret}}
\newcommand{\Perf}{\operatorname{Perf}}

\newcommand{\Proj}{\operatorname{Proj}}
\newcommand{\pr}{\operatorname{pr}}
\newcommand{\Qcoh}{\operatorname{Qcoh}}

\newcommand{\Sh}{\operatorname{Sh}}

\newcommand{\Spec}{\operatorname{Spec}}
\newcommand{\ssupp}{\operatorname{ss}}

\newcommand{\VVert}{\operatorname{Vert}}

\newcommand{\cF}{\mathcal{F}}

\newcommand{\cH}{\mathcal{H}}
\newcommand{\cI}{\mathcal{I}}

\newcommand{\cL}{\mathcal{L}}

\newcommand{\cS}{\mathcal{S}}
\newcommand{\cT}{\mathcal{T}}

\newcommand{\bA}{\mathbb{A}}
\newcommand{\bC}{\mathbb{C}}

\newcommand{\bL}{\mathbb{L}}
\newcommand{\bN}{\mathbb{N}}
\newcommand{\bP}{\mathbb{P}}

\newcommand{\bR}{\mathbb{R}}
\newcommand{\bT}{\mathbb{T}}
\newcommand{\bZ}{\mathbb{Z}}

\newcommand{\bff}{\mathbf{f}}
\newcommand{\bfg}{\mathbf{g}}

\newcommand{\bfq}{\mathbf{q}}
\newcommand{\bfv}{\mathbf{v}}
\newcommand{\bfA}{\mathbf{A}}
\newcommand{\bfB}{\mathbf{B}}

\newcommand{\bfH}{\mathbf{H}}

\newcommand{\bfS}{\mathbf{S}}
\newcommand{\bfT}{\mathbf{T}}

\newcommand{\bfX}{\mathbf{X}}
\newcommand{\bfY}{\mathbf{Y}}

\newcommand{\scrA}{\mathscr{A}}
\newcommand{\scrB}{\mathscr{B}}
\newcommand{\scrC}{\mathscr{C}}
\newcommand{\scrD}{\mathscr{D}}

\newcommand{\scrF}{\mathscr{F}}

\newcommand{\scrO}{\mathscr{O}}

\newcommand{\frakj}{\mathfrak{j}}

\newcommand{\frakt}{\mathfrak{t}}

\newcommand{\frakD}{\mathfrak{D}}

\newcommand{\frakL}{\mathfrak{L}}

\newcommand{\frakS}{\mathfrak{S}}

\newcommand{\frakU}{\mathfrak{U}}

\usepackage{tikz-cd}

\begin{document}

\title{Homological mirror symmetry for complete intersections in algebraic tori}
\author[H.~Morimura]{Hayato Morimura}
\address{Kavli IPMU, University of Tokyo, 5-1-5 Kashiwanoha, Kashiwa, Chiba, 277-8583, Japan}
\email{hayato.morimura@ipmu.jp}

\author[N.~Sibilla]{Nicolò Sibilla}
\address{SISSA, via Bonomea 265, 34136 Trieste, Italy}
\email{nsibilla@sissa.it}

\author[P.~Zhou]{Peng Zhou}
\address{University of California, Berkeley, CA 94720-3840, Berkeley, USA}
\email{pzhou.math@berkeley.edu}

\date{}
\pagestyle{plain}

\begin{abstract}
We prove one direction of homological mirror symmetry for complete intersections in algebraic tori,
in all dimensions.
The mirror geometry is not a space but a LG model,
i.e.
a pair given by a space and a regular function.
We show that
the Fukaya category of the complete intersection is equivalent to the category of matrix factorizations of the LG pair.
Our approach yields new results also in the hypersurface setting,
which was treated earlier by
Gammage
and
Shende.
Our argument depends on breaking down the complete intersection into smaller more manageable pieces,
i.e.
finite covers of products of higher dimensional pairs-of-pants,
thus implementing a program first suggested by Seidel.
\end{abstract}

\maketitle

\section{Introduction}
Mirror symmetry is a mysterious duality discovered by string theorists in the '80-s.
It asserts that  
string theory backgrounds should come in pairs (called \emph{mirror partners})
that,
despite having different geometric properties,  
give rise to the same physics.
To the untrained eye mirror partners might look nothing alike,
but string theory predicts the existence of an intricate dictionary allowing to transfer geometric information across between them.
Roughly,
complex geometric information on a space is encoded as symplectic data on its mirror partner,
and
vice versa.
Since the early '90-s,
mathematicians have made various attempts to distill the geometric meaning of mirror symmetry.
Homological Mirror Symmetry (HMS) is one of the most influential mathematical formulations of mirror symmetry.
It posits that mirror symmetry is,
at bottom,
an equivalence of categories.
HMS was first proposed by Kontsevich in 1994 and
it is still,
thirty years on,
the focus on intense research.
It is fundamental,
in the sense that
it is expected to encompass most other mathematical formulations of mirror symmetry. 

According to HMS
if
$X$
and
$X^\prime$
are mirror partners the
\emph{derived category}
of coherent sheaves of
$X$
should be equivalent to the
\emph{Fukaya category}
of
$X^\prime$,
and vice versa.
The derived category is a repository of algebraic information.
The objects living inside it include,
for instance,
vector bundles
and
the structure sheaves of subvarieties of
$X$.
The Fukaya category is a highly sophisticated symplectic invariant,
which captures the quantum intersection theory of Lagrangian submanifolds of
$X^\prime$.
The original formulation of HMS requires the mirror partners
$X$
and
$X^\prime$
to be compact Calabi--Yau (CY) varieties,
and
under these assumptions it has been established in many important cases,
starting with
\cite{PZ} \cite{Sei1} \cite{She}.
However HMS has also been generalized to wider non-proper and non-CY settings.
This requires readjusting the nature of the objects involved in the equivalence.
In this article we contribute to this line of research by studying HMS for a particularly interesting class of non-compact symplectic manifolds. 

We prove one direction of HMS for complete intersections in algebraic tori,
in all dimensions.
We adopt the formulation of HMS for complete intersections proposed in
\cite{AAK},
see also
\cite{GKR}.
We remark that
the other direction of HMS for complete intersections was proved in
\cite{AA}.
The mirror geometry is not a space but a LG model,
i.e.
a pair given by
a space
and
a regular function.
Our proof follows by implementing an algorithm 
that allows us to break down the complexity of complicated symplectic manifolds into small computable pieces,
and
that ultimately goes back to ideas of Seidel
\cite{Sei}.
The key input is given by recent advances in the study of the Fukaya category,
which reveal its hidden local nature,
at least for non-compact symplectic manifolds. 
We will briefly review this story in the next section. 

\subsection{Fukaya categories and locality} 
The Fukaya category of a symplectic manifold 
$M$
was introduced by Fukaya
\cite{Fuk}.
It is a highly non-trivial symplectic invariant. 
Providing adequate foundations for the theory in the general setting is delicate,
and this has been accomplished only relatively recently by Fukaya and his collaborators
\cite{FOOO1, FOOO2}.
In fact,
the Fukaya category is not quite a category:
it is an $A_\infty$-category.
In an $A_\infty$-category the composition of morphisms is associative only up to homotopy, 
and
homotopies are themselves just the first layer in an infinite tower of higher composition laws.
Roughly,
the objects of the Fukaya category of
$M$
are Lagrangian submanifolds,
while the Hom space between two Lagrangians is the linear span of their intersection points.
The actual picture is much more complicated:
for instance,
as Lagrangians intersect in finitely many points only under transversality assumptions,
all these data are well defined only up to appropriate choices of perturbations.

The higher $A_\infty$-operations in the Fukaya category are controlled by counts of pseudoholomorphic disks with Lagrangian boundary conditions. 
This is the source of  some of the biggest challenges in the theory.  In particular,
pseudoholomorphic disks are non-local in nature,
so higher operations in the Fukaya category depend on the global geometry of the manifold.
This is in sharp contrast with the derived category of coherent sheaves,
that satisfies descent with respect to the analytic topology and most other Grothendieck topologies commonly used in algebraic geometry. 
Around 2010 however,
two new paradigms emerged suggesting that
under favourable assumptions the Fukaya category should also exhibit a good local-to-global behaviour.
The computational payoffs would be tremendous,
as complicated global computations would be reduced to more manageable local ones.
The first of these approaches breaks down a
\emph{Liouville manifold}
into pieces called
\emph{Liouville sectors};
while the second,
which originated with Seidel
\cite{Sei},
relies on the availability of higher dimensional
\emph{pants decompositions}. 
These two point of views are subtly different,
and rely on somewhat distinct sets of assumptions.
As we will show,
in the setting of symplectic submanifolds of 
$(\mathbb{C}^*)^N$
they turn out to give compatible pictures of the locality of the Fukaya category.
In fact,
this is one of the key inputs in our argument.
Before explaining our results in greater detail,
let us briefly explain these two stories. 

The Fukaya category was long expected to be a kind of quantization of the symplectic manifold. 
A precise proposal was made in the influential paper of Kapustin--Witten 
\cite{KW},
where the authors model the Fukaya category of a holomorphic cotangent bundle in terms of $\frakD$-modules over the base.  Motivated by this and by earlier work of Fukaya,
Nadler--Zaslow show that
the (infinitesimally wrapped) Fukaya category of a cotangent bundle 
$T^*X$
is equivalent to the category of contructible sheaves over
$X$
(which is assumed to be an analytic manifold)
\cite{NZ}.
Via microlocalization,
the category of constructible sheaves sheafifies over
$T^*X$.
This implies,
in particular,
that
the Fukaya category of cotangent bundles displays suprisingly good local-to-global properties.
An extension of this picture to
\emph{Weinstein manifolds}
was later proposed by Kontsevich
\cite{Kon}.
Cotangent bundles are exact:
the standard symplectic form admits a primitive,
called a Liouville $1$-form.
Weinstein manifolds are a class of exact symplectic manifolds satisfying some extra regularity assumptions on the Liouville1-form. 
Weinstein manifolds retract to an exact Lagrangian core,
called the
\emph{skeleton}
which is a kind of generalized zero section with singularities.
Kontsevich conjectured that the wrapped Fukaya category localizes on the skeleton. That is, it defines a (co)sheaf of categories whose global sections recover the wrapped Fukaya category, and whose local sections are in many cases easily computable.
This line of research has been intensely pursued in the last ten years, 
and we now have a robust theory of the local behaviour of the Fukaya category in this setting.
The state-of-the-art is provided by a series of works by Ganatra--Pardon--Shende
\cite{GPS1, GPS2, GPS3},
one of whose main results is a complete descent package for the Fukaya category relative to a class of covers of Weinstein manifolds called Weinstein sectors.
 
This point of view has had numerous applications to HMS,
starting from
\cite{Kon} \cite{STZ} \cite{DK}.
In
\cite{GS1},
Gammage--Shende use this framework to prove HMS for hypersurfaces in 
$(\mathbb{C}^*)^N$.
In this paper we study HMS for all complete intersections in
$(\mathbb{C}^*)^N$,
but our approach differs from Gammage--Shende already in the hypersurface case,
and yields more general results.
Our methods combine
sectorial descent
and
a different locality with respect to pants decompositions,
that was first suggested by Seidel,
and that we explain next.
 
Pants decompositions have long played a central role in understanding the topology of complex curves.
Higher dimensional pants decompositions were studied by Mikhalkin in 
\cite{Mik}.
A higher dimensional pair-of-pants is the complement of
$N+2$
generic hyperplanes in
$\mathbb{P}^N$.
Mikhalkin proves that
hypersurfaces in
$(\mathbb{C}^*)^N$
admit a higher dimensional pants decomposition.
Mikhalkin's result is formulated in purely topological terms but,
as he points out,
it can be upgraded so as to be compatible with the natural symplectic structures.
Mikhalkin's work has had many applications in HMS,
and it plays for instance an important role in
\cite{She}.
Higher dimensional pants decompositions exist in more general settings,
and sometimes also for compact varieties:
for instance,
hypersurfaces in abelian varieties admit pants decompositions.
Seidel suggested that
when such a decomposition exists,
pairs-of-pants should provide the building blocks of the global Floer theory of the manifold.
In particular,
the Fukaya category should be expressible as a limit of the Fukaya category of the pairs-of-pants making up the decomposition.
It is important to stress that
this provides a very different kind of local-to-global principle for the Fukaya category.
Pants are very different from the Liouville/Weinstein sectors underpinning the locality on the skeleton which we have described above. 

Remarkably the locality of the Fukaya category with respect to pants decomposition is expected to match neatly,
under mirror symmetry,
Zariski descent on the mirror category.
This opens the way to implement divide-and-conquer algorithms in HMS,
reducing a difficult global mirror symmetry statement to a much more computable local one.
Up to now,
there have been only a few attempts to implement rigorously Seidel's proposal.
One instance was the beautiful paper of Lee
\cite{Lee}
that proves Hori-Vafa HMS for curves in
$(\mathbb{C}^*)^2$.
The same result was proved independently,
and with very different methods by
Pascaleff
and
the second author in
\cite{PS1},
with follow-ups in the compact setting in
\cite{PS2} \cite{PS3}.

Since
\cite{PS1}
serves as the blueprint for some of the key arguments in this paper,
it is worthwhile to review its main ideas here.
The critical point is exactly the interplay between the two regimes of locality.
In
\cite{PS1}
the authors work within the framework of microlocal sheaves on skeleta. 
Thus,
sectorial descent is built in their underlying theory.
They show that
sectorial descent,
supplemented with a local calculation,
implies the seemingly very different Seidel type localization on pants. Proving this involves setting up a recursion which builds up the Riemann surface in a step-by-step fashion,
by gluing together the pants making up the pants decomposition.
Crucially,
at each step the Weinstein structure of the surface is modified so as to be adapted to the gluing.
Geometrically,
this means deforming the skeleton in such a way that some of its components are pushed towards the portions on the boundary along which the gluing is taking place.

The vertices of the diagram implementing Seidel's locality are the Fukaya categories of the pairs-of-pants 
and
of their intersections.
The latter,
in the surface case,
are isomorphic to symplectic annuli.
The arrows are Viterbo restrictions.
The claim is that the limit of this diagram of categories is equivalent to the wrapped Fukaya category.
Geometrically,
the Seidel type localization is a mechanism that allows to glue skeleta along a common closed subskeleton,
on the condition that the latter lies in a separating contact hypersurface.
So it can be rephrased as a kind of descent for the wrapped Fukaya category with respect to
\emph{closed covers}
of the skeleton,
subject to appropriate assumptions.
Sectorial descent,
on the other hand,
captures a more straightforward descent of the wrapped Fukaya category with respect to
\emph{open covers}
of the skeleton. 
 
In this paper,
we adapt this strategy to the higher dimensional case.
We recover Gammage--Shende's result on HMS for hypersurfaces in
$(\mathbb{C}^*)^N$
in a more general form,
as we remove the assumptions that the triangulation of the associated Newton polytope is star-shaped.
Our methods extend to give a proof of HMS for complete intersections in 
$(\mathbb{C}^*)^N$.
We remark that
extending the approach of Gammage--Shende to complete intersections is unfeasible with current technology.
Their argument requires controlling the global skeleton of the Weinstein manifolds,
which is not known with current methods in the case of complete intersections.
Our approach bypasses this delicate issue,
as it depends on constructing only smaller local pieces of the skeleton near the place where the gluing is taking place.
As such it provides an algorithm for proving HMS that has potential to be applicable in more general settings beyond the one we consider in this article.

We explain our main results and the structure of the paper next. 
 
\subsection{Main results. Hypersurfaces} 
Let
$\bT
=
M^\vee_{\bR / \bZ}
=
M^\vee_\bR / M^\vee$
be a real $(d+1)$-dimensional torus with cocharacter lattice
$M^\vee$.
Let
$M$
be the character lattice of
$T$. 
We denote by
$\bT_\bC$ 
the associated complex torus. Let 
$H \subset \bT^\vee_\bC$  
be a hypersurface cut out by the Laurent polynomial 
$$
W \colon \bT^\vee_\bC \to \bC, \
x 
\mapsto
\sum_{\alpha \in A} c_\alpha x^\alpha
$$
where
$A \subset M^\vee$
is a finite set of monomials.
The hypersurface
$H$
is a closed subvariety of  
$\bT^\vee_\bC$
and
is thereofore naturally Weinstein.
An
\emph{adapted triangulation}
$\cT$
of the convex hull
$\Conv(A)$
of
$A$
is,
by definition,
a triangulation arising as the corner locus of a convex piecewise linear function.
These data determine a tropical  hypersurface
$\Pi$
in
$M_\mathbb{R}$,
the
\emph{tropicalization}
of
$H$,
and
equip
$H$
with a higher-dimensional pair-of-pants decomposition.
As an abstract topological space,
$\Pi$
is a homeomorphic to the dual intersection complex of the  pair-of-pants decomposition of
$H$. 
 
The mirror of
$H$
is a $(d+2)$-dimensional toric LG model.
Let
$Y$
be the noncompact toric variety associated with the fan
\begin{align*}
\Sigma_Y
=
\Cone(-\cT \times \{ 1 \})
\subset
M^\vee_\bR \times \bR
\end{align*} 
Via the usual toric dictionary,
the map of fans induced by the projection 
$M^\vee_\bR \times \bR \to \bR$
determines a regular funtion 
$W_Y\colon Y \to \bC$.
On this side of the mirror correspondence,
the tropicalization
$\Pi$
of
$H$
is naturally identified with the  image of the singular locus of
$W_Y^{-1}(0)$
under the moment map.

\begin{customthm}{A} 
\label{intromain1}
There is an equivalence 
$$
\Fuk(H) \simeq \MF(Y, W_Y)
$$
\end{customthm}

Let us outline the argument.
Both side of the equivalence are local in nature.
The key is that,
as we explained,
the manifold
$H$
carries a pants decomposition,
and
therefore its wrapped Fukaya category can be built out of the local Fukaya categories of the individual pants
$P$.
Although our proof of this fact is remarkably simple,
it relies in a crucial way on the machinery developed by Ganatra--Pardon--Shende. 
The locality of
$\MF(Y, W_Y)$
is straightforward.
Indeed,
the category of matrix factorizations satisfies Zariski descent.
As
$Y$
is smooth,
it has a canonical toric open cover by affine spaces. Further the restriction of
$W_Y$
on each patch coincides,
up to coordinates change, with the standard superpotential 
$$
y_1 \cdots y_{d+2} \colon \bA^{d+2} \to \bA^1
$$  
Our argument involves two steps. The first consists in establishing  the local equivalence 
\begin{equation}
\label{local}
\Fuk(P)
\simeq
\MF(\bA^{d+2}, y_1 \cdots y_{d+2})
\end{equation}
This was proved by Nadler in
\cite{Nad}
using an alternative model for the B-side category.
Recall that, by a theorem of Orlov,
there is an equivalence 
$$
\MF(\mathbb{A}^{d+2}, y_1 \ldots y_{d+2})
\simeq
\Coh(\{ y_1 \cdots y_{d+1}=0 \})_{\mathbb{Z}_2}
$$
where the latter is the $\mathbb{Z}_2$-folding of the ordinary category of coherent sheaves.
Nadler describes a family of Weinstein structures on
$P$ depending on the choice of a
\emph{leg}
of
$P$. 
The corresponding skeleton has non-trivial intersections with all the legs of
$P$,
except with the chosen one.
Nadler computes the Fukaya category in terms of microlocal sheaves on the  skeleton,
and
proves in this way that
$\Fuk(P)$
is equivalent to
$\Coh(\{ y_1 \cdots y_{d+1}=0 \})_{\mathbb{Z}_2}$.  The choice of skeleton collapses the natural $\frakS_{d+2}$-symmetry of the pair-of-pants to the smaller $\frakS_{d+1}$-symmetry of the skeleton.
This matches the $\frakS_{d+1}$-action on
$\Coh(\{ y_1 \cdots y_{d+1}=0 \})_{\mathbb{Z}_2}$
by permutation of coordinates,
which is also the residue of the richer $\frakS_{d+2}$-symmetry of
$$\MF(\mathbb{A}^{d+2}, y_1 \cdots y_{d+2})$$
For our purposes,
we need to restore the complete $\frakS_{d+2}$-symmetry of the problem  
which remains hidden in Nadler's formulation.
The locality of the two mirror categories is neatly encoded in the combinatorics of the tropicalization
$\Pi$.
Both
$\Fuk(H)$
and
$\MF(Y, W_Y)$
define in a natural way two constructible sheaves of categories over
$\Pi$,
where
$\Pi$
is equipped with its natural stratification.
The final globalization step consists in noticing that as the local sections and local restrictions of these two sheaves match,
their global sections must also be equivalent.
This is the content of Theorem
\ref{intromain1}. 

\subsection{Main results. Complete intersections}
Let us describe next the complete intersection setting.
The underlying toric geometry is a simple extension of the ideas entering in the hypersurface setting,
so we will give a somewhat abbreviated treatment of this story
and
refer the reader to the main text for full details.
We keep the notations from the previous section. 

Consider hypersurfaces  
$H_1, \ldots, H_r \subset \bT_\bC$
in sufficiently general position,
cut out by Laurent polynomials
$$
W_1, \ldots, W_r \colon \bT^\vee_\bC \to \bC, \
x
\mapsto
\sum_{\alpha \in A_i} c_\alpha x^\alpha
$$
By the genericity assumption,
they meet transversely in a subvariety of dimension
$n-r$
$$
\bfH = H_1 \cap \cdots \cap H_r \subset \bT^\vee_\bC
$$
The subvariety
$\bfH$
carries a natural Weinstein structure.
We form a total superpotential
$W_{\bfH}$
by adding an extra factor
$\bC^r$
with coordinates
$u_1, \ldots, u_r$
$$
W_{\bfH}=u_1 W_1 + \cdots + u_r W_r
\colon
\bT^\vee_\bC \times \bC^r \to \bC
$$
Note that
$\bfH$
can be recovered as the critical locus of
$W_{\bfH}$. 

The Newton polytope of
$W_{\bfH}$
is the convex hull of
$$
\bfA = \bigcup_{i} - A_i \times \{e_i\} \subset M^\vee_\mathbb{R} \times \mathbb{R}^r
$$
where
$e_1, \ldots, e_r$
is the standard basis of
$\mathbb{R}^r$. 
A choice of adapted triangulations of the convex polytopes
$\Conv(A_i) \subset M^\vee_\mathbb{R}$ 
determines a triangulation
$\bfT$ of $\Conv(\bfA)$.
Following
\cite{AAK},
the mirror of
$\bfH$
is the higher dimensional LG model determined by
$\bfT$.
More precisely,
let
$$\Sigma_\bfY \subset M^\vee_\bR \times \bR^r$$
be the fan corresponding to
$\bfT$, 
and
let 
$\bfY$
be 
the noncompact $(d+r+1)$-dimensional toric variety associated with
$\Sigma_\bfY$.
The fan
$\Sigma_\bfY$
admits
$r$
projections to
$\mathbb{R}$,
and
the sum of the corresponding monomials induces a  regular function
$W_{\bfY}$
on
$\bfY$.
The mirror of
$\bfH$
is the LG model
$(\bfY, W_{\bfY})$.  

\begin{customthm}{B} 
\label{intromain2}
There is an equivalence 
$$
\Fuk(\bfH)
\simeq
\MF(\bfY, W_{\bfY})
$$
\end{customthm}

The proof strategy follows the pattern of  the hypersurface case,
but there are some  features which are specific to the complete intersection setting which are worth highlighting.
Strictly speaking,
complete intersections do not admit a higher dimensional pair-of-pants decomposition.
Rather,
generic intersections of pants are locally isomorphic to (finite covers of) products of lower dimensional pairs-of-pants.
The appearance of finite covers cannot be avoided,
however it is easily controlled.
For clarity,
in this introduction,
we shall ignore this issue.
Via the K\"unneth formula for the wrapped Fukaya category,
the Fukaya category of
$\bfH$
is thus locally equivalent to the tensor product of the Fukaya categories of the factors,
i.e.
lower dimensional pairs-of-pants.
This is matched,
on the B-side,
by the Zariski local behaviour of
$\MF(\bfY, W_{\bfY})$.
 
On each affine toric open subset of 
$\bfY$,
the superpotential 
$W_{\bfY}$
can be written as a sum of monomials.
Preygel's Thom-Sebastiani theorem implies that,
locally,
the category of matrix factorizations factors  as a tensor product of matrix factorizations of lower dimensional superpotentials.
Thus,
in the complete intersection case,
the local HMS equivalence is just a tensor product of the fundamental local equivalences 
$$
\Fuk(P)
\simeq
\MF(\mathbb{A}^{d+2}, y_1 \cdots y_{d+2})
$$
underpinning the hypersurface case.
The globalization step follows along exactly parallel lines as in the hypersurface case.  

As we have already remarked,
our methods allow us to give a description of the Fukaya category of
$\bfH$
bypassing the difficult task of describing a global skeleton. 
The explicit calculation of the skeleton is,
in contrast,
a key input in the approach of Gammage--Shende in the hypersurface case.
We obviate the absence of a computable model of the  global skeleton by setting up a recursion that builds the complicated global symplectic geometry of
$\bfH$
out of simple pieces amenable to computation:
products of lower dimensional pants,
and
their finite covers.
This is a mild generalization of the set-up originally envisioned by Seidel in terms of pants decompositions. 
This allows us to get away with building skeleta,
or
rather Weinstein structures,
for these local pieces only.
This extra flexibility crucially relies on the invariance of the wrapped Fukaya category under Liouville homotopy,
which allows us to engineer Weinstein structures with good properties near a boundary where local pieces are glued together.
In the last section of the paper,
we revisit Gammage–Shende’s setting,
and show that
in fact we can recover their main result with our methods.
In the setting they place themselves in,
the local skeleta built on the individual pairs-of-pants glue to yield global skeleta.
This,
together with a well-known dictionary relating
matrix factorizations
and
categories of coherent sheaves,
readily implies their result.  Additionally, the paper contains an Appendix where the locality of the Fukaya category with respect to pants decompositions (and their generalizations, in the complete intersection setting) is rigorously established by means of the semi-tropicalization technique due to the third author
\cite{Zho}.
We believe that
the results in the Appendix might be of independent interest,
as they introduce an approach to the locality of the Fukaya category 
that  should  be applicable to more general settings than the one considered in the present paper.

Our
methods
and
results
open the way to several potential directions for future investigations.
We limit ourselves to mention one,
which we intend to pursue in future work.
In this paper we espouse the viewpoint on  HMS for complete intersections proposed in
\cite{AAK}.
There is however another important model for mirror symmetry for complete intersections in toric ambient varieties,
which was developed by Batyrev--Borisov in
\cite{BB}.
That framework encompasses both the non-compact regime where the ambient manifold is a torus
(which is the setting we work in this article),
and
its toric compactifications.
In future work we will explore in which ways our methods can be used to obtain Batyrev--Borisov type HMS for complete intersections.

$ $

{\bf Acknowledgements: }Some of the key inputs in this article are generalizations of ideas that were developed in joint work by the second author with James Pascaleff.
Over the years,
the second author has benefited  from countless discussions with James that have shaped his understanding of
the Fukaya category
and
mirror symmetry. 
We want thank James for his generosity in sharing his ideas,
and
for his encouragement and interest in this project.     
H.M. is partially supported by
JSPS KAKENHI Grant Number
JP23KJ0341.

\section{Review on HMS for pairs of pants}
In this section,
we review HMS for pairs of pants established by Nadler in
\cite{Nad}.
This gives local equivalences
which we glue to yield HMS for hypersurfaces in an algebraic torus.
Also,
thorough understanding of such local equivalences plays an important role
when working in the complete intersection setting.

\subsection{Tailored pants}
Let
$T^{d+1} = (\bR / 2 \pi \bZ)^{d+1}$
be the real torus with coordinates
$\theta = (\theta_1, \ldots, \theta_{d+1})$.
Fix the usual identification
$T^* T^{d+1} \cong T^{d+1} \times \bR^{d+1}$
with canonical coordinates
$(\theta, \xi)$
for
$\xi = (\xi_1, \ldots, \xi_{d+1})$.
The symplectic manifold
$T^* T^{d+1}$
carries the standard Liouville structure
\begin{align*}
\alpha_{d+1} = \sum^{d+1}_{i = 1} \xi_i d \theta_i, \
\omega_{d+1} = \sum^{d+1}_{i = 1} d \xi_i \wedge d \theta_i
\end{align*}
whose skeleton is the zero section
$T^{d+1} \subset T^* T^{d+1}$. 
The self-action of
$T^{d+1}$
lifts to a Hamiltonian action on
$T^* T^{d+1}$
with the moment map
\begin{align*}
\mu_{d+1}
\colon T^* T^{d+1}
\to
\Lie(T^{d+1})^\vee
\cong
\bR^{d+1}, \
(\theta, \xi) \mapsto \xi.
\end{align*}
Taking its squared length,
one obtains a Weinstein manifold
$(T^* T^{d+1}, \alpha_{d+1}, |\mu_{d+1}|^2)$.
Note that
the function
$|\mu_{d+1}|^2$
is Morse--Bott.

Fix the identification
$T^* T^{d+1} \cong T^{d+1}_\bC = (\bC^*)^{d+1}$
with coordinates
$x = (x_1, \ldots, x_{d+1})$
via
$x_i = e^{\xi_i + \ii \theta_i}$.
Then
$\mu_{d+1}$
transports to the log projection
\begin{align*}
\Log_{d+1} \colon T^{d+1}_\bC \to \bR^{d+1}, \
x \mapsto (\log |x_1|, \ldots, \log |x_{d+1}|).
\end{align*}

\begin{definition}
The $d$-dimensional
\emph{standard pair of pants}
is a complex hypersurface
\begin{align*}
P_d
=
\{ 1 + x_1 + \cdots + x_{d+1} = 0 \}
\subset
T^{d+1}_\bC.
\end{align*}
\end{definition}

We regard
$P_d$
as an exact symplectic manifold
equipped with the restricted standard Liouville structure.
Via the open embedding
$T^{d+1}_\bC
\hookrightarrow
\bP^{d+1}_\bC
=
\Proj \bC[x_0, x_1, \ldots, x_{d+1}]$
the pants
$P_d$
maps to the complement of
\begin{align*}
\bigcup^{d+1}_{i = 0}
(\{ x_0 + x_1 + \cdots + x_{d+1} = 0 \}
\cap
\{ x_i = 0 \})
\end{align*}
in the hyperplane
$\{ x_0 + x_1 + \cdots + x_{d+1} = 0 \}
\subset
\bP^{d+1}_\bC$.
Hence the symmetric group
$\frakS_{d+2}$
naturally acts on
$P_d$
by permutation of the homogeneous coordinates.

The next statement gives the key defining properties of a Liouville submanifold
$(\tilde{P}_d, \alpha_{\tilde{P}_d} = \alpha_{d+1}|_{\tilde{P}_d})$
which is in a precise sense a better-behaved replacement of the pair of pants.

\begin{lemma}[{\cite[Proposition 4.6]{Mik}, \cite[Proposition 4.2, 4.9]{Abo}, \cite[Section 5.2]{Nad}}] \label{lem:tailored}
There is a $\frakS_{d+2}$-equivariant Hamiltonian isotopy of symplectic submanifolds of
$T^{d+1}_\bC$
from
$P_{d}$
to
$(\tilde{P}_d, \alpha_{\tilde{P}_d} = \alpha_{d+1}|_{\tilde{P}_d})$
with the following properties:
\begin{itemize}
\item
The isotopy is constant inside
$\Log^{-1}_{d+1}(\Delta_d(R))$
for some constant
$0 \ll R$
with
\begin{align*}
\Delta_d(R)
=
\{ \xi \in \bR^{d+1} \ | -R \leq \xi_1, \ldots, -R \leq \xi_{d+1}, \ \sum^{d+1}_{i=1} \xi_i \leq R  \}.
\end{align*}
\item
We have the inductive compatibility
\begin{align*}
L_{d, d+1}(K)
=
\tilde{P}_d \cap T^{d+1}_{\bC, d+1}(K)
\cong
\tilde{P}_{d-1} \times \bC^*_{d+1}(K)
\end{align*}
for some constant
$0 \ll K$,
where
\begin{align*}
\bC^*_{d+1}(K)
=
\{ x_{d+1} \in \bC^* \ | \log |x_{d+1}| < -K \} ,\
T^{d+1}_{\bC, d+1}(K)
=
\{ x \in T^{d+1}_\bC \ | \log |x_{d+1}| < -K \}.
\end{align*}
\end{itemize}
\end{lemma}

\begin{remark}
As mentioned in
\cite[Remark 5.3.3]{GS1},
presumably the above isotopy is an $\frakS_{d+2}$-equivariant isotopy of Liouville submanifolds.
\end{remark}

Note that
$\tilde{P}_d$
coincides with
$P_d$
inside
$\Log^{-1}_{d+1}(\underline{\Delta}_d(R))$.
The $\frakS_{d+2}$-action implies similar compatibilities in other directions.
Namely,
we have the inductive compatibility
\begin{align} \label{eq:recursive}
L_{d, i}(K)
=
\tilde{P}_d \cap T^{d+1}_{\bC, i}(K)
\cong
\tilde{P}_{d-1} \times \bC^*_{i}(K)
\end{align}
for
some constant
$0 \ll K$
and
fixed
$i = 1, \ldots, d$,
where
\begin{align*}
\bC^*_i(K)
=
\{ x_i \in \bC^* \ | \log |x_i| < -K \} ,\
T^{d+1}_{\bC, i}(K)
=
\{ x \in T^{d+1}_\bC \ | \log |x_i| < -K \}.
\end{align*}

\begin{definition}
We call the Liouville manifold
$(\tilde{P}_d, \alpha_{\tilde{P}_d})$
the $d$-dimensional
\emph{tailored pants}.
We call the open Liouville submanifold
$L_{d, i}(K)$
the
\emph{$i$-th leg}
of
$(\tilde{P}_d, \alpha_{\tilde{P}_d})$.
The
\emph{$(d+2)$-th leg}
of
$(\tilde{P}_d, \alpha_{\tilde{P}_d})$
is the remaining open Liouville submanifold
with respect to the $\frakS_{d+2}$-action,
which corresponds to the positive diagonal direction in
$\bR^{d+1}$.
\end{definition}

In order to provide
$\tilde{P}_d$
with particularly simple skeleton,
Nadler
broke symmetry by introducing the translated Liouville structure
\begin{align} \label{eq:tLiouville}
\alpha^l_{d+1} = \sum^{d+1}_{i=1} (\xi_i + l) d \theta_i,
\omega^l_{d+1} = \sum^{d+1}_{i = 1} d (\xi_i + l) \wedge d \theta_i = \omega_{d+1}
\end{align}
on
$T^{d+1}_\bC$
for some constant
$0 \ll l$.

\begin{definition}
We call the Weinstein structure given by a triple
\begin{align*}
(\tilde{P}_d, \beta_{\tilde{P}_d} = \alpha^l_{d+1} |_{\tilde{P}_d}, \Sigma^{d+1}_{i = 1} (\log |x_i| + l)^2)
\end{align*}
\emph{Nadler's Weinstein structure}.
We write
$\Core(\tilde{P}_d)$
for its skeleton.
We call the $(d+2)$-th leg of
$\tilde{P}_d$
equipped with Nadler's Weinstein structure the
\emph{final leg}.
\end{definition}

\begin{remark}
All the legs but the final remain symmetric under the $\frakS_{d+2}$-action. 
\end{remark}

For a proper subset
$I \subsetneq \{ 1, \ldots, d+1 \}$
let
\begin{align*}
\Delta_I(l)
=
\{ x \in \tilde{P}_d \cap T^{d+1}_{\bR < 0} \ | \ \log |x_i| = -l \ \text{for} \ i \in I, \ \log |x_j| > -l \ \text{for} \ j \in I^c \}
\end{align*}
be the relatively open subsimplex of the closed simplex
\begin{align*}
\Delta_d(l)
=
\{ x \in \tilde{P}_d \cap T^{d+1}_{\bR < 0} \ | \ \log |x_i| \geq -l \}.
\end{align*}
We denote by
$\delta_I(l)$
the barycenter
\begin{align*}
\{ x \in \tilde{P}_d \cap T^{d+1}_{\bR < 0} \ | \ \log |x_i| = -l \ \text{for} \ i \in I, \ \log |x_j| = \log |x_{j^\prime}| > -l \ \text{for} \ j, j^\prime \in I^c \}
\end{align*}
of the subsimplex
$\Delta_I(l)$.

\begin{lemma}[{\cite[Theorem 5.13]{Nad}}] 
For a proper subset
$I \subsetneq \{ 1, \ldots, d+1 \}$
let
$T^I \subset T^{d+1}$
be the subtorus defined by
$\theta_i = 0, i \in I^c$.
Then we have
\begin{align*}
\Core(\tilde{P}_d)
=
\bigcup_{I \subsetneq \{ 1, \ldots, d+1 \}}
T^I \cdot \Delta_I(l).
\end{align*}
\end{lemma}
\begin{proof}
The original proof uses induction on
$d$.
The case
$d = 0$ 
is obvious.
When
$d = 1$,
a nonempty proper subset
$I \subsetneq \{ 1, \ldots, d+1 \}$
is either
$\{ 1 \}$
or
$\{ 2 \}$.
On
$\delta_i(l) = \Delta_i(l)$
the Liouville form
$\beta_{\tilde{P}_1}$
vanishes
and
their stable manifolds are isomorphic to
$T^{i} \cdot \delta_{ \{ i \} }(l) = T^{i} \cdot \Delta_{ \{ i \} }(l)$.
From
\cite[Corollary 4.4, 4.5, Proposition 4.6]{Mik}
it follows that
the critical locus
$\Crit(\Sigma^{d+1}_{i = 1} (\log |x_i| + l)^2)$
coincides with
$\tilde{P}_d \cap T^{d+1}_{\bR_{<0}}$.
In particular,
since
$\tilde{P}_d$
is Weinstein, 
$\tilde{P}_d \cap T^{d+1}_{\bR_{<0}}$
contains the zero locus
$Z(\beta_{\tilde{P_d}})$.
The negative real points
$\tilde{P}_d \cap T^{d+1}_{\bR_{<0}}$
is a Lagrangian
as
$d \theta_i$
vanishes there
and
the Liouville flow on
$\tilde{P}_d \cap T^{d+1}_{\bR_{<0}}$
attracts the points to a point
$\delta_\emptyset(l)$.
Hence,
aside
$T^{i} \cdot \Delta_{ \{ i \} }(l)$,
only the stable manifold
$T^\emptyset \cdot \Delta_\emptyset(l)$
of
$T^\emptyset \cdot \delta_\emptyset(l)$
contributes to
$\Core(\tilde{P}_1)$.
For general
$d$,
combine
the same argument
and
the inductive compatibility from
\pref{lem:tailored}.
\end{proof}

\subsection{Microlocal interpretation}
Next,
we review the geometry of
$\Core(\tilde{P}_d).$
The action of the diagonal circle
$T^1_\Delta \subset T^{d+1}$
by translation lifts to a Hamiltonian action with the moment map
\begin{align*}
\mu_\Delta \colon T^* T^{d+1} \to \bR, \
(\theta, \xi) \mapsto \sum^{d+1}_{i = 1} \xi_i.
\end{align*}
Distinguishing the final coordinate
$\theta_{d+1}$
on
$T^{d+1}$,
we identify the quotient 
$\bT^d = T^{d+1} / T^1_\Delta$
with
$T^d$
via
$[\theta]
\mapsto
(\theta_1 - \theta_{d+1}, \ldots, \theta_d - \theta_{d+1})$.
Denoting by
$\frakt^*_d$
the dual of
$\Lie (\bT^d)
=
\{ \xi \in \bR^{d+1} | \sum^{d+1}_{i = 1} \xi_i = 0 \}$,
we identify
$T^* \bT^d$
with
$\bT^d \times \frakt^*_d$.
The product conic Lagrangian
\begin{align*}
\Lambda_{d+1}
=
(\Lambda_1)^{d+1}
\subset
\mu^{-1}_\Delta(\bR_{\geq 0})
\subset
T^* T^{d+1}
\end{align*}
is transverse to
$\mu^{-1}_\Delta(\chi)$
for
$\chi > 0$,
where
\begin{align*}
\Lambda_1
=
\{ (\theta, 0) \ | \ \theta \in T^1 \}
\cup
\{ (0, \xi) | \ \xi \in \bR_{\geq 0} \}
\subset
T^1 \times \bR
\cong
T^* T^1.
\end{align*}

Consider the twisted Hamiltonian reduction correspondence
\begin{align*}
T^*T^{d+1}
\overset{q_\chi}{\hookleftarrow}
\mu^{-1}_\Delta(\chi)
=
\{ (\theta, \xi) \in T^* T^{d+1} | \ \sum^{d+1}_{i = 1} \xi_i = \chi \}
\overset{p_\chi}{\twoheadrightarrow}
T^* \bT^d
\end{align*}
where
$q_\chi$
is the canonical inclusion
and
$p_\chi$
is the translated projection
\begin{align*}
p_\chi((\theta, \xi))
=
([\theta], \xi_1 - \hat{\chi}, \ldots, \xi_{d+1} - \hat{\chi}), \
\hat{\chi} = \chi / (d+1).
\end{align*}
The map
$p_{\chi}$
becomes an isomorphism
\begin{align*}
\mu^{-1}_\Delta(\chi) \cap \Lambda_{d+1}
=
\mu^{-1}_{d+1}(\tilde{\Xi}_d(\chi))
\to
\frakL_d
\coloneqq
p_\chi (q^{-1}_\chi(\Lambda_{d+1}))
\end{align*}
when restricted to the inverse image under
$\mu_{d+1}$
of the closed subsimplex
\begin{align*}
\tilde{\Xi}_d(\chi)
=
\{ \xi \in \bR^{d+1}_{\geq 0} | \ \xi_1, \ldots, \xi_{d+1} \geq 0, \sum^{d+1}_{i=1} \xi_i = \chi \}.
\end{align*}

Let
$\tilde{\Xi}_I(\chi)
=
\tilde{\Xi}_d(\chi)
\cap
\sigma_I$
be the relatively open subsimplex with
\begin{align*}
\sigma_I
=
\{ \xi \in \bR^{d+1}_{\geq 0} | \ \xi_i = 0 \ \text{for} \ i \in I, \ \xi_j > 0 \ \text{for} \ j \in I^c \}
\end{align*}
for a proper subset
$I \subsetneq \{ 1, \ldots, d+1 \}$.
From
$\mu_{d+1}(\Lambda_{d+1}) = (\bR_{\geq 0})^{d+1}$
it follows
$\mu^{-1}_{d+1}(\xi) \cap \Lambda_{d+1}
\cong
T^I$
for
$\xi \in \sigma_I$
and
$p_\chi$
restricts to an isomorphism
$\bigcup_{I \subsetneq \{ 1, \ldots, d+1\}}
T^I \times \tilde{\Xi}_I(\chi)
\to
\frakL_d$.
Hence we obtain an identification of subspaces
\begin{align*}
\frakL_d
\cong
\bigcup_{I \subsetneq \{ 1, \ldots, d+1\}}
\bT^I \times \Xi_I(\chi)
\subset
\bT^d \times \frakt^*_d
\cong
T^* \bT^d,
\end{align*}
where
$\bT^I \subset \bT^d$
are isomorphic images of
$T^I$
under the quotient
and
\begin{align*}
\Xi_I(\chi)
=
\{ \xi \in \bR^{d+1}_{\geq -\hat{\chi}} | \ \xi_i = -\hat{\chi} \ \text{for} \ i \in I, \ \xi_j > -\hat{\chi} \ \text{for} \ j \in I^c, \sum^{d+1}_{i=1} \xi_i = 0 \}
\end{align*}
are isomorphic images of
$\tilde{\Xi}_I(\chi)$
under translation by
$-\hat{\chi}$.

The symplectic geometry of a certain open neighborhood of
$\Core(\tilde{P}_d)$
in
$\tilde{P}_d$
is equivalent to that of the associated open neighborhood of
$\frakL_d$
in
$T^* \bT^d$.

\begin{lemma}[{\cite[Theorem 5.13]{Nad}}] \label{lem:Symp-Symp1}
There is an open neighborhood
$U_d \subset \tilde{P}_d$
of
$\Core(\tilde{P}_d)$
with an open symplectic embedding
$\frakj \colon U_d \hookrightarrow T^* \bT^d$
which makes the diagram
\begin{align*}
\begin{gathered}
\xymatrix{
U_d \ar[r]^-{\sim}
& \frakU_d \coloneqq \frakj(U_d) \\
\Core(\tilde{P}_d) \ar@{_(->}[u] \ar[r]^-{\sim}
& \frakL_d \ar@{_(->}[u]
}
\end{gathered}
\end{align*}
commute,
where the vertical arrows are the canonical inclusions.
\end{lemma}
\begin{proof}
Let
$U^\circ_d$
be a sufficiently small open neighborhood of
$\Delta_d(l)$
in
$\tilde{P}_d$.
By the inductive compatibility from
\pref{lem:tailored},
near
$\Delta_I(l)$
for
$\emptyset \neq I \subsetneq \{ 1, \ldots, d+1 \}$
the functions
$\log |x_i|, i \in I$
define a local coisotropic foliation of
$U^\circ_d$.
If
$I^\prime \subsetneq \{ 1, \ldots, d+1 \}$
contains
$I$,
then the foliation defined by
$\log |x_i|, i \in I^\prime$
refines that defined by
$\log |x_i|, i \in I$.
Let
$\frakU^\circ_d$
be a sufficiently small open neighborhood of
$\Xi_d(\chi)$
in
$T^* \bT^d$.
Near
$\Xi_I(\chi)$
for
$\emptyset \neq I \subsetneq \{ 1, \ldots, d+1 \}$
the functions
$\xi_i, i \in I$
define a local coisotropic foliation of
$\frakU^\circ_d$.
If
$I^\prime \subsetneq \{ 1, \ldots, d+1 \}$
contains
$I$,
then the foliation defined by
$\xi_i, i \in I^\prime$
refines that defined by
$\xi_i, i  \in I$.

The neighborhoods
$U^\circ_d$
and
$\frakU^\circ_d$
are symplectomorphic to the cotangent bundles of their Lagrangians
$U^\circ_d \cap T^{d+1}_{\bR < 0}$
and
$\frakU^\circ_d \cap \frakt^*_d$. 
Hence one finds a symplectomorphism
$\frakj^\circ \colon U^\circ_d \to \frakU^\circ_d$
restricting to a diffeomorphism
$U^\circ_d \cap T^{d+1}_{\bR < 0} \to \frakU^\circ_d \cap \frakt^*_d$,
which is compatible with the above local coisotropic foliations. 
Choose a sufficiently small open neighborhood
$U^\circ_I$
of
$\Delta_I(l)$
in
$U^\circ_d$
for each
$I \subsetneq \{ 1, \ldots, d+1 \}$.
We denote by
$\frakU^\circ_I$
the open neighborhood
$\frakj^\circ(U^\circ_I)$
of
$\Xi_I(\chi)$.
Then
\begin{align*}
U_d
=
\bigcup_{I \subsetneq \{ 1, \ldots, d+1 \}}
T^I \cdot (U^\circ_I \cap T^{d+1}_{\bR < 0}), \
\frakU_d
=
\bigcup_{I \subsetneq \{ 1, \ldots, d+1 \}}
\bT^I \cdot (\frakU^\circ_I \cap \frakt^*_d)
\end{align*}
are respectively open neighborhoods of
$\Core(\tilde{P}_d), \frakL_d$.
Since the matched local coisotropic foliations correspond to the moment maps for the Hamiltonian actions of
$T^I, \bT^I$,
the symplectomorphism
$\frakj^\circ$
canonically extends to
$\frakj \colon U_d \to \frakU_d$.
\end{proof}

On
$\frakU_d$
there are two Liouville forms
$\alpha_{T^*\bT^d} |_{\frakU_d}$
and
$\beta_{d+1}
=
(\frakj^{-1})^* \beta_{\tilde{P}_d}$.
When
$\hat{\chi} = \chi / (d+1) \in \bZ$
the function
\begin{align*}
\mu^{-1}_\Delta (\chi) \cap \Lambda_{d+1}
\to
T^1, \
(\theta, \xi)
\mapsto
\sum^{d+1}_{i = 1} (\xi_i + \hat{\chi}) \theta_i
\end{align*}
is invariant under the $T^1_\Delta$-action
and
descends to an integral structure
$f
\colon
\frakL_d
\to
T^1$
\cite[Definition 5.17(1)]{Nad}.
By
\cite[Remark 5.18(1)]{Nad}
the graph
$\Gamma_{\frakL_d, -f}$
of
$-f$
gives a Legendrian lift of
$\frakL_d$
to the circular contactification
\begin{align*}
(N_d, \lambda_d)
=
(\frakU_d \times T^1, \alpha_{T^*\bT^d} |_{\frakU_d} + dt).
\end{align*}
Since we have
$\beta_{d+1} |_{\frakL_d} = 0$,
the Lagrangian
$\frakL_d$
is exact
\cite[Definition 5.17(2)]{Nad}.
By
\cite[Definition 5.18(2)]{Nad}
the zero section
$\frakL_d \times \{ 0 \}$
gives a Legendrian lift of
$\frakL_d$
to the circular contactification
\begin{align*}
(N^\prime_d, \lambda^\prime_d)
=
(\frakU_d \times T^1, \beta_{d+1} + dt).
\end{align*}

The contact geometry of
$(N_d, \lambda_d)$
near
$\Gamma_{\frakL_d, -f}$
is equivalent to that of
$(N^\prime_d, \lambda^\prime_d)$
near
$\frakL_d \times \{ 0 \}$.

\begin{lemma}[{\cite[Section 5.3]{Nad}}] \label{lem:Conta-Conta1}
There is a contactomorphism
\begin{align*}
G \colon (N_d, \lambda_d) \to (N^\prime_d, \lambda^\prime_d), \
(([\theta], \xi), t)
\mapsto
(([\theta], \xi), t + g([\theta], \xi))
\end{align*}
which makes the diagram
\begin{align*}
\begin{gathered}
\xymatrix{
N_d \ar[r]^-{\sim}
& N^\prime_d \\
\Gamma_{\frakL_d, -f} \ar@{_(->}[u] \ar[r]^-{\sim}
& \frakL_d \times \{ 0 \} \ar@{_(->}[u]
}
\end{gathered}
\end{align*}
commute,
where the vertical arrows are the canonical inclusions.
\end{lemma}
\begin{proof}
The difference
$\alpha_{T^*\bT^d} |_{\frakU_d} - \beta_{d+1}$
is
closed
and
integral,
as we have
$\beta_{d+1} |_{\frakL_d} = 0$.
Since the inclusion
$\frakL_d 
\subset
\frakU_d$
is a homotopy equivalence,
there is a unique function
$g \colon \frakU_d \to T^1$
such that
$dg = \alpha_{T^*\bT^d} |_{\frakU_d} - \beta_{d+1}$
with normalization
$g |_{\frakL_d} = f$.
Then one obtains the desired map from
\cite[Remark 5.16]{Nad}.
\end{proof}

We denote by
$\Omega^\infty_{d+1}
\subset
S^* T^{d+1}
=
(T^* T^{d+1} \setminus T^{d+1}) / \bR_{>0}$
and
$\Lambda^\infty_{d+1}$
the spherical projectivizations of the open conic subset
$\Omega_{d+1}
=
\mu^{-1}_\Delta(\bR_{> 0})
\subset T^* T^{d+1}$
and
the Lagrangian
$\Lambda_{d+1}$.
The projection
$\Omega_{d+1}
\to
\Omega^\infty_{d+1}
=
\Omega_{d+1} / \bR_{>0}$
induces a contactomorphism from
$\mu^{-1}_\Delta(\chi)$
to
$\Omega^\infty_{d+1}$.
Let
$\delta \colon T^{d+1} \to T^1$
be the diagonal character.
By
\cite[Lemma 5.19]{Nad}
the map
\begin{align*}
(p_\chi, \delta)
\colon
\Omega^\infty_{d+1}
\cong
\mu^{-1}_\Delta(\chi)
\to 
T^* \bT^d \times T^1, \
(\theta, \xi) \mapsto (([\theta], \xi_1 - \hat{\chi}, \ldots, \xi_{d+1} - \hat{\chi}), \sum^{d+1}_{i = 1} \theta_i)
\end{align*}
defines a finite contact cover for
$\chi = d+1$.
The cover is trivializable over
$(N_d, \lambda_d)$
with a canonical section
$s \colon N_d \to \Omega^\infty_{d+1}$
satisfying
$s(\Gamma_{\frakL_d, -f})
=
\Lambda^\infty_{d+1}$.

The contact geometry of
$(N_d, \lambda_d)$
near
$\Gamma_{\frakL_d, -f}$
is equivalent to that of
$\Omega^\infty_{d+1}$
near
$\Lambda^\infty_{d+1}$.
 
\begin{lemma}[{\cite[Lemma 5.19]{Nad}}] \label{lem:Conta-Conta2}
There is an open contactomorphism
$s \colon (N_d, \lambda_d) \to \Omega^\infty_{d+1}$
which makes the diagram
\begin{align*}
\begin{gathered}
\xymatrix{
N_d \ar[r]^-{\sim}
& s(N_d) \\
\Gamma_{\frakL_d, -f} \ar@{_(->}[u] \ar[r]^-{\sim}
& \Lambda^\infty_{d+1} \ar@{_(->}[u]
}
\end{gathered}
\end{align*}
commute,
where the vertical arrows are the canonical inclusions.
\end{lemma}

Consider the symplectization
$\tilde{P}_d \times T^1 \times \bR$
of the circular contactification
$(\tilde{P}_d \times T^1, \beta_{d+1} + dt)$
whose Liouville form is given by
$e^u(\beta_{d+1} + dt)$.
The skeleton
$\Core(\tilde{P}_d)
\subset
\tilde{P}_d$
lifts to the Legendrian submanifold
$\Core(\tilde{P}_d) \times \{ 0 \}
\subset
\tilde{P}_d \times T^1$,
which in turn lifts to a conic Lagrangian
$\Core(\tilde{P}_d) \times \{ 0 \} \times \bR
\subset
\tilde{P}_d \times T^1 \times \bR$
along the canonical projections.
Note that
the contact geometry of a cooriented contact manifold is equivalent to the conic symplectic geometry of its symplectization.
In particular,
taking the inverse image under the canonical projection induces a bijection from subspaces of the contact manifold to conic subspaces of its symplectization.

The symplectic geometry of the open neighborhood
$U_d \times T^1 \times \bR$
of
$\Core(\tilde{P}_d) \times \{ 0 \} \times \bR$
in
$\tilde{P}_d \times T^1 \times \bR$
is equivalent to that of the open neighborhood
$s(N_d) \times \bR$
of
$\Lambda^\infty_{d+1} \cap \Omega_{d+1}$
in
$\Omega_{d+1}$.

\begin{theorem}[{\cite[Theorem 5.23]{Nad}}] \label{lem:Symp-Symp2}
There is an open symplectomorphism
$U_d \times T^1 \times \bR
\hookrightarrow
\Omega_{d+1}$
which makes the diagram
\begin{align*}
\begin{gathered}
\xymatrix{
U_d \times T^1 \times \bR \ar@{^(->}[r]
& \Omega_{d+1} \\
\Core(\tilde{P}_d) \times \{ 0 \} \times \bR \ar@{_(->}[u] \ar[r]^-{\sim}
& \Lambda^\infty_{d+1} \cap \Omega_{d+1} \ar@{_(->}[u]
}
\end{gathered}
\end{align*}
commute,
where the vertical arrows are the canonical inclusions.
\end{theorem}
\begin{proof}
The restriction of
$\frakj$
from
\pref{lem:Symp-Symp1}
induces a symplectomorphism
\begin{align*}
U_d \times T^1 \times \bR
\to
\frakU_d \times T^1 \times \bR 
\end{align*}
which sends
$\Core(\tilde{P}_d) \times \{ 0 \} \times \bR$
to
$\frakL_d \times \{ 0 \} \times \bR$.
The inverse of the contactomorphism
$G$
from
\pref{lem:Conta-Conta1}
induces a symplectomorphism
\begin{align*}
(N^\prime_d, \lambda^\prime_d) \times \bR
=
\frakU_d \times T^1 \times \bR
\to
\frakU_d \times T^1 \times \bR
=
(N_d, \lambda_d) \times \bR
\end{align*}
which sends
$\frakL_d \times \{ 0 \} \times \bR$
to
$\Gamma_{\frakL_d, -f} \times \bR$.
The contactomorphism
$s$
from
\pref{lem:Conta-Conta2}
induces a symplectomorphism
\begin{align*}
(N_d, \lambda_d) \times \bR
=
\frakU_d \times T^1 \times \bR
\to
s(\frakU_d \times T^1) \times \bR 
=
s(N_d) \times \bR
\end{align*}
which sends
$\Gamma_{\frakL_d, -f} \times \bR$
to
$\Lambda^\infty_{d+1} \cap \Omega_{d+1}$.
Note that
the symplectization of
$\Omega^\infty_{d+1}$
is isomorphic to
$\Omega_{d+1}$,
as
$\Omega_{d+1}$
does not intersect the zero section
$T^{d+1}$.
\end{proof}

\subsection{$A$-side category}
Let
$Z$
be a real analytic manifold. 
We denote by
$\Sh^\diamondsuit (Z)$
the dg category of large constructible sheaves on
$Z$.
Here,
by a
\emph{large constructible sheaf}
$\scrF$
we will mean a complex of $\bC$-vector spaces for
which there exists a Whitney stratification
$\cS = \{ Z_\alpha \}$
of
$Z$
such that
$\cH^i (\scrF) |_{Z_\alpha}$
are locally constant for all
$i$.
We denote by
$\Sh^\diamondsuit_\cS(Z)$
the full dg subcategory of such sheaves,
called
large $\cS$-constructible sheaves.

Fix a point
$(z, \xi) \in T^*Z$.
Let
$B \subset Z$
be a sufficiently small open ball around
$z \in Z$
and
$f \colon B \to \bR$
a
\emph{compatible test function},
i.e.,
a smooth function with
$f(z) = 0$
and
$df|_z = \xi$.
Consider the
\emph{vanishing cycle functor}
\begin{align*}
\phi_f \colon \Sh^\diamondsuit(Z) \to \Mod(\bC), \
\scrF \mapsto \Gamma_{f \geq 0}(B, \scrF |_B). 
\end{align*}

\begin{definition}
The
\emph{microsupport}
$\ssupp(\scrF) \subset T^*Z$
of
$\scrF \in \Sh^\diamondsuit(Z)$
is the largest closed subset
with
$\phi_f(\scrF) \cong 0$
for any
$(z, \xi) \in T^*Z \setminus \ssupp(\scrF)$
and
its compatible test function
$f$.
\end{definition}

By
\cite[Theorem 8.4.2]{KS}
the microsupport
$\ssupp(\scrF)$
is a closed conic Lagrangian submanifold of
$T^*Z$.
For a conic Lagrangian
$\Lambda \subset T^*Z$,
we denote by
$\Sh^\diamondsuit_\Lambda(Z)$
the full dg subcategory of large constructible sheaves with microsupport in
$\Lambda$.
Consider the dg category
$\Sh(Z)$
of constructible sheaves on
$Z$.
Here,
by a
\emph{constructible sheaf}
$\scrF$
we will mean a complex of $\bC$-vector spaces for
which there exists a Whitney stratification
$\cS = \{ Z_\alpha \}$
of
$Z$
such that
$\cH^i (\scrF) |_{Z_\alpha}$
are locally constant of finite rank for all
$i$.
We denote by
$\Sh_\cS(Z)$
the full dg subcategory of such sheaves,
called
$\cS$-constructible sheaves.

For
a closed conic Lagrangian
$\Lambda \subset T^*Z$
and
an open conic subspace
$\Omega \subset T^* Z$,
consider the Verdier localization
\begin{align*}
\mSh^{\diamondsuit, pre}_\Lambda(\Omega)
=
\Sh^\diamondsuit_{\Lambda \cup (T^* Z \setminus \Omega)}(Z)
/
\Sh^\diamondsuit_{T^* Z \setminus \Omega}(Z).
\end{align*}
Given an inclusion
$\Omega \subset \Omega^\prime$
of open conic subspaces of
$T^*Z$,
there is a restriction functor
\begin{align*}
\rho^{pre}_{\Omega \subset \Omega^\prime}
\colon
\mSh^{\diamondsuit, pre}_\Lambda(\Omega^\prime)
\to
\mSh^{\diamondsuit, pre}_\Lambda(\Omega).
\end{align*}
The assignments
$\Omega
\mapsto
\mSh^{\diamondsuit, pre}_\Lambda(\Omega)$
and
$(\Omega \subset \Omega^\prime)
\mapsto
\rho^{pre}_{\Omega \subset \Omega^\prime}$
assemble into a presheaf of dg categories supported along
$\Lambda$.
We denote
by
$\mSh^\diamondsuit_\Lambda$
its sheafification
and
by
$\rho_{\Omega \subset \Omega^\prime}$
the restriction functor associated with
$\rho^{pre}_{\Omega \subset \Omega^\prime}$.

\begin{definition}
The
\emph{category of large microlocal sheaves}
on
$\Omega$
supported along
$\Lambda$
is the dg category
$\mSh^\diamondsuit_\Lambda (\Omega)$.
\end{definition}

\begin{remark}
The sheaf
$\mSh^\diamondsuit$
is called the
\emph{Kashiwara--Shapira stack}
in
\cite[Section 7.3.1]{GS1}
and
$\mSh^\diamondsuit_\Lambda$
is a sheaf of full dg subcategories of
$\mSh^\diamondsuit$.
\end{remark}

Given an inclusion
$\Lambda
\subset
\Lambda^\prime$
of closed conic Lagrangians,
there is a full embedding
\begin{align*}
i_{\Lambda \subset \Lambda^\prime}
\colon
\mSh^\diamondsuit_\Lambda
\hookrightarrow
\mSh^\diamondsuit_{\Lambda^\prime}
\end{align*}
which preserves products.
Hence it admits a left adjoint 
\begin{align*}
i^l_{\Lambda \subset \Lambda^\prime}
\colon
\mSh^\diamondsuit_{\Lambda^\prime}
\to
\mSh^\diamondsuit_\Lambda
\end{align*}
which preserves coproducts.
Moreover,
$i^l_{\Lambda \subset \Lambda^\prime}$
is a Verdier localization.

\begin{definition}
For
a closed conic Lagrangian
$\Lambda \subset T^*Z$
and
an open conic subspace
$\Omega \subset T^* Z$,
the
\emph{category of wrapped microlocal sheaves}
on
$\Omega$
supported along
$\Lambda$
is the full dg category
$\mSh_\Lambda (\Omega)
\subset
\mSh^\diamondsuit_\Lambda (\Omega)$
of compact objects.
\end{definition}

Since
$\rho_{\Omega \subset \Omega^\prime}$
preserves products,
it admits a left adjoint
which preserves coproducts
and
the restriction to compact objects yields a corestriction functor 
\begin{align*}
\rho^l_{\Omega \subset \Omega^\prime}
\colon
\mSh_\Lambda(\Omega)
\to
\mSh_\Lambda(\Omega^\prime).
\end{align*}

\begin{lemma}[{\cite[Proposition 3.16]{Nad}}] \label{lem:cosheaf}
The assignments
$\Omega
\mapsto
\mSh_\Lambda(\Omega)$
and
$(\Omega \subset \Omega^\prime)
\mapsto
\rho^l_{\Omega \subset \Omega^\prime}$
assemble into a cosheaf of dg categories supported along
$\Lambda$.
Moreover,
there exists a Whitney stratification of
$\Lambda$ 
the restriction of
$\mSh_\Lambda$
to whose strata are locally constant.
\end{lemma}

Recall that
the intersection of
the microsupport
$\ssupp(\scrF)$
and
$Z$
coincides with the support of
$\scrF$.
The second statement of
\pref{lem:cosheaf}
means that
$\mSh_\Lambda$
is a constructible cosheaf of dg categories over
$\Lambda$.
So it makes sense to denote its global sections by
$\mSh_\Lambda(\Lambda)$.
Up to taking the opposite category,
$\mSh_\Lambda(\Lambda)$
is equivalent to the wrapped Fukaya category of some Weinstein manifold in the following sense.

\begin{lemma}[{\cite[Theorem 1.4]{GPS3}}] \label{lem:Fuk-mSh}
Let
$H$
be a real analytic Weinstein manifold.
For any stable polarization of
$H$,
there is a canonical equivalence
\begin{align} \label{eq:GPS3}
\Fuk^c(H)^{op} = \mSh_{\Core(H)}(\Core(H))
\end{align}
where
$\Fuk^c(H)^{op}$
is the wrapped Fukaya category of
$H$.
\end{lemma}

It is known that
all complete intersections in
$\bC^N$
are stably polarized.
Rather than passing to the opposite category,
one can equivalently negate the symplectic form on
$H$.
Throughout the paper,
this way we will avoid passing to the opposite category without comments.
For instance,
when we write
\begin{align*}
\Fuk(H) = \mSh^\diamondsuit_{\Core(H)}(\Core(H))
\end{align*}
for the equivalence between their ind-completions,
we tacitly negate the symplectic form on
$H$
to obtain
\begin{align*}
\Fuk^c(H) = \mSh_{\Core(H)}(\Core(H))
\end{align*}
from
\pref{eq:GPS3}
and
then pass to ind-completions.

\subsection{$B$-side category}
For any stable dg category
$\scrC$
we denote by
$\scrC_{\bZ_2}$
its folding,
i.e.,
$\scrC_{\bZ_2}$
is the stable envelope of the
$\bZ_2$-dg category with the same objects as
$\scrC$
whose morphism complex for
$c_1, c_2 \in \scrC_{\bZ_2}$
is given by
\begin{align*}
\Hom^0_{\scrC_{\bZ_2}}(c_1, c_2)
=
\bigoplus_{n \in \bZ} \Hom^{2n}_\scrC (c_1, c_2), \
\Hom^1_{\scrC_{\bZ_2}}(c_1, c_2)
=
\bigoplus_{n \in \bZ} \Hom^{2n+1}_\scrC (c_1, c_2).
\end{align*} 
For any stable $\bZ_2$-dg category
$\scrC$
we denote by
$\scrC_{2 \bZ}$
its unfurling,
i.e.,
$\scrC_{2 \bZ}$
is the $2$-periodic dg category with the same objects as
$\scrC$
whose morphism complex for
$c_1, c_2 \in \scrC_{2 \bZ}$
is given by
\begin{align*}
\Hom^n_{\scrC_{2 \bZ}}(c_1, c_2)
=
\Hom^{\bar{n}}_\scrC (c_1, c_2), \
n \in \bZ
\mapsto
\bar{n} \in \bZ_2.
\end{align*}
The folding
and
unfurling
define equivalences between stable
$\bZ_2$-dg categories
and
$2$-periodic dg categories
\cite[Section 2.1]{Nad}. 

Let
$\bA^{d+2} = \Spec R_{d+2}$
for
$R_{d+2} = \bC [y_1, \ldots, y_{d+2}]$
and
$W_{d+2} = y_1 \cdots y_{d+2} \in R_{d+2}$.

\begin{definition}
A matrix factorization for the pair
$(\bA^{d+2}, W_{d+2})$
is given by the diagram
\begin{align*}
V^0 \xrightarrow{d_0} V^1 \xrightarrow{d_1} V^0
\end{align*}
where
$V^0 \oplus V^1$
is a $\bZ_2$-graded free $R_{d+2}$-modules of finite rank
and
$d_0 \in \Hom(V^0, V^1), d_1 \in \Hom(V^1, V^0)$
satisfy
$d_1 d_0
=
W_{d+2} \cdot \id_{V^0},
d_0 d_1
=
W_{d+2} \cdot \id_{V^1}$.
We denote by
$\MF(\bA^{d+2}, W_{d+2})$
the $\bZ_2$-dg category of the matrix factorizations for
$(\bA^{d+2}, W_{d+2})$
with obvious morphisms.
\end{definition}

Let
$\underline{\scrO}^i_{d+1}$
be the matrix factrization
\begin{align*}
R_{d+2}
\xrightarrow{W^i_{d+2}}
R_{d+2}
\xrightarrow{y_i}
R_{d+2}
\end{align*}
with
$W^i_{d+2} = W_{d+2} / y_i$
for
$i = 1, \ldots, d+2$.

\begin{lemma}[{\cite[Proposition 2.1]{Nad}}]
The $\bZ_2$-dg category
$\MF(\bA^{d+2}, W_{d+2})$
is split-generated by
$\{ \underline{\scrO}^i_{d+1} \}^{d+1}_{i = 1}$.
There are isomorphisms of $\bZ_2$-graded $\bC$-modules
\begin{align*}
\begin{gathered}
H^*(\Hom(\underline{\scrO}^i_{d+1}, \underline{\scrO}^i_{d+1}))
\cong
R_{d+2} / (y_i, W^i_{d+2}), \
1 \leq i \leq d+2, \\
H^*(\Hom(\underline{\scrO}^i_{d+1}, \underline{\scrO}^j_{d+1}))
\cong
R_{d+2} / (y_i, y_j) [-1], \
1 \leq i \neq j \leq d+2.
\end{gathered}
\end{align*}
\end{lemma}

\begin{remark}
The matrix factorization
$\underline{\scrO}^{d+2}_{d+1}$
belongs to the stable envelope of
$\{ \underline{\scrO}^i_{d+1} \}^{d+1}_{i = 1}$.
\end{remark}

Let
$Y_d = \Spec R_{d+1} / (W_{d+1})$
be the union of the coordinate hyperplanes
$Y^i_d = \Spec R_{d+1} / (y_i),
1 \leq i \leq d+1$.
We denote by
$\scrO^i_d$
the structure sheaf of
$Y^i_d$.

\begin{lemma}[{\cite[Proposition 2.2]{Nad}}]
The dg category
$\Coh(Y_d)$
is split-generated by
$\{ \scrO^i_d \}^{d+1}_{i = 1}$.
There are isomorphisms of $\bZ$-graded $\bC$-modules
\begin{align*}
H^*(\Hom(\scrO^i_d, \scrO^i_d))
\cong
R_{d+1}[u] / (y_i, uW^i_{d+1}), \
1 \leq i \leq d+1, \\
H^*(\Hom(\scrO^i_d, \scrO^j_d))
\cong
R_{d+1}[u] / (y_i, y_j) [-1], \
1 \leq i \neq j \leq d+1,
\end{align*}
where
$u$
is a variable of cohomological degree
$2$
\end{lemma}

\begin{lemma}[{\cite[Theorem 3.7]{Orl}}]
Let
$D_{sing}(Y_{d+1}) = \Coh(Y_{d+1}) / \Perf(Y_{d+1})$
be the $2$-periodic dg quotient category of singularities.
Then there is an equivalence
\begin{align*}
\begin{gathered}
\MF(\bA^{d+2}, y_1 \cdots y_{d+2})
\to
D_{sing}(Y_{d+1}), \\
(V^0 \xrightarrow{d_0} V^1 \xrightarrow{d_1} V^0)
\mapsto
\Coker(d_1). 
\end{gathered}
\end{align*} 
\end{lemma}

\begin{lemma}[{\cite[Proposition 2.3]{Nad}}] \label{lem:Coh-MF}
Let
$\pi_{d+1, d} \colon Y_{d+1} \to Y_d$
be the natural projection.
Then the pullback functor
$\pi^*_{d+1, d} \colon \Coh(Y_d) \to \Coh(Y_{d+1})$
induces an equivalence
\begin{align*}
\Coh(Y_d)_{\bZ_2}
\simeq
\MF(\bA^{d+2}, y_1 \cdots y_{d+2})
\end{align*}
which sends
$\scrO^i_d$
to
$\underline{\scrO}^i_{d+1}$
for
$1 \leq i \leq d+1$
and
$u$
to
$y_{d+2}$.
\end{lemma}

\subsection{Homological mirror symmetry}
Let
$\cI_{d+1}$
be the category whose objects are subsets
$I \subset \{ 1, \ldots, d+1 \}$
and
whose morphisms are given by inclusions.
We denote by
$\cI^\circ_{d+1}$
the full subcategory of proper subsets.
For
$I \in \cI^\circ_{d+1}$
we define
$\Lambda_I$
as the product conic Lagrangian
$(\Lambda_1)^I \subset (T^* T^1)^I$.
Consider the hyperbolic restriction
\begin{align*}
\eta_{I \subset I^\prime}
=
(p_{I \subset I^\prime})_* (q_{I \subset I^\prime})^!
\colon
\Sh^\diamondsuit_{\Lambda_{I^\prime}}(T^{I^\prime})
\to
\Sh^\diamondsuit_{\Lambda_I}(T^I)
\end{align*}
where
$p_{I \subset I^\prime}
\colon
T^I \times [0, \frac{1}{2})^{I^\prime \setminus I}
\to
T^I$
is the projection
and
$q_{I \subset I^\prime}
\colon
T^I \times [0, \frac{1}{2})^{I^\prime \setminus I}
\hookrightarrow
T^{I^\prime}$
is the canonical inclusion.
Note that
$\eta_{I \subset I^\prime}$
is the product of hyperbolic restrictions in the coordinate directions indexed by
$I^\prime \setminus I$
and
the identity in the coordinate directions indexed by
$I$.
We denote by
$\eta_I$
the hyperbolic restriction with
$I^\prime = \{ 1, \ldots, d+1 \}$.

\begin{lemma}[{\cite[Lemma 5.25, 5.26]{Nad}}]
There is an equivalence
\begin{align*}
\Sh^\diamondsuit_{\Lambda_{d+1}}(T^{d+1})
\simeq
\Qcoh(\bA^{d+1})
\end{align*}
which makes the diagram
\begin{align*}
\begin{gathered}
\xymatrix{
\Sh^\diamondsuit_{\Lambda_{d+1}}(T^{d+1}) \ar[r]^-{\sim} \ar_{\eta_I}[d]
& \Qcoh(\bA^{d+1}) \ar^{\iota^*_I}[d] \\
\Sh^\diamondsuit_{\Lambda_I}(T^I) \ar[r]^-{\sim}
& \Qcoh(\bA^I)
}
\end{gathered}
\end{align*}
commute,
where
\begin{align*}
\iota_I
\colon
\bA^I = \Spec \bC [y_i \ | \ i \in I]
\hookrightarrow
\bA^{d+1} = \Spec \bC [y_1, \ldots, y_{d+1}]
\end{align*}
is the canonical inclusion of the subvariety defined by
$y_j = 0$
for
$j \in I^c$.
\end{lemma}

For
$I \in \cI^\circ_{d+1}$
we define
$\Omega_I$
as the open conic subset
\begin{align*}
\Omega_I
=
\{ (\theta, \xi) \in T^* T^{d+1} | \Sigma^{d+1}_{i=1} \xi_i > 0, \xi_j \neq 0 \ \text{for}\ j \in I^c \}
\subset
\Omega_{d+1}.
\end{align*}
The collection
$\{ \Omega_I \}_{I \in \cI^\circ_{d+1}}$
forms an open conic cover of
$\Omega_{d+1}$
satisfying
$\Omega_{I \cap I^\prime} = \Omega_I \cap \Omega_{I^\prime}$.
Note that
we have
$\Omega_I \subset \Omega_{I^\prime}$
whenever
$I \subset I^\prime$.
Let
$^{**}\DG$
be the category of
cocomplete dg categories
and
functors
which preserve
colimits
and
compact objects
\cite[Section 7.3.1]{GS1}.
Consider a functor
\begin{align*}
\mSh^\diamondsuit
\colon
(\cI^\circ_{d+1})^{op}
\to
^{**}\DG, \
I
\mapsto
\mSh^\diamondsuit_{\Lambda_I}(\Omega_I) 
=
\mSh^\diamondsuit_{\Lambda_{d+1}}(\Omega_I)
\end{align*}
which sends inclusions
$I \subset I^\prime$
to the restriction functors
$\rho_{I \subset I^\prime}$
along the inclusions
$\Omega_{I} \subset \Omega_{I^\prime}$.
We denote by
$\rho_I$
the restriction functor with
$I^\prime = \{ 1, \ldots, d+1 \}$.
As
$\mSh^\diamondsuit_{\Lambda_{d+1}}$
is a sheaf,
the canonical functor
\begin{align*}
\mSh^\diamondsuit_{\Lambda_{d+1}}(\Omega_{d+1})
\to
\lim_{I \in (\cI^\circ_{d+1})^{op}} \mSh^\diamondsuit_{\Lambda_I}(\Omega_I)
\end{align*}
is an equivalence.

\begin{theorem}[{\cite[Theorem 5.27]{Nad}}] \label{thm:mSh-Coh}
There is an equivalence
\begin{align*}
\mSh^\diamondsuit_{\Lambda_{d+1}}(\Omega_{d+1})
\simeq
\IndCoh(Y_d)
=
\lim_{I \in (\cI^\circ_{d+1})^{op}} \Qcoh(\bA^I)
\end{align*}
which makes the diagram
\begin{align*}
\begin{gathered}
\xymatrix{
\mSh^\diamondsuit_{\Lambda_{d+1}}(\Omega_{d+1}) \ar[r]^-{\sim} \ar_{\rho_I}[d]
& \IndCoh(Y_d) \ar^{\tau_I}[d] \\
\mSh^\diamondsuit_{\Lambda_I}(\Omega_I) \ar[r]^-{\sim}
& \Qcoh(\bA^I)
}
\end{gathered}
\end{align*}
commute,
where
$\tau_I$
is the canonical functor.
\end{theorem}
\begin{proof}
There is a natural isomorphism
$\mSh^\diamondsuit_{\Lambda_{d+1}}
\to
\Sh^\diamondsuit_{\Lambda_{d+1}}$
as a sheaf of categories over
$\cI^\circ_{d+1}$
induced by
$\eta_I$.
Indeed,
$\eta_I$
factors through the microlocalization
\begin{align*}
\Sh^\diamondsuit_{\Lambda_{d+1}}(T^{d+1})
\to
\mSh^\diamondsuit_{\Lambda_I}(\Omega_I)
\xrightarrow{\tilde{\eta}_I}
\Sh^\diamondsuit_{\Lambda_I}(T^I)
\end{align*}
and
for inclusions
$I \subset I^\prime$
the diagrams
\begin{align*}
\begin{gathered}
\xymatrix{
\mSh^\diamondsuit_{\Lambda_{I^\prime}}(\Omega_{I^\prime}) \ar[r]^-{\tilde{\eta}_{I^\prime}} \ar_{\rho_{I \subset I^\prime}}[d]
& \Sh^\diamondsuit_{\Lambda_{I^\prime}}(T^{I^\prime}) \ar^{\eta_{I \subset I^\prime}}[d] \\
\mSh^\diamondsuit_{\Lambda_I}(\Omega_I) \ar[r]^-{\tilde{\eta}_I}
& \Sh^\diamondsuit_{\Lambda_I}(T^I)
}
\end{gathered}
\end{align*}
commute.
Note that
the hyperbolic restriction in the coordinate direction indexed by
$j \in I^c$
vanishes on sheaves
whose microsupport does not intersect the locus
$\{ \xi_j > 0 \} \subset T^* T^{d+1}$. 
For each
$I \in (\cI^\circ_{d+1})^{op}$ 
the functor
$\tilde{\eta}_I$
is an equivalence,
since it admits an inverse induced by the pushforward
\begin{align*}
\Sh^\diamondsuit_{\Lambda_I}(T^I)
\to
\Sh^\diamondsuit_{\Lambda_{d+1}}(T^{d+1})
\end{align*}
along the inclusion
$T^I \hookrightarrow T^{d+1}$.
\end{proof}

\begin{corollary}[{\cite[Corollary 5.28]{Nad}}] \label{cor:Nadler}
There is an equivalence
\begin{align*}
\Fuk(\tilde{P}_d)_{\bZ_2}
=
\mSh^\diamondsuit_{\Lambda_{d+1}}(\Omega_{d+1})_{\bZ_2}
\simeq
\MF^\infty(\bA^{d+2}, y_1 \cdots y_{d+2})
\end{align*}
where
$\MF^\infty(\bA^{d+2}, y_1 \cdots y_{d+2})$
is the ind-completion of
$\MF(\bA^{d+2}, y_1 \cdots y_{d+2})$. 
\end{corollary}
\begin{proof}
The first equivalence follows from
\pref{lem:Fuk-mSh}.
Passing to $\bZ_2$-folding in
\pref{thm:mSh-Coh},
one obtains the second equivalence from
\pref{lem:Coh-MF}
extended to ind-completions.
\end{proof}

\section{Critical loci of Landau--Ginzburg models for very affine hypersufaces}
In this section,
following
\cite[Section 3]{AAK},
we realize the mirror pair for very affine hypersuface as critical loci of the associated Landau--Ginzburg model.
They give rise to fibrations over the tropical hypersurface equipped with the canonical stratification.
 
\subsection{Very affine hypersurfaces}
Let
$\bT
=
M^\vee_{\bR / \bZ}
=
M^\vee_\bR / M^\vee$
be a real $(d+1)$-dimensional torus with cocharacter lattice
$M^\vee$.
We denote by
$\bT_\bC = M^\vee_{\bC^*}$
the associated complex torus.
Taking its dual,
one obtains the complex torus
$\bT^\vee_\bC = M_{\bC^*}$
associated with
$\bT^\vee = M_{\bR / \bZ} = M_\bR / M$
whose cocharacter lattice is
$M$.
We choose an inner product to identify
$T \bT^\vee$
with
$T^* \bT^\vee$
and
regard
$\bT^\vee_\bC \cong T \bT^\vee \cong T^* \bT^\vee$
as an exact symplectic manifold
equipped with the standard Liouville structure.

\begin{definition}
Let
$\cT$
be a triangulation of a lattice polytope
$\Delta^\vee \subset M^\vee_\bR$.
We call
$\cT$
\emph{adapted}
if there is a convex piecewise function
$\rho \colon \Delta^\vee \to \bR$
whose corner locus is
$\cT$.
We call
$\cT$
\emph{unimodular}
if each cell is congruent to the standard $(d+1)$-simplex
$\Delta_{d+1}$
under the $\GL(d+1, \bZ)$-action. 
\end{definition}

For a lattice polytope
$\Delta^\vee \subset M^\vee_\bR$,
choose an adapted unimodular triangulation
$\cT$.
We denote by
$A$
the set of vertices of
$\cT$.
The convex piecewise function
$\rho \colon \Delta^\vee \to \bR$
defines a Laurent polynomial
\begin{align} \label{eq:Laurent}
W_t \colon \bT^\vee_\bC \to \bC, \
x
\mapsto
\sum_{\alpha \in A} c_\alpha t^{-\rho(\alpha)}x^\alpha
\end{align}
in coordinates
$x = (x_1, \ldots, x_{d+1})$
on
$\bT^\vee_\bC$,
where
$c_\alpha \in \bC^*$
are arbitrary constants
and
$t \gg 0$
is a tropicalization parameter.
In the sequel,
we will assume
$c_\alpha = 1$
for all
$\alpha$.
This does not affect the isomorphism type of
$H_t$
as a Liouville manifold,
as long as
$t$
is sufficiently large,
and
therefore it does not impact the HMS statement for
$H_t$
which is our main result.

\begin{definition}
For sufficiently large
$t$
we call the hypersurface
$H_t = W^{-1}_t(0)$
\emph{very affine}.
\end{definition}

Since
$t$
is sufficiently large,
a very affine hypersurface
$H_t$
is smooth.
Due to the above choice of inner product,
we may regard
$H_t$
as a Liouville submanifold of
$\bT^\vee_\bC$.

\begin{definition}
The
\emph{amoeba}
$\Pi_t$
of
$H_t$
is its image under
$\Log_{d+1} \colon \bT^\vee_\bC \to \bR^{d+1}$.
\end{definition}

\begin{definition}
The
\emph{tropical hypersurface}
$\Pi_\Sigma$
associated with
$H_t$
is the hypersurface defined by the
\emph{tropical polynomial}
\begin{align*}
\varphi_W \colon M_\bR \to \bR, \
\varphi_W(m)
=
\max\{\langle m, n \rangle - \rho(n) \ | \ n \in \Delta^\vee \}.
\end{align*}
Namely,
$\Pi_\Sigma$
is the set of points
where the maximum is achieved more than once.
\end{definition}

According to
\cite[Corollary 6.4]{Mik},
when
$t \to \infty$
the rescaled amoeba
$\Pi_t / \log t$
converges to
$\Pi_\Sigma$.
It is known that
$\Pi_\Sigma$
is a deformation retract of
$\Pi_t$
for
$t \gg 0$.
Combinatorially,
$\Pi_\Sigma$
is the dual cell complex of
$\cT$.
In particular,
the set of connected components of
$M_\bR \setminus \Pi_\Sigma$
bijectively corresponds to
$A$
according to
which
$\alpha \in A$
achieves the maximum of
$\langle m, \alpha \rangle - \rho(\alpha)$
for
$m \in M_\bR \setminus \Pi_\Sigma$.
Note that
$M_\bR \setminus \Pi_t$
for
$t \gg 0$
has the same combinatrics as
$M_\bR \setminus \Pi_\Sigma$.

\begin{remark}
Each connected component
$C_\alpha$
of
$M_\bR \setminus \Pi_\Sigma$
is the locus
where the monomial
$t^{- \rho(\alpha)} x^\alpha$
becomes dominant.
\end{remark}

In the sequel,
we will
fix sufficiently large 
$t$
and
drop
$t$
from the notation,
as the specific choice of
$t$
is immaterial for our results.
See Appendix for more detailed explanations.

\subsection{Landau--Ginzburg $A$-models for very affine hypersurfaces}
For
$X = \bT^\vee_\bC \times \bC$
with coordinates
$(x, u) = (x_1, \ldots, x_{d+1}, u)$,
consider a Laurent polynomial
\begin{align*}
W_X \colon X \to \bC, \
(x, u)
\mapsto
u W(x)
\end{align*}
where
$W$
is the Laurent polynomial
\pref{eq:Laurent}.

\begin{definition}
Let
$H \subset \bT^\vee_\bC$
be a very affine hypersurface defined by the Laurent polynomial
$W$
from
\pref{eq:Laurent}.
We call the pair
$(X, W_X)$
the
\emph{Landau--Ginzburg $A$-model}
for
$H$.
\end{definition}

\begin{definition}
The
\emph{Newton polytope}
$\Delta^\vee_X$
of
$W_X$
is the convex hull
\begin{align*}
\Conv(0, -\Delta^\vee \times \{ 1 \})
\subset
M^\vee_\bR \times \bR.
\end{align*}
\end{definition}

\begin{remark}
The polytope
$\Delta^\vee_X$
admits an adapted unimodular star-shaped triangulation
$\tilde{\cT}$
canonically induced by
$\cT$.
Recall that
a triangulation of
$\Delta^\vee_X$
is
\emph{star-shaped}
if all of its simplices not contained in the boundary
$\del \Delta^\vee_X$
share a common vertex
$0$
\cite[Definition 3.3.1]{GS1}.
\end{remark}

\begin{lemma} \label{lem:A-critical}
The critical locus
$\Crit(W_X)$
is given by
$\{ u = 0 \}
\cap
\{ W = 0 \}
\subset
X$.
\end{lemma}
\begin{proof}
Express the tangent map
$d W_X$
of
$W_X$
as a vector
$(u dW, W)$.
Since
$H \subset \bT^\vee_\bC$
is smooth,
$dW$
nowhere vanishes.
Hence
$\rank (dW_X) = 0$
if and only if 
$u = 0$
and
$W = 0$.
\end{proof}

\begin{remark}
By
\pref{lem:A-critical}
the projection
$\pr_1 \colon X = \bT^\vee_\bC \times \bC \to \bT^\vee_\bC$
preserves
$\Crit(W_X)$.
Let
$\ret \colon \Pi \to \Pi_\Sigma$
be the continuous map induced by the retraction.
Then the composition
\begin{align} \label{eq:A-fibration}
f
\colon
H
\cong
\Crit(W_X)
\hookrightarrow
X
\xrightarrow{\Log_{d+1} \circ \pr_1}
\Pi
\xrightarrow{\ret}
\Pi_\Sigma
\end{align}
gives the stratified fibration from
\cite[Theorem 1']{Mik}.
Recall from
\pref{lem:tailored}
that
$k$-th intersections of legs of a tailored pants for
$k \leq d$
has torus factor of dimension
$k$.
Away from lower dimensional strata,
the fiber over a point in a $k$-stratum of
$\Pi_\Sigma$
contains a real $k$-torus in the torus factor.
\end{remark}

\subsection{Landau--Ginzburg $B$-models for very affine hypersurfaces}
Let
$Y$
be the noncompact $(d+2)$-dimensional toric variety associated with the fan
\begin{align*}
\Sigma_Y
=
\Cone(-\cT \times \{ 1 \})
\subset
M^\vee_\bR \times \bR.
\end{align*}
The primitive ray generators of
$\Sigma_Y$
are the vectors of the form
$(-\alpha, 1)$
with
$\alpha \in A$.
Such vectors span a smooth cone of
$\Sigma_Y$
if and only if
$\alpha$
span a cell of
$\cT$.

Dually,
$Y$
is associated with the noncompact moment polytope
\begin{align*}
\Delta_Y
=
\{ (m, u) \subset M_\bR \times \bR \ | \ u \geq \varphi(m) \}.
\end{align*}
The facets of
$\Delta_Y$
correspond to the maximal domains of linearity of
$\varphi$.
Hence the irreducible toric divisors of
$Y$
bijectively correspond to the connected components of
$M_\bR \setminus \Pi_\Sigma$.
In particular,
the combinatrics of toric strata of
$Y$
can be read off
$\Pi_\Sigma$.

\begin{remark}
The noncompact polytope
$\Delta_Y$
is homeomorphic to the image of
$Y$
under the composition
\begin{align} \label{eq:moment}
Y \to (Y)_{\geq 0} \to M_\bR \times \bR
\end{align}
of
the map induced by retraction to the nonnegative real points
with
the restriction of negated algebraic moment map
\cite[Section 12.2]{CLS}.
\end{remark}

\begin{lemma} \label{lem:projection}
Let
$q \colon M_\bR \times \bR \to M_\bR$
be the natural projection.
Then under
$q$
the union of facets of
$\Delta_Y$
homeomorphically maps to
$M_\bR$.
Moreover,
the union of codimension
$2$
faces of
$\Delta_Y$
homeomorphically maps to
$\Pi_\Sigma$.
\end{lemma}
\begin{proof}
By construction of
$\Sigma_Y$
under
$q$
each facet of
$\Delta_Y$
homeomorphically maps to the maximal domain of linearity of
$\varphi$
corresponding to the same
$\alpha \in A$.
Any codimension
$2$
face of
$\Delta_Y$
can be obtained as the intersection of two distinct facets.
Hence
$q$
restricted to the union of codimension
$2$
faces of
$\Delta_Y$
gives an injection to
$\Pi_\Sigma$.
This is also surjective,
as each full dimensional face of
$\Pi_\Sigma$
is adjacent to exactly two maximal domains of linearity.
\end{proof}

For each
$\alpha = (\alpha_1, \ldots, \alpha_{d+1}) \in A$
let
$Y_\alpha =(\bC^*)^{d+1} \times \bC$
with coordinates
$y_\alpha = (y_{\alpha, 1}, \ldots, y_{\alpha, d+1}, v_\alpha)$,
where
$y_{\alpha, 1}, \ldots, y_{\alpha, d+1}, v_\alpha$
are the monomials with weights
\begin{align*}
\eta_1 = (-1,  0, \ldots, 0, -\alpha_1),
\ldots,
\eta_{d+1} = (0, \ldots, 0, -1, -\alpha_{d+1}),
\eta_{d+2} = (0, \ldots, 0, 1)
\in
M \times \bZ.
\end{align*}
Their pairing with the monomial with weight
$(-\alpha, 1) \subset M^\vee \times \bZ$
yields
$0, \ldots, 0, 1$
respectively.

\begin{lemma} \label{lem:chart}
The complex algebraic variety
$Y_\alpha$
is the affine open subset of
$Y$
associated with the ray spanned by
$(-\alpha, 1) \subset M^\vee \times \bZ$.
\end{lemma}
\begin{proof}
Suppose that
$\sigma \in \Sigma_Y(1)$
is the cone associated with the affine open subset
$Y_\alpha \subset Y$.
We have
\begin{align*}
\ddiv (y^{\pm 1}_{\alpha, i})
=
\sum_{\xi \in \sigma(1)}
\langle \pm \eta_i, u_\xi \rangle D_\xi, \
\ddiv (v_\alpha)
=
\sum_{\xi \in \sigma(1)}
\langle \eta_{d+2}, u_\xi \rangle D_\xi, 
\end{align*}
where
$u_\xi$
are primitive ray generators of
$\xi$
and
$D_\xi = \overline{O(\xi)}$
are the closures of the orbits corresponding to
$\xi$.
Since
$y^{\pm 1}_{\alpha, 1}, \ldots, y^{\pm 1}_{\alpha, d+1}$
never vanish on
$Y_\alpha$,
pairing of
$\eta_i$
with the primitive ray generators in
$\sigma$
must yield
$0$
for
$1 \leq i \leq d+1$.
On the other hand,
pairing of
$\eta_{d+2}$
with the primitive ray generators of
$\sigma$
must yield
$1$.
\end{proof}

Due to the above lemma,
$Y_\alpha$
covers the open stratum of
$Y$
and
the open stratum of the irreducible toric divisor corresponding to
$\alpha$.
If
$\alpha, \beta \in A$
are connected by an edge in
$\cT$,
then we glue
$Y_\alpha$
to
$Y_\beta$
with the coordinate transformations
\begin{align*}
y_{\alpha, i} = v^{\beta_i - \alpha_i}_\beta y_{\beta, i}, \
v_\alpha = v_\beta, \
1 \leq i \leq d+1.
\end{align*}
Thus the coordinate charts
$\{ Y_\alpha \}_{\alpha \in A}$
cover the complement in
$Y$
of the codimension more than
$1$
strata.

We may write
$v$
for
$v_\alpha$
as it does not depend on the choice of
$\alpha \in A$. 
Since the weight
$(0, \ldots, 0, 1)$
pairs nonnegatively with the primitive ray generators of
$\Sigma_Y$,
the monomial
$v$
defines a regular function on
$Y$,
which we denote by
$W_Y$.

\begin{definition}
Let
$H \subset \bT^\vee_\bC$
be a very affine hypersurface defined by the Laurent polynomial
$W$
from
\pref{eq:Laurent}.
We call the pair
$(Y, W_Y)$
the
\emph{Landau--Ginzburg $B$-model}
for
$H$.
\end{definition}

\begin{remark}
The pair
$(Y, W_Y)$
is a conjectural SYZ mirror to
$H$
\cite[Theorem 1.6]{AAK}.
\end{remark}

The critical locus
$\Crit(W_Y)$
is the preimage of the codimension
$2$
strata of
$\Delta_Y$
under
\pref{eq:moment}.
One can check this locally in each affine chart, which is isomorphic to
$\bC^{d+2}$
as
$\cT$
is unimodular.  

\begin{lemma} \label{lem:B-critical}
The critical locus
$\Crit(W_Y)$
is given by
$\bigcup_{\alpha \in A}
\overline{Y}_\alpha \setminus Y_\alpha
\subset
Y$.
\end{lemma}
\begin{proof}
For each
$\alpha \in A$
the intersection
$\Crit(W_Y) \cap Y_\alpha$
is empty.
Indeed,
when restricted to
$Y_\alpha$,
the tangent map
$d W_Y$
of
$W_Y$
is expressed as a vector
whose last factor is
$1$.
Hence
$d W_Y |_{Y_\alpha}$
is surjective
and
we obtain
\begin{align*}
\Crit(W_Y)
\subset
Y \setminus \bigcup_{\alpha \in A} Y_\alpha
=
\bigcup_{\alpha \in A} \overline{Y}_\alpha \setminus Y_\alpha.
\end{align*}

Take any point
$y \in \overline{Y}_\alpha \setminus Y_\alpha$.
Suppose that
there is a vertex
$\alpha^\prime \in A$
connected to
$\alpha$
by an edge in
$\cT$
such that
$y \in \overline{Y}_{\alpha^\prime} \setminus Y_{\alpha^\prime}$.
Let
$\sigma \in \Sigma_Y$
be the cone generated by two rays
$\xi_\alpha = \Cone(-\alpha, 1),
\xi_{\alpha^\prime} = \Cone(-\alpha^\prime, 1)$.
Then we have
\begin{align*}
\ddiv (v)
=
\langle \eta_{d+2}, u_{\xi_\alpha} \rangle D_{\xi_\alpha}
+
\langle \eta_{d+2}, u_{\xi_{\alpha^\prime}} \rangle D_{\xi_{\alpha^\prime}}
\end{align*}
on the associated affine open subset
$\Spec \bC[\sigma^\vee \cap (M \times \bZ)]
\cong
(\bC^*)^{d} \times \bC^2$
of
$Y$.
Hence the restriction of
$dW_Y$
vanishes on
$D_{\xi_\alpha} \cap D_{\xi^\prime_{\alpha^\prime}}$,
which is
\begin{align*}
((\overline{Y}_\alpha \setminus Y_\alpha)
\cap
(\overline{Y}_{\alpha^\prime} \setminus Y_{\alpha^\prime})) |_{\Spec \bC[\sigma^\vee \cap (M \times \bZ)]}.
\end{align*}
The union of such intersections for all
$\alpha^\prime \in A$
is
$\overline{Y}_\alpha \setminus Y_\alpha$.
Applying the same argument to the other cases,
we obtain
\begin{align*}
\bigcup_{\alpha \in A} \overline{Y}_\alpha \setminus Y_\alpha
\subset
\Crit(W_Y).
\end{align*}
\end{proof}

\begin{remark}
Since the map
\pref{eq:moment}
sends each $k$-th intersection of
$D_{\xi_\alpha}, \alpha \in A$
to a codimension
$k$
face of
$\Delta_Y$,
by
\pref{lem:B-critical}
it sends
$\Crit(W_Y)$
to the union of codimension
$2$
faces.
On the other hand,
by
\pref{lem:projection}
the map
$q \colon M_\bR \times \bR \to M_\bR$
homeomorphically sends the union of codimension
$2$
faces of
$\Delta_Y$
to
$\Pi_\Sigma$.
Hence the composition
\begin{align} \label{eq:B-fibration}
g
\colon
\Crit(W_Y)
\hookrightarrow
Y
\xrightarrow{q \circ \pref{eq:moment}}
\Pi_\Sigma
\end{align}
gives a stratified fibration.
The fiber over a point in a $k$-stratum is a real $k$-torus
\cite[Prop 12.2.3(b)]{CLS}. 
\end{remark}

\section{Constructible sheaves of categories}
In this section,
we define two constructible sheaves of categories over the tropical hypersurface
$\Pi_\Sigma \subset \bR^{d+1}$
with the canonical stratification
and
a certain topology generated by the vertices.
In the sequel,
by a
\emph{Liouville manifold}
we will mean a Liouville manifold of finite type,
i.e.,
the completion of some Liouville domain.
By a
\emph{Weinstein manifold}
we will mean a Liouville manifold together with a Morse--Bott function constant on the cylindrical ends for
which the Liouville vector field is gradient-like.

\begin{definition} \label{dfn:topology}
We define a topology on
$\Pi_\Sigma$
induced by its canonical stratification.
Namely,
for each vertex
$v \in \Pi_\Sigma$
we define the associated open subset
$U_v$
as the union of all strata adjacent to
$v$,
which is homeomorphic to a $d$-dimensional tropical pants.
For each edge
$e \subset \Pi_\Sigma$
connecting two vertices
$v_1, v_2$
we define the associated open subset
$U_e$
as the intersection
$U_{v_1} \cap U_{v_2}$.
Similarly,
for each $k$-stratum
$S^{(k)} \subset \Pi_\Sigma$
adjacent to
$l$
vertices
$v_1, \ldots, v_l$
we define the associated open subset
$U_{S^{(k)}}$
as the intersection
$U_{v_1} \cap \cdots \cap U_{v_l}$.
\end{definition}

\subsection{$A$-side partially defined presheaf of categories for very affine hypersurfaces}
Fix a pants decomposition of
$H \subset \bT^\vee_\bC \cong T^* T^{d+1}$
\cite[Theorem 1']{Mik}.
Equip
$\Pi_\Sigma$
with the topology from Definition
\pref{dfn:topology}
whose basis is given by the open subsets
$U_{S^{(k)}}$.

\begin{definition}
The
\emph{$A$-side partially defined presheaf}
$\cF^{pre}_A$
of categories for
$H$
is a collection
\begin{align*}
\{ \cF^{pre}_A(U_{S^{(k)}}), R^A_{S^{(k)}, S^{(l)}} \}
\end{align*}
of
sections
and
restriction functors
defined on the basis
$U_{S^{(k)}}$
as follows:
\begin{itemize}
\item
The section over
$U_{S^{(k)}}$
is given by $\bZ_2$-folding of the ind-completion of the wrapped Fukaya category
\begin{align*}
\cF^{pre}_A(U_{S^{(k)}})
=
\Fuk(H_{S^{(k)}})_{\bZ_2}
\end{align*}
where
$H_{S^{(k)}}$
is the inverse image of suitably shrunk
$U_{S^{(k)}}$
under
\pref{eq:A-fibration}.
In other words,
$H_{S^{(k)}}$
is symplectomorphic to the intersection
$\tilde{P}_{S^{(k)}}$
of the corresponding
$k$
legs of
$\tilde{P}_d$.
Here,
we equip
$H_{S^{(k)}}$
with the induced Weinstein structure by the canonical one on
$\tilde{P}_d$.
\item
Along an inclusion
$U_{S^{(l)}} \hookrightarrow U_{S^{(k)}}$
the restriction functor
\begin{align*}
R^A_{S^{(k)}, S^{(l)}}
\colon
\Fuk(H_{S^{(k)}})_{\bZ_2}
\to
\Fuk(H_{S^{(l)}})_{\bZ_2}
\end{align*}
is induced by $\bZ_2$-folding of the ind-completion of the Viterbo restriction
\begin{align*}
\Fuk(\tilde{P}_{d-k}, \beta^i_{\tilde{P}_{d-k}})
\to
\Fuk(L_{d-k, j}(K), \beta^i_{\tilde{P}_{d-k}} |_{L_{d-k, j}(K)})
\end{align*}
from
\cite[Proposition 11.2]{GPS2}
to some leg
$L_{d-k, j}(K)$
for Nadler's Weinstein structure
$\beta^i_{\tilde{P}_{d-k}}$
on
$\tilde{P}_{d-k}$
and
the inductive compatibility
\pref{eq:recursive}
of
$\tilde{P}_{d-k}$.
Here,
$L_{d-k, j}(K)$
is associated with an edge not bounding
$S^{(k)}$
but
$S^{(l)}$
to
which the final leg with respect to
$\beta^i_{\tilde{P}_{d-k}}$
does not correspond. 
\end{itemize}
\end{definition}

To see that
$R^A_{S^{(k)}, S^{(l)}}$
is well defined,
first recall the invariance of the wrapped Fukaya category up to canonical equivalence under deformations of Liouville structures,
which follows from our assumption on Liouville manifolds to be of finite type.
Here,
what we need is the following special case of
\cite[Lemma 3.4]{GPS2}.

\begin{lemma} \label{lem:A-section_unique}
Let
$\lambda_{S^{(k)}}, \lambda^\prime_{S^{(k)}}$
be the completions of two Liouville forms on a Liouville domain
$[H_{S^{(k)}}]$
completing to
$H_{S^{(k)}}$.
Then there is a canonical equivalence
\begin{align*}
\Fuk(H_{S^{(k)}}, \lambda_{S^{(k)}})
\simeq
\Fuk(H_{S^{(k)}}, \lambda^\prime_{S^{(k)}}). 
\end{align*}
\end{lemma}
\begin{proof}
Our argument is essentially the same as
\cite[Lemma 2]{Jef}.
Since the space of Liouville forms for a compact symplectic manifold-with-boundary is convex,
any two Liouville forms on
$[H_{S^{(k)}}]$
are canonically homotopic
and
the homotopy completes to that for
$\lambda_{S^{(k)}}, \lambda^\prime_{S^{(k)}}$.
Then one can apply
\cite[Proposition 11.8]{CE}
to obtain a strictly exact symplectomorphism
$\psi \colon H_{S^{(k)}} \to H_{S^{(k)}}$.
By definition it satisfies
$\psi^* \lambda^\prime - \lambda = df$
for some compactly supported function
$f \colon H_{S^{(k)}} \to \bR$.
In particular,
$\psi$
defines a trivial inclusion of open Liouville sectors in the sense of
\cite[Definition 3.3]{GPS2}.
Then one can apply
\cite[Lemma 3.4]{GPS2}
to see that
the pushforward functor from
\cite[Section 3.6]{GPS1}
gives the canonical equivalence.
\end{proof}

\begin{corollary} \label{cor:A-local}
The section
$\cF^{pre}_A(U_{S^{(k)}})$
is canonically equivalent to
$\Fuk(\tilde{P}_{S^{(k)}}, \beta^i_{\tilde{P}_d} |_{\tilde{P}_{S^{(k)}}})_{\bZ_2}$,
where the final leg with respect to
$\beta^i_{\tilde{P}_d}$
does not correspond to any edge bounding
$S^{(k)}$.
\end{corollary}
\begin{proof}
Each piece of the pants decomposition can be made symplectomorphic to
the domain
$[\tilde{P}_d]$
by
\cite[Remark 5.2]{Mik}.
\end{proof}

Next,
recall that
the Viterbo restriction along an inclusion of Weinstein domains coincides with the quotient by the cocores not in the subdomain.
Here,
what we need is the following special case of
\cite[Proposition 11.2]{GPS2}.

\begin{lemma} \label{lem:A-local-prerestriction}
The Viterbo restriction
\begin{align*}
\Fuk(\tilde{P}_{S^{(k)}}, \beta^i_d |_{\tilde{P}_{S^{(k)}}})
\to
\Fuk(\tilde{P}_{S^{(l)}}, \beta^i_d |_{\tilde{P}_{S^{(l)}}})
\end{align*}
coincides with the quotient by the cocores of
$\tilde{P}_{S^{(k)}}$
not in
$\tilde{P}_{S^{(l)}}$.
\end{lemma}
\begin{proof}
As the other cases can be proved similarly,
we restrict ourselves to the case
where
$S^{(k)} = U_v$
and
$S^{(l)} = U_e$
for some edge
$e$
connecting
$v$
to another vertex
$v^\prime$.
Then the leg
$\tilde{P}_e$
differs from the final leg of
$\tilde{P}_d$
with respect to
$\beta^i_d$.
Since both
$[\tilde{P}_e]$
and
the cobordism
$[\tilde{P}_d] \setminus [\tilde{P}_e]^\circ$
are Weinstein,
one can apply
\cite[Proposition 11.2]{GPS2}
to see that
the Viterbo restriction coincides with the quotient by the cocores of
$\tilde{P}_d$
not in
$\tilde{P}_e$.
\end{proof}

Combining
Corollary
\pref{cor:A-local}
and
\pref{lem:A-local-prerestriction},
we obtain

\begin{lemma} \label{lem:A-local-restriction}
The restriction functor
$R^A_{S^{(k)}, S^{(l)}}$
coincides with the quotient by the cocores of
$H_{S^{(k)}}$
not in
$H_{S^{(l)}}$.
\end{lemma}
\begin{proof}
The claim follows immediately from the
definition of
$R^A_{S^{(k)}, S^{(l)}}$
and
compatibility of the Liouville homotopies in the proof of
\pref{lem:A-section_unique}
with restrictions.
Here,
we give an alternative proof
which will clarify the relationship to the $B$-side counterpart. 
Again,
we restrict ourselves to the case
where
$S^{(k)} = U_v$
and
$S^{(l)} = U_e$
for some edge
$e$
connecting
$v$
to another vertex
$v^\prime$.
First,
we claim that
up to Hamiltonian isotopy the cocores of
$\tilde{P}_d$
with the canonical Weinstein structure are given by the extension of
$\Delta^i_\emptyset(l)$
by attaching suitable cylindrical ends.
Here,
$\Delta^i_\emptyset(l)$
are the stable manifolds of the points
$\delta^i_\emptyset(l)$
to
which the Liouville flow of Nadler's Weinstein structure
$\beta^i_d$
attracts.
In other words,
the cocores are the Lagrangian cocore planes for
$\delta^i_\emptyset(l)$
in the sense of
\cite[Section 2.2]{CDGG}.
By definition a cocore of
$\tilde{P}_d$
is an unstable manifold of a critical point of index
$\dim_\bR \tilde{P}_d / 2 = d$.
In particular,
it is an exact cylindrical Lagrangian
which intersects the core transversally once.
With respect to the canonical Weinstein structure
$\delta^i_\emptyset(l)$
are the critical points of index
$d$.

By
\cite[Theorem 1.1]{CDGG},
\cite[Theorem 1.13]{GPS2}
the cocores generate
$\Fuk(\tilde{P}_d)$.
Their images under the canonical equivalence
$\Fuk(\tilde{P}_d)
\to
\Fuk(H_v)$
are the cocores of
$H_v$
and
generate 
$\Fuk(H_v)$.
The canonical equivalence
$\Fuk(H_v)
\to
\Fuk(\tilde{P}_d, \beta^j_d)$
sends any of the cocores to a cocore but the image of the Lagrangian cocore plane for
$\delta^j_d(l)$,
which becomes a mere exact cylyindrical Lagrangian.
Still the images must generate 
$\Fuk(\tilde{P}_d, \beta^j_d)$.
By
\pref{lem:A-local-prerestriction}
the quotient
$\Fuk(\tilde{P}_d, \beta^j_d)
\to
\Fuk(\tilde{P}_e, \beta^j_d |_{\tilde{P}_e})$
kills the cocore of
$\tilde{P}_d$
not in
$\tilde{P}_e$,
which is the Lagrangian cocore plane for
$\delta^e_d(l)$.
Here,
$\delta^e_d(l)$
is the point
to
which the Liouville flow of Nadler's Weinstein structure
whose final leg corresponds to
$\tilde{P}_e$
attracts.
Then the images under the quotient of Lagrangian cocore planes for
$\delta^i_d(l) \neq \delta^e_d(l)$
are the cocores of
$\tilde{P}_e$
and
generate
$\Fuk(\tilde{P}_e, \beta^j_d |_{\tilde{P}_e})$.
Since the canonical equivalence
$\Fuk(\tilde{P}_e, \beta^j_d |_{\tilde{P}_e})
\to
\Fuk(H_e)$
in turn sends the images to the cocores of
$H_e$,
the restction functor
$R^A_{v, e}$
coincides with the quotient by the cocores of
$H_v$
not in
$H_e$.
\end{proof}

\subsection{$A$-side constructible sheaves of categories for very affine hypersurfaces}
Since
$U_{S^{(k)}}$
form a basis of the topology of
$\Pi_\Sigma$,
we may pass to the sheafification.

\begin{definition} \label{dfn:cF_A}
The
\emph{$A$-side constructible sheaf of categories}
for
$H$
is the sheafification
\begin{align*}
\cF_A \colon \Open(\Pi_\Sigma)^{op} \to ^{**}\DG,
\end{align*}
where
$\Open(\Pi_\Sigma)$
is the category of open subsets of
$\Pi_\Sigma$
with respect to the topology from Definition
\pref{dfn:topology}.
\end{definition}

\begin{remark}
In general,
the existence of sheafification might be delicate because of size issues.
However,
this is not the case in our setting as
$^{**}\DG$
has
small limits
and 
colimits,
and
the topology on
$\Pi_\Sigma$
has finite cardinality.
\end{remark}

We will show that
the global sections can be identified with the wrapped Fukaya category of
$H$. 
In our proof,
the following two lemmas play key roles.

\begin{lemma} \label{lem:key1}
Let
$\scrC$
be a stable presentable dg category
and
$\scrA, \scrB$
its full presentable dg subcategories
such that
\begin{align*}
\Hom_\scrC(A, B) = \Hom_\scrC(B, A) = 0
\end{align*}
for any
$A \in \scrA, B \in \scrB$.
Then there is a fiber product
\begin{align*}
\begin{gathered}
\xymatrix{
\scrC \ar[d]_{} \ar[r]^{} & \scrC / \scrB \ar[d]_{}\\
\scrC / \scrA  \ar[r]^{} & \scrC / \langle \scrA, \scrB \rangle.
}
\end{gathered}
\end{align*}
\end{lemma}
\begin{proof}
Since we have the pushouts
\begin{align*}
\begin{gathered}
\xymatrix{
\scrA \ar[d]_{} \ar[r]^{} & \scrC \ar[d]_{} & \scrB \ar[d]_{} \ar[r]^{} & \scrC \ar[d]_{} \\
0  \ar[r]^{} & \scrC / \scrA, & 0 \ar[r]^{} & \scrC / \scrB,
}
\end{gathered}
\end{align*}
the Verdier localizations
$\scrC \to \scrC / \scrA, \scrC \to \scrC / \scrB$
admit right adjoints as well as the inclusions
$\scrA \hookrightarrow \scrC, \scrB \hookrightarrow \scrC$.
Hence we obtain two semiorthogonal decompositions of
$\scrC$
by
$\scrA, \scrA^\perp$
and
by
$\scrB, \scrB^\perp$,
which respectively yield cofiber sequences
\begin{align*}
C_\scrA \to C \to C_{\scrA^\perp}, \
C_\scrB \to C \to C_{\scrB^\perp}, \
C_{\scrA \oplus \scrB} \to C \to C_{(\scrA \oplus \scrB)^\perp}
\end{align*}
for any object
$C \in \scrC$.
Note that
the full dg subcategory
$\langle \scrA, \scrB \rangle \subset \scrC$
is equivalent to
$\scrA \oplus \scrB$,
as
$\scrA, \scrB$
are mutually orthogonal.
Here,
the morphism
$C_{\scrA \oplus \scrB} \to C$
in the last cofiber sequence is the direct sum of that
$C_\scrA \to C, C_\scrB \to C$
in the first two.
Since
$\scrA, \scrB$
are mutually orthogonal,
its cone can be computed by taking the cone of
$C_\scrB \to C$
followed by taking the cone of
$C_\scrA \to C_{\scrB^\perp}$.
Hence we obtain a cofiber sequence
\begin{align*}
C_\scrA \to C_{\scrB^\perp} \to C_{(\scrA \oplus \scrB)^\perp}.
\end{align*}

The conclusion is equivalent to there being a fiber product
\begin{align} \label{eq:fiberproduct1}
\begin{gathered}
\xymatrix{
\Hom_\scrC(C_1, C_2) \ar[d]_{} \ar[r]^{} & \Hom_{\scrC / \scrB} (C_1, C_2) \ar[d]_{}\\
\Hom_{\scrC / \scrA}(C_1, C_2)  \ar[r]^{} & \Hom_{\scrC / \langle \scrA, \scrB \rangle}(C_1, C_2)
}
\end{gathered}
\end{align}
of morphism complexes for any
$C_1, C_2 \in \scrC$.
Note that
it suffices to check the latter
when
$C= C_1 = C_2$.
Since
$C_{\scrA^\perp}, C_{\scrB^\perp}, C_{(\scrA \oplus \scrB)^\perp}$
are the images of
$C$
under the right adjoints of the Verdier localizations
$\scrC \to \scrC / \scrA, \scrC \to \scrC / \scrB, \scrC \to \scrC / \langle \scrA, \scrB \rangle$,
one can rewrite
\pref{eq:fiberproduct1}
as
\begin{align*}
\begin{gathered}
\xymatrix{
\Hom_\scrC(C, C) \ar[d]_{} \ar[r]^{} & \Hom_\scrC(C, C_{\scrB^\perp}) \ar[d]_{}\\
\Hom_\scrC(C, C_{\scrA^\perp})  \ar[r]^{} & \Hom_\scrC(C, C_{(\scrA \oplus \scrB)^\perp}).
}
\end{gathered}
\end{align*}
Now,
as the functor
$\Hom_\scrC(C, -)$
preserves fiber products,
it suffices to show that
\begin{align} \label{eq:fiberproduct2}
\begin{gathered}
\xymatrix{
C \ar[d]_{} \ar[r]^{} & C_{\scrB^\perp} \ar[d]_{}\\
C_{\scrA^\perp}  \ar[r]^{} & C_{(\scrA \oplus \scrB)^\perp}
}
\end{gathered}
\end{align}
is a fiber product in
$\scrC$.
Consider the diagram
\begin{align} \label{eq:fiberproduct3}
\begin{gathered}
\xymatrix{
C_\scrA \ar[d]_{} \ar[r]^{} & C \ar[d]_{} \ar[r]^{} & C_{\scrB^\perp} \ar[d]_{} \\
0 \ar[r]_{} & C_{\scrA^\perp} \ar[r]^{} & C_{(\scrA \oplus \scrB)^\perp}.
}
\end{gathered}
\end{align}
Since both
the left
and
the outer
squares are pushouts,
the right square is also a pushout. 
It follows that
\pref{eq:fiberproduct2}
is a fiber product,
as for stable dg categories
any fiber product is a bicartesian.
\end{proof}

\begin{lemma} \label{lem:key2}
Let
$W$
be a Weinstein manifold obtained by gluing its Weinstein submanifolds
$W_1, W_2$
along their Weinstein submanifolds
$W_1 \cap W_2$.
Assume that
$W_1 \cap W_2$
separates
$W_1$
and
$W_2$
in the sense that
there is no Reeb chord connecting the ideal boundary in
$W_1 \setminus W_2$
to that in
$W_2 \setminus W_1$.
Then we have
\begin{align*}
\Hom_{\Fuk(W)}(L_1, L_2)
=
\Hom_{\Fuk(W)}(L_2, L_1)
=
0
\end{align*}
for any cocores
$L_1, L_2$
of
$W$
respectively not in
$W_2, W_1$.
\end{lemma}
\begin{proof}
Any cocores
$L_1, L_2$
respectively not in
$W_2, W_1$
do not intersect in the initial position.
By assumption the time
$1$
trajectory of 
$L_1$
under the Reeb flow never intersects
$L_2$.
\end{proof}

\begin{remark} \label{rmk:adv}
The fact that
the above separatedness assumption applies in our setting will be carefully
formulated
and
justified in Appendix.
\end{remark}

\begin{corollary} \label{cor:key}
Under the same assumption as above,
there is a fiber product
\begin{align*}
\begin{gathered}
\xymatrix{
\Fuk(W) \ar[d]_{} \ar[r]^{} & \Fuk(W_1) \ar[d]_{}\\
\Fuk(W_2)  \ar[r]^{} & \Fuk(W_1 \cap W_2).
}
\end{gathered}
\end{align*}
\end{corollary}
\begin{proof}
Since the subdomains
$[W_1], [W_2]$
and
the cobordisms
$[W] \setminus [W_1]^\circ
=
[W_2],
[W] \setminus [W_2]^\circ
=
[W_1]$
are Weinstein,
by
\cite[Proposition 11.2]{GPS2}
the Viterbo restrictions
\begin{align*}
\Fuk(W) \to \Fuk(W_1), \
\Fuk(W) \to \Fuk(W_2)
\end{align*}
are the quotients by the cocores respectively not in
$W_2, W_1$.
Similarly,
the Viterbo restrictions
\begin{align*}
\Fuk(W_1) \to \Fuk(W_1 \cap W_2), \
\Fuk(W_2) \to \Fuk(W_1 \cap W_2)
\end{align*}
are the quotients by the cocores not in
$W_1 \cap W_2$.
Now,
the claim follows from
\pref{lem:key1}
and
\pref{lem:key2}.
\end{proof}

\begin{example} \label{eg:key}
Let
$\tilde{P}^1_d \cup_{\Core(C)} \tilde{P}^2_d$
be the gluing of two $d$-dimensional tailored pants along a leg
$C$,
where
$\tilde{P}^1_d, \tilde{P}^2_d$
are equipped with Nadler's Weinstein structures.
Here,
we choose their final legs different from
$C$.
Since the subdomains
$[\tilde{P}^1_d], [\tilde{P}^2_d]$
and
the cobordisms
$[\tilde{P}^2_d], [\tilde{P}^1_d]$
are Weinstein,
by
\cite[Proposition 11.2]{GPS2}
the Viterbo restrictions
\begin{align*}
\Fuk(\tilde{P}^1_d \cup_{\Core(C)} \tilde{P}^2_d) \to \Fuk(\tilde{P}^1_d), \
\Fuk(\tilde{P}^1_d \cup_{\Core(C)} \tilde{P}^2_d) \to \Fuk(\tilde{P}^2_d)
\end{align*}
are the quotients by the cocores respectively not in
$\tilde{P}^1_d, \tilde{P}^2_d$.
Similarly,
the Viterbo restrictions
\begin{align*}
\Fuk(\tilde{P}^1_d) \to \Fuk(C), \
\Fuk(\tilde{P}^2_d) \to \Fuk(C)
\end{align*}
are the quotients by the cocores not in
$C$.

Any cocores
$L_1, L_2$
respectively not in
$W_2, W_1$
do not intersect in the initial position.
Suppose that the time
$1$
trajectory
$\phi^1_{\Ham}(L_1)$
of 
$L_1$
under the Reeb flow
$\phi_{\Ham}$
intersects
$L_2$.
Then one finds a point
$p \in L_1$
which needs to be pushed from the initial position through
$W_1 \cap W_2$
to reach
$L_2$.
Let
$0 < t < 1$
be the minimum time
such that
$\phi^t_{\Ham}(p)
\in
\del \overline{W_1 \cap W_2}$
and
$\phi^{t + \epsilon}_{\Ham}(p)
\in
W_1 \cap W_2$
for 
$0 < \epsilon \ll 1$.
Now,
$\phi^{t + \epsilon}_{\Ham}(p)$
belongs to
a cylindrical end of
$W_1 \cap W_2$.
Consider the product decompositions
\begin{align*}
C \cong \bC^*_z \times \tilde{P}_{d-1}, \
C \cong \bC^*_{\bar{z}} \times \tilde{P}_{d-1}
\end{align*}
of
$C$
respectively as a leg of
$\tilde{P}^1_d, \tilde{P}^2_d$.
The restricted Liouville flow to
$\bC^*_z, \bC^*_{\bar{z}}$
is parallel to the radial coordinate direction,
while that of the Reeb flow is orthogonal to it.
Hence all points of
$C$
remains confined under the wrapping.
\end{example}

Due to the following result,
our $A$-side category can be computed as the global sections of
$\cF_A$.
See
\pref{thm:main1-Peng}
for stronger statement with simpler proof,
which should be of independent interest.
Here,
we give an alternative proof to
introduce notations for later use
and
clarify the relationship to
\cite{GS1},
the $B$-side counterpart
and
complete intersections.

\begin{theorem} \label{thm:global}
There is an equivalence
\begin{align*}
\Fuk(H)_{\bZ_2}
\simeq
\cF_A(\Pi_\Sigma)
=
\lim
\left( \prod_{v \in \VVert(\Pi_\Sigma)} \cF_A(U_v)
\to
\prod_{e \in \Edge(\Pi_\Sigma)} \cF_A(U_e)
\to
\cdots \right).
\end{align*}
\end{theorem}
\begin{proof}
The argument is recursive
and
depends on an enumeration of the vertices of
$\Pi_\Sigma$.
We say that
the two vertices 
$v$
and
$v^\prime$
are at distance
$1$
if they are
distinct 
and
connected by a single edge.
More generally,
for any
$k \in \bN$
we say that
$v$
and
$v^\prime$
are at distance
$k$
if the following conditions hold:
\begin{itemize}
\item
There is a sequence of
$k+1$
vertices of
$\Pi_\Sigma$
\begin{align*}
v_0 = v,
\ldots,
v_k = v^\prime
\end{align*}
such that
$v_j$
and
$v_{j+1}$
are at distance
$1$
for all
$0 \leq j \leq k-1$.
\item
$k$
is the smallest integer for
which such a sequence exists.
\end{itemize}
Let
$v^0_1 \in \Pi_\Sigma$
be an external vertex of
$\Pi_\Sigma$,
i.e.,
a vertex to
which some free edge is adjacent.
Here,
by a free edge we will mean an edge not connecting any pair of distinct vertices. 
Replacing it if necessary,
we may assume that
$v^0_1 \in \Pi_\Sigma$
has the maximum number of free edges among all vertices.
We enumerate the vertices as follows.
There are in total
$l_1 < d+2$
vertices
$v^1_1, \ldots, v^1_{l_1} \in \Pi_\Sigma$
at distance
$1$
from
$v^0_1$,
i.e.,
connected to it by single edges
$e^1_{1 1}, \ldots, e^1_{1 l_1}$.
Similarly,
for any
$k \in \bN$
let
$v^k_1, \ldots, v^k_{l_k}$
be the vertices of
$\Pi_\Sigma$
at distance
$k$
from
$v^0_1$.
Each
$v^k_i$
is connected to at least one vertex
$v^{k-1}_j$
by a single edge
$e^k_{j i}$
for some
$1 \leq j \leq l_{k-1}$.

We will compute the section of
$\cF_A$
over the union
\begin{align*}
\Pi_k(v^0_1)
=
U_{v^0_1}
\cup
\bigcup^{l_1}_{i_1 = 1} U_{v^1_{i_1}}
\cup
\cdots
\cup
\bigcup^{l_k}_{i_k = 1} U_{v^k_{i_k}}.
\end{align*}
Let
$H_k(v^0_1)$
be the inverse image of suitably shrunk
$\Pi_k(v^0_1)$
under
\pref{eq:A-fibration}.
By induction on
$k$
we will show that
there is an equivalence
\begin{align*}
\Fuk(H_k(v^0_1))_{\bZ_2}
\simeq
\cF_A(\Pi_k(v^0_1))
\end{align*}
where
$H_k(v^0_1)$
is equipped with a certain Weinstein structure.
Note that
the distance between
$v^0_1$
and
any other vertex is bounded.
Thus the family of tropical submanifolds
$\Pi_k(v^0_1)$
eventually
stabilizes
and
exhausts
$\Pi_\Sigma$,
i.e.,
we have
$\Pi_\Sigma
=
\Pi_l(v^0_1)$
for some
$l \in \bN$.
The base case immediately follows from the definition.
In order to build
$H_{k+1}(v^0_1)$
from
$H_k(v^0_1)$,
we attach a single tailored pants at a time. 
The key is that
having sliced
$\Pi_\Sigma$
in these layers
$H_k(v^0_1)$,
one can engineer a compatible Weinstein structure with the attachment of the additional tailored pants.

First,
although not strictly necessary,
we run an argument
which in general works only for
$k = 1$
to illustrate the situation.
Consider Nadler's Weinstein structures transported to
$H_{v^0_1}, H_{v^1_1}$
whose final legs correspond to edges
which are different from
$e^1_{1 1}$.
Then up to permuting the final legs
$H_{v^0_1}, H_{v^1_1}$
glue along their Weinstein submanifold
$H_{e^1_{1 1}}$
to yield
a Weinstein manifold
$H_{v^0_1} \cup H_{v^1_1}$.
Note that
the pants decomposition from
\cite[Theorem 1']{Mik}
is nothing but the gluing of the closures of tailored pants along their boundaries.
By
\cite[Remark 5.2]{Mik}
this gluing is compatible with the canonical symplectic structures.
To the product of their boundaries with a sufficiently small open interval,
one can transport the Weinstein structure on
$H_{e^1_{1 1}}$
via radial deformation. 
See also
\cite[Section 6.2]{GS1}.
By
\pref{lem:A-local-restriction}
and
Corollary
\pref{cor:key}
we obtain an equivalence
\begin{align*}
\Fuk(H_{v^0_1} \cup H_{v^1_1})_{\bZ_2}
\simeq
\cF_A(U_{v^0_1} \cup U_{v^1_1})
\end{align*}
canonically induced by the commutative diagram
\begin{align*}
\begin{gathered}
\xymatrix{
\Fuk(H_{v^0_1})_{\bZ_2}  \ar[r] \ar[d]^{\simeq}
& \Fuk(H_{e^1_{11}})_{\bZ_2} \ar[d]^{\simeq}
& \Fuk(H_{v^1_1})_{\bZ_2} \ar[l] \ar[d]^{\simeq} \\
\cF_A(U_{v^0_1}) \ar[r]
& \cF_A(U_{e^1_{11}})
& \cF_A(U_{v^1_1}), \ar[l]
}
\end{gathered}
\end{align*}
where the upper horizontal arrows are Viterbo restrictions
and
the lower horizontal arrows are the restrictions of
$\cF_A$.

Consider Nadler's Weinstein structures transported to
$H_{v^0_1}, H_{v^1_1}, H_{v^1_2}$
whose final legs correspond to different edges from
$e^1_{1 1}, e^1_{1 2}$
and
possibly existing edge
$e^{11}_{12}$
connecting
$v^1_1$
to
$v^1_2$.
Note that
each of
$H_{v^0_1}, H_{v^1_1}, H_{v^1_2}$
has at least one free leg
which is not involved in this gluing. 
Then up to permuting the final legs
$H_{v^0_1}, H_{v^1_1}, H_{v^1_2}$
glue along their Weinstein submanifolds
$H_{e^1_{1 1}}, H_{e^1_{1 2}}, H_{e^{11}_{1 2}}$
to yield a Weinstein manifold
$H_{v^0_1} \cup H_{v^1_1} \cup H_{v^1_2}$.
By
\pref{lem:A-local-restriction},
Corollary
\pref{cor:key}
and
the previous step
we obtain an equivalence
\begin{align*}
\Fuk(H_{v^0_1} \cup H_{v^1_1} \cup H_{v^1_2})_{\bZ_2}
\simeq
\cF_A(U_{v^0_1} \cup U_{v^1_1} \cup U_{v^1_2})
\end{align*}
canonically induced by the commutative diagram
\begin{align*}
\begin{gathered}
\xymatrix{
\Fuk(H_{v^0_1} \cup H_{v^1_1})_{\bZ_2}  \ar[r] \ar[d]^{\simeq}
& \Fuk(H_{e^1_{12}} \cup H_{e^{11}_{12}})_{\bZ_2} \ar[d]^{\simeq}
& \Fuk(H_{v^1_2})_{\bZ_2} \ar[l] \ar[d]^{\simeq} \\
\cF_A(U_{v^0_1} \cup U_{v^1_1}) \ar[r]
& \cF_A(U_{e^1_{12}} \cup U_{e^{11}_{12}})
& \cF_A(U_{v^1_2}), \ar[l]
}
\end{gathered}
\end{align*}
where
the upper horizontal arrows are Viterbo restrictions
and
the lower horizontal arrows are the restrictions of
$\cF_A$.
Iteratively,
we obtain an equivalence
\begin{align*}
\Fuk(H_{v^0_1} \cup \bigcup^{l_1}_{i=1} H_{v^1_i})_{\bZ_2}
\simeq
\cF_A(U_{v^0_1} \cup \bigcup^{l_1}_{i=1} U_{v^1_i}).
\end{align*}

Unlike the case
$k = 1$
and
the setting of
\cite{GS1},
it is in general impossible to transport Nadler's Weinstein structures to all pieces in the fixed pants decomposition,
which glue to yield a global Weinstein structure.
However,
there exists a global Weinstein structure
which fits into Corollary
\pref{cor:key}
yielding the desired equivalence under an additional assumption on the restriction functors of
$\cF_A$.

Suppose that
there is an equivalence
\begin{align} \label{eq:hypothesis}
\Fuk(H_{v^0_1} \cup \bigcup^{l_1}_{i_1 = 1} H_{v^1_{i_1}} \cup \cdots \cup \bigcup^{l_{k-1}}_{i_{k-1} = 1} H_{v^{k-1}_{i_{k-1}}}, \lambda_{k-1})_{\bZ_2}
\simeq
\cF_A(U_{v^0_1} \cup \bigcup^{l_1}_{i_1 = 1} U_{v^1_{i_1}} \cup \cdots \cup \bigcup^{l_{k-1}}_{i_{k-1} = 1} U_{v^{k-1}_{i_{k-1}}})
\end{align}
for some Liouville structure
$\lambda_{k-1}$.
Suppose further that
via
\pref{lem:A-section_unique}
the equivalence
\pref{eq:hypothesis}
restricts to the canonical equivalence
\begin{align*}
\Fuk(H_{v^{k-1}_{i_{k-1}}}, \lambda_{k-1}(v^{k-1}_{i_{k-1}}) |_{H_{v^{k-1}_{i_{k-1}}}})_{\bZ_2}
\simeq
\cF_A(U_{v^{k-1}_{i_{k-1}}})
\end{align*}
for any Weinstein structure
$\lambda_{k-1}(v^{k-1}_{i_{k-1}})$
on
$H_{v^0_1} \cup \bigcup^{l_1}_{i_1 = 1} H_{v^1_{i_1}} \cup \cdots \cup \bigcup^{l_{k-1}}_{i_{k-1} = 1} H_{v^{k-1}_{i_{k-1}}}$
turning
$H_{v^{k-1}_{i_{k-1}}}$
into a Weinsten submanifold.
In other words,
the restriction of
$\cF_A$
along the inclusion
$U_{v^{k-1}_{i_{k-1}}}
\hookrightarrow
U_{v^0_1} \cup \bigcup^{l_1}_{i_1 = 1} U_{v^1_{i_1}} \cup \cdots \cup \bigcup^{l_{k-1}}_{i_{k-1} = 1} U_{v^{k-1}_{i_{k-1}}}$
is induced by $\bZ_2$-folding of the ind-completion of the Viterbo restriction
\begin{align*}
\Fuk(H_{v^0_1} \cup \bigcup^{l_1}_{i_1 = 1} H_{v^1_{i_1}} \cup \cdots \cup \bigcup^{l_{k-1}}_{i_{k-1} = 1} H_{v^{k-1}_{i_{k-1}}}, \lambda_{k-1}(v^{k-1}_{i_{k-1}}))
\to
\Fuk(H_{v^{k-1}_{i_{k-1}}}, \lambda_{k-1}(v^{k-1}_{i_{k-1}}) |_{H_{v^{k-1}_{i_{k-1}}}}).
\end{align*}
By definition of the restriction functors of
$\cF_A$
this is true for
$k = 0$.

When
$v^k_i$
is connected to only one vertex
$v^{k-1}_j$,
consider Nadler's Weinstein structures transported to
$H_{v^{k-1}_j}, H_{v^k_i}$
whose final legs correspond to different edges from
$e^k_{j i}$.
Then up to permuting the final legs
$H_{v^{k-1}_j}, H_{v^k_i}$
glue along their Weinstein submanifold
$H_{e^k_{j i}}$
to yield
a Weinstein manifold
$H_{v^{k-1}_j} \cup H_{v^k_i}$.
Let
$\alpha^k_{ji} \in A$
be the vertex of the triangulation
$\cT$
corresponding to a connected component
$C^k_{ji}$
of
$\bR^{d+1} \setminus \Pi_\Sigma$,
which is uniquely determined by requiring its boundary
$\del \Pi_\Sigma$
to contain
$e^k_{ji}$
but the tropical final legs of
$H_{v^{k-1}_j}, H_{v^k_i}$.
Regard
$\alpha^k_{ji}$
as a point in
$\bR^{d+1} \setminus \Pi_\Sigma$
via the fixed inner product
$T \bT^\vee \cong T^* \bT^\vee$.
Translate by
$\alpha^k_{ji}$
the canonical Weinstein structure on
$\bT^\vee_\bC$
and
then restrict to
$H_{v^0_1} \cup \bigcup^{l_1}_{i_1 = 1} H_{v^1_{i_1}} \cup \cdots \cup \bigcup^{l_{k-1}}_{i_{k-1} = 1} H_{v^{k-1}_{i_{k-1}}} \cup H_{v^k_i}$.
The result
$\lambda_k(v^{k-1}_j, v^k_i)$
is an extension of the Weinstein structure on
$H_{v^{k-1}_j} \cup H_{v^k_i}$.
In particular,
$H_{e^k_{j i}}$
is a Weinstein submanifold of
$H_{v^0_1} \cup \bigcup^{l_1}_{i_1 = 1} H_{v^1_{i_1}} \cup \cdots \cup \bigcup^{l_{k-1}}_{i_{k-1} = 1} H_{v^{k-1}_{i_{k-1}}}$.

Locally around
$\alpha^k_{ji}$
we are in the same situation as
\cite[Section 6.2]{GS1}.
In particular,
near the gluing region we are in the same situation as
\pref{eg:key}.
Hence by Corollary
\pref{cor:key}
and
the additional assumption
we obtain an equivalence
\begin{align} \label{eq:induction}
\Fuk(H_{v^0_1} \cup \cdots \cup \bigcup^{l_{k-1}}_{i_{k-1} = 1} H_{v^{k-1}_{i_{k-1}}} \cup H_{v^k_i}, \lambda_k(v^{k-1}_j, v^k_i))_{\bZ_2}
\simeq
\cF_A(U_{v^0_1} \cup \cdots \cup \bigcup^{l_{k-1}}_{i_{k-1} = 1} U_{v^{k-1}_{i_{k-1}}} \cup U_{v^k_i})
\end{align}
canonically induced by the commutative diagram
\begin{align*}
\begin{gathered}
\xymatrix{
\Fuk(H_{v^0_1} \cup \bigcup^{l_1}_{i_1 = 1} H_{v^1_{i_1}} \cup \cdots \cup \bigcup^{l_{k-1}}_{i_{k-1} = 1} H_{v^{k-1}_{i_{k-1}}}, \lambda_{k-1}(v^{k-1}_j))_{\bZ_2}  \ar[r] \ar[d]^{\simeq}
& \Fuk(H_{e^k_{j i}})_{\bZ_2} \ar[d]^{\simeq}
& \Fuk(H_{v^k_i})_{\bZ_2} \ar[l] \ar[d]^{\simeq} \\
\cF_A(U_{v^0_1} \cup \bigcup^{l_1}_{i_1 = 1} U_{v^1_{i_1}} \cup \cdots \cup \bigcup^{l_{k-1}}_{i_{k-1} = 1} U_{v^{k-1}_{i_{k-1}}}) \ar[r]
& \cF_A(U_{e^k_{j i}})
& \cF_A(U_{v^k_i}) \ar[l]
}
\end{gathered}
\end{align*}
where
the upper horizontal arrows are Viterbo restrictions
and
the lower horizontal arrows are the restrictions of
$\cF_A$.
Note that
$\lambda_{k-1}(v^{k-1}_j)$
coincides with the restriction of
$\lambda_k(v^{k-1}_j, v^k_i)$,
which defines a Weinstein structure on
$H_{v^0_1} \cup \bigcup^{l_1}_{i_1 = 1} H_{v^1_{i_1}} \cup \cdots \cup \bigcup^{l_{k-1}}_{i_{k-1} = 1} H_{v^{k-1}_{i_{k-1}}}$. 
Note also that
$\lambda_k(v^{k-1}_j, v^k_i)$
is an extension of the Weinstein structure on
$H_{v^{k-1}_j} \cup H_{v^k_i}$,
which is in turn obtained by gluing Nadler's Weinstein structures transported to
$H_{v^{k-1}_j}, H_{v^k_i}$.
Hence the left commutative diagram shares the middle vertical arrow with the right one.

It remains to check that the additional assumption on
$\cF_A$
keeps holding.
By construction the equivalence
\pref{eq:induction}
restricts to
\begin{align*}
\begin{gathered}
\Fuk(H_{v^0_1} \cup \bigcup^{l_1}_{i_1 = 1} H_{v^1_{i_1}} \cup \cdots \cup \bigcup^{l_{k-1}}_{i_{k-1} = 1} H_{v^{k-1}_{i_{k-1}}}, \lambda_{k-1}(v^{k-1}_j))_{\bZ_2}
\simeq
\cF_A(U_{v^0_1} \cup \bigcup^{l_1}_{i_1 = 1} U_{v^1_{i_1}} \cup \cdots \cup \bigcup^{l_{k-1}}_{i_{k-1} = 1} U_{v^{k-1}_{i_{k-1}}}), \\
\Fuk(H_{v^{k-1}_j}, \lambda_k(v^{k-1}_j, v^k_i) |_{H_{v^{k-1}_j}})_{\bZ_2}
\simeq
\cF_A(U_{v^{k-1}_j}), \
\Fuk(H_{v^k_i}, \lambda_k(v^{k-1}_j, v^k_i) |_{H_{v^k_i}})_{\bZ_2}
\simeq
\cF_A(U_{v^k_i})
\end{gathered}
\end{align*}
where the 
second
and
third
ones are the canonical equivalences from
\pref{lem:A-section_unique}.
Let
$\lambda^\prime_k(v^{k-1}_j, v^k_i)$
be any Weinstein structure on
$H_{v^0_1} \cup \bigcup^{l_1}_{i_1 = 1} H_{v^1_{i_1}} \cup \cdots \cup \bigcup^{l_{k-1}}_{i_{k-1} = 1} H_{v^{k-1}_{i_{k-1}}} \cup H_{v^k_i}$
turning
$H_{v^k_i}$
into a Weinsten submanifold.
Then we have the commmutative diagram
\begin{align*}
\begin{gathered}
\xymatrix{
\Fuk(H_{v^0_1} \cup \bigcup^{l_1}_{i_1 = 1} H_{v^1_{i_1}} \cup \cdots \cup \bigcup^{l_{k-1}}_{i_{k-1} = 1} H_{v^{k-1}_{i_{k-1}}} \cup H_{v^k_i}, \lambda^\prime_k(v^{k-1}_j, v^k_i))_{\bZ_2}  \ar[r] \ar[d]^{\simeq}
& \Fuk(H_{v^k_i}, \lambda^\prime_k(v^{k-1}_j, v^k_i))_{\bZ_2} \ar[d]^{\simeq} \\
\Fuk(H_{v^0_1} \cup \bigcup^{l_1}_{i_1 = 1} H_{v^1_{i_1}} \cup \cdots \cup \bigcup^{l_{k-1}}_{i_{k-1} = 1} H_{v^{k-1}_{i_{k-1}}} \cup H_{v^k_i}, \lambda_k(v^{k-1}_j, v^k_i))_{\bZ_2}\ar[r]
& \Fuk(H_{v^k_i}, \lambda_k(v^{k-1}_j, v^k_i))_{\bZ_2}
}
\end{gathered}
\end{align*}
where the vertical arrows are the canonical equivalences from
\pref{lem:A-section_unique}.
Hence via
\pref{lem:A-section_unique}
the equivalence
\pref{eq:induction}
restricts to the canonical equivalence
\begin{align*}
\Fuk(H_{v^{k-1}_{i_{k-1}}}, \lambda^\prime_k(v^{k-1}_j, v^k_i) |_{H_{v^{k-1}_{i_{k-1}}}})_{\bZ_2}
\simeq
\cF_A(U_{v^{k-1}_{i_{k-1}}}), \
\Fuk(H_{v^k_i}, \lambda^\prime_k(v^{k-1}_j, v^k_i) |_{H_{v^k_i}})_{\bZ_2}
\simeq
\cF_A(U_{v^k_i}).
\end{align*}
The same argument also works for the vertices
$v^{k-1}_{i_{k-1}}$.

One can similarly argue also when
$v^k_i$
is connected to more than one vertices
$v^{k-1}_{j_1}, \ldots, v^{k-1}_{j_l}$
by single edges.
Iteratively,
we obtain an equivalence
\begin{align*}
\Fuk(H_{v^0_1} \cup \bigcup^{l_1}_{i_1 = 1} H_{v^1_{i_1}} \cup \cdots \cup \bigcup^{l_k}_{i_k = 1} H_{v^k_{i_k}})_{\bZ_2}
\simeq
\cF_A(U_{v^0_1} \cup \bigcup^{l_1}_{i_1 = 1} U_{v^1_{i_1}} \cup \cdots \cup \bigcup^{l_k}_{i_k = 1} U_{v^k_{i_k}}).
\end{align*}
\end{proof}

\subsection{$B$-side constructible sheaves of categories for very affine hypersurfaces}
Recall that
$\Pi_\Sigma$
is the dual cell complex of
$\cT$,
which we assume to be an adapted unimodular triangulation of the convex lattice polytope
$\Delta^\vee \subset M^\vee_\bR$.
In particular,
each vertex
$v \in \Pi_\Sigma$
bijectively corresponds to a cell congruent to a standard simplex under the $\GL(d+1, \bZ)$-action.
The cell in turn bijectively corresponds to a cone
$\sigma_v \in \Sigma_Y$,
which defines an affine open subvariety
$Y_v = \Spec \bC[\sigma^\vee_v \cap (M \times \bZ)] \subset Y$
isomorphic to
$\bA^{d+2}$.
Introduce coordinates
$(y_{e^v_1}, \ldots, y_{e^v_{d+2}})$
on
$Y_v$,
where
$e^v_i$
stand for edges adjacent to
$v$
and
dual to facets
$\sigma_{e^v_i}$
of
$\sigma_v$
so that
$Y_{e^v_i}
=
\Spec \bC[\sigma^\vee_{e^v_i} \cap (M \times \bZ)]
\subset
Y_v$
are the open subvarieties defined by
$y_{e^v_i} \neq 0$.

Recall from
\cite[Proposition A.3.1]{Pre}
that
$\MF^\infty$
is a sheaf on
$Y$. 
As we explained, we have a dictionary relating open subsets of $\Pi_\Sigma$ and open subsets of $Y$: in symbols, this gives a functor 
$$
\Open(\Pi_\Sigma) \to \Open(Y)
$$

\begin{definition}
The
\emph{$B$-side constructible sheaf of categories}
for
$H$
is the sheaf on $\Pi_\Sigma$ given by the composite
\begin{align*}
\cF_B \colon \Open(\Pi_\Sigma)^{op} \longrightarrow \Open(Y)^{op} \, \stackrel{\MF^\infty} \longrightarrow   \,  ^{**}\DG.
\end{align*}
\end{definition}

For concreteness,
let us describe explicitly the 
sections
and
restrictions
of
$\cF_B$
on the basis
$U_{S^{(k)}}$: 
\begin{itemize}
\item
The section over
$U_{S^{(k)}}$
is   the ind-completion of the category of matrix factorizations
\begin{align*}
\cF_B(U_{S^{(k)}})
=
\MF^\infty(Y_v |_{\bigcap^k_{m = 1} \{ y_{e^v_{i_m}} \neq 0 \}}, y_{e^v_1} \cdots y_{e^v_{d+2}}),
\end{align*}
where
$e^v_{i_1}, \ldots, e^v_{i_k}$
are the edges adjacent to
$v$
determining the $k$-stratum
$S^{(k)}$.
\item
Along an inclusion
$U_{S^{(l)}} \hookrightarrow U_{S^{(k)}}$
the restriction functor
\begin{align*}
R^B_{S^{(k)}, S^{(l)}}
\colon
\MF^\infty(Y_v |_{\bigcap^k_{m = 1} \{ y_{e^v_{i_m}}\neq 0 \}}, y_{e^v_1} \cdots y_{e^v_{d+2}})
\to
\MF^\infty(Y_v |_{\bigcap^l_{n = 1} \{ y_{e^v_{j_n}}\neq 0 \}}, y_{e^v_1} \cdots y_{e^v_{d+2}})
\end{align*}
is given by the canonical restriction functor,
where
$e^v_{i_1}, \ldots, e^v_{i_k}$
and
$e^v_{j_1}, \ldots, e^v_{j_l}$
are edges adjacent to
$v$
respectively determining
$S^{(k)}$
and
$S^{(l)}$.
\end{itemize}

\section{Isomorphism of the constructible sheaves}
In this section,
we give a proof of HMS for very affine hypersurfaces by gluing HMS for pairs of pants established in
\cite{Nad}.
When gluing such equivalences,
the combinatorial duality over
$\Pi_\Sigma$
from Section
$3$
plays a crucial role.

\subsection{Local equivalences}
Passing to the category of matrix factorizations might be delicate because of Kn\"{o}rrer periodicity.
First,
we show how to lift Nadler's equivalence,
i.e.,
the equivalence from Corollary
\pref{cor:Nadler}
to the category of matrix factorizations preserving the compatibility with restrictions.

\begin{lemma} \label{lem:clsk}
Let
$[H_U]$
be a Weinstein subdomain of
$[H]$.
Then the skeleton
$\Core(H_U)$
of
$H_U$
is a closed subset of the skeleton
$\Core(H)$
of
$H$.
\end{lemma}
\begin{proof}
We may assume that
$\Core(H_U)$
is disjoint from
$\del [H_U] = \del_\infty H_U$.
Indeed,
skeleta of Liouville manifolds of finite type are maximal compact subsets conic with respect to Liouville vector fields.
Since any point in
$\del_\infty H_U$
escapes to infinity along the Liouville flow,
it never contributes to
$\Core(H_U)$.
We may assume further that
$\Core(H)$
is contained in
$[H]$.
Then
$\Core(H) \cap ([H] \setminus [H_U])$
is an open subset of
$\Core(H)$.
Since,
up to deformation,
$\Core(H_U)$
coincides with the complement of
$\Core(H) \cap ([H] \setminus [H_U])$
in
$\Core(H)$,
it is a closed subset of
$\Core(H)$.
\end{proof}

\begin{lemma} \label{lem:lift1}
For each vertex
$v \in \Pi_\Sigma$
there is an equivalence
\begin{align*}
\Fuk(H_v)
\simeq
\IndCoh(Y_{v, d}), \
Y_{v, d} 
=
\{ y_{e^v_1} \cdots y_{e^v_{d+1}} = 0 \}
\subset
\bA^{d+1}
\end{align*}
compatible with restrictions.
In particular,
for every open subset
$U_{S^{(k)}} \subset U_v$
determined by edges
$e^v_{i_1}, \ldots, e^v_{i_k}$
adjacent to
$v$
but different from
$e^v_{d+2}$,
we have the commutative diagram
\begin{align*}
\begin{gathered}
\xymatrix{
\Fuk(H_v) \ar[r] \ar_{\simeq}[d] & \Fuk(H_{S^{(k)}}) \ar^{\simeq}[d] \\
\IndCoh(Y_{v, d}) \ar[r] & \IndCoh(Y_{v, d} |_{\bigcap^k_{j = 1} \{ y_{e^v_{i_j}} \neq 0 \} }),
}
\end{gathered}
\end{align*}
where 
the upper right horizontal arrow is the ind-completion of the Viterbo restriction
and
the lower horizontal arrow is the canonical restriction functor.
\end{lemma}
\begin{proof}
From the argument in the proof of
\pref{thm:mSh-Coh}
we obtain a commutative diagram
\begin{align} \label{eq:connector1}
\begin{gathered}
\xymatrix{
\Coh(\bigcup^k_{j=1} \bA^{ \{ i_j \}^c}) \ar[r]^-{} \ar^{\simeq}[d]
&
\Coh(Y_{v, d}) \ar[r]^-{} \ar^{\simeq}[d]
&
\Coh(Y_{v, d} |_{\bigcap^k_{j = 1} \{ y_{e^v_{i_j}} \neq 0 \} }) \ar[r]^-{} \ar^{\simeq}[d]
&
0 \\
\mSh(\bigcup^k_{j=1} \Lambda^\infty_{ \{ i_j \}^c}) \ar[r]^-{} 
&
\mSh(\Lambda^\infty_{d+1}) \ar[r]_-{}
&
\mSh(\Lambda^\infty_{d+1}) / \mSh(\bigcup^k_{j=1} \Lambda^\infty_{ \{ i_j \}^c}) \ar[r]^-{}
&
0
}
\end{gathered}
\end{align}
where the horizontal arrows form cofiber sequences.
Here,
the upper left horizontal arrow is the canonical functor to the colimit
and
the lower left horizontal arrow is left adjoint to the restriction. 

Recall that
$H_{S^{(k)}}$
is symplectomorphic to the intersection of legs
$L_{d, i_1}(K), \ldots, L_{d, i_k}(K) \subset \tilde{P}_d$.
Consider Nadler's Weinstein structure transported to
$H_{S^{(k)}}$
whose final leg corresponds to
$e^v_{d + 2}$.
We write 
$\Lambda_{S^{(k)}}$
for the isomorphic image of
$\Core(H_{S^{(k)}}) \times \{ 0 \} \times \bR$
under the symplectomorphism
$U_d \times T^1 \times \bR
\hookrightarrow
\Omega_{d+1}$
from
\pref{lem:Symp-Symp2}.
Let
$\tilde{P}_I \subset T^* T^I$
be the tailored pants for
$I = \{ i_1, \ldots, i_k \}$.
We write
$\frakL_I$
for the isomorphic image of
$\Core(\tilde{P}_I)$
under the symplectomorphism from
\pref{lem:Symp-Symp2}.
Unwinding the proof of
\pref{lem:Symp-Symp1}
and
\pref{lem:Symp-Symp2},
one sees that
\begin{align*}
\Core(H_{S^{(k)}})
\cong
T^{I^c} \times \Core(\tilde{P}_I)
\cong
\bT^{I^c} \times \frakL_I
\cong
T^{I^c} \times s(\Gamma_{\frakL_I, -f |_{\frakL_I}})
\cong
T^{I^c} \times \Lambda^\infty_I
\subset
\Lambda^\infty_{d+1}
\end{align*}
and
$\Lambda^\infty_{S^{(k)}} \subset \Lambda^\infty_{d+1}$
is defined as
$\bigcap^{k}_{j = 1} \{ \xi_{i_j} = 0 \}$.
Hence we obtain
\begin{align*}
\bigcup^k_{j=1} \Lambda^\infty_{ \{ i_j \}^c}
=
\Lambda^\infty_{d+1} \setminus \Lambda^\infty_{S^{(k)}}, \
\mSh(\bigcup^k_{j=1} \Lambda^\infty_{ \{ i_j \}^c})
=
\mSh(\Lambda^\infty_{d+1} \setminus \Lambda^\infty_{S^{(k)}}).
\end{align*}

From the argument in the end of
\cite[Section 4]{GS1}
we obtain another commutative diagram
\begin{align} \label{eq:connector2}
\begin{gathered}
\xymatrix{
\mSh^\diamondsuit(\Lambda^\infty_{d+1} \setminus \Lambda^\infty_{S^{(k)}}) \ar[r]^-{} \ar@{=}[d]
&
\mSh^\diamondsuit(\Lambda^\infty_{d+1}) \ar[r]_-{} \ar@{=}[d]
&
\mSh^\diamondsuit(\Lambda^\infty_{S^{(k)}}) \ar[r]^-{} \ar@{=}[d]
&
0 \\
\langle \text{Cocores of} \ [H_v] \setminus [H_{S^{(k)}}] \rangle \ar[r]^-{}
&
\Fuk(H_v) \ar[r]_-{}
&
\Fuk(H_{S^{(k)}}) \ar[r]^-{} 
&
0
}
\end{gathered}
\end{align}
where
the horizontal arrows form exact sequences
and
the middle and right vertical arrows are the equivalences from
\cite[Theorem 1.4]{GPS3}.
Here,
the upper left horizontal arrow is left adjoint to the restriction to
$\Lambda^\infty_{d+1} \setminus \Lambda^\infty_{S^{(k)}}$
which is open by
\pref{lem:clsk},
and
the lower right horizontal arrow is the ind-completion of the Viterbo restriction.
Hence we may concatenate
\pref{eq:connector1}
and
\pref{eq:connector2}
to obtain a commutative diagram
\begin{align} \label{eq:connector3}
\begin{gathered}
\xymatrix{
\IndCoh(\bigcup^k_{j=1} \bA^{ \{ i_j \}^c})\ar[r]^-{} \ar^{\simeq}[d]
&
\IndCoh(Y_{v, d}) \ar[r]^-{} \ar^{\simeq}[d]
&
\IndCoh(Y_{v, d} |_{\bigcap^k_{j = 1} \{ y_{e^v_{i_j}} \neq 0 \} } ) \ar[r]^-{} \ar^{\simeq}[d]
&
0 \\
\langle \text{Cocores of} \ [H_v] \setminus [H_{S^{(k)}}] \rangle \ar[r]^-{}
&
\Fuk(H_v) \ar[r]_-{}
&
\Fuk(H_{S^{(k)}}) \ar[r]^-{} 
&
0.
}
\end{gathered}
\end{align}
\end{proof}

\begin{remark}
Recall that
Nadler broke the symmetry of
$\tilde{P}_d$
so that
the final leg attracts the Liouville flow
while leaving the other legs symmetric.
In particular,
each of the other legs defines a Weinstein submanifold,
which is isomorphic to a product of
$\bC^*$
and
a tailored pants
which is one dimension lower.
Via the combinatorial duality incorporated into the definition of
$\cF_B$,
the Viterbo restriction to such a Weinstein submanifold
associated with the $j$-th leg 
corresponds to
the restriction to the open subset defined by
$y_{e^v_j} \neq 0$
for
$j = 1, \ldots, d+1$. 
\end{remark}

Consider the natural projection
\begin{align*}
Y_{v, d+1} = \{ y_{e^v_1} \cdots y_{e^v_{d+1}} y_{e^v_{d+2}} = 0 \}
\subset
Y_v
\cong
\bA^{d+2}
\to
Y^i_{v, d+1} = \{ y_{e^v_1} \cdots \widehat{y_{e^v_i}} \cdots y_{e^v_{d+2}} = 0 \}
\subset
\bA^{d+1}
\end{align*}
of the union of the coordinate hyperplanes.
Up to canonical equivalence defined by cyclic permutation of the coordinates,
by
\pref{lem:Coh-MF}
the pullback induces an equivalence
\begin{align} \label{eq:Orlov}
\IndCoh(Y^i_{v, d+1})_{\bZ_2}
\simeq
\MF^\infty(\bA^{d+2}, y_{e^v_1} \cdots y_{e^v_{d+1}} y_{e^v_{d+2}}).
\end{align}

\begin{theorem} \label{thm:lift2}
For each vertex
$v \in \Pi_\Sigma$
there is an equivalence
\begin{align*}
\cF_A(U_v) \simeq \cF_B(U_v)
\end{align*}
compatible with restrictions.
In particular,
for every open subset
$U_{S^{(k)}} \subset U_v$
determined by edges
$e^v_{i_1}, \ldots, e^v_{i_k}$
adjacent to
$v$,
we have the commutative diagram
\begin{align*}
\begin{gathered}
\xymatrix{
\cF_A(U_v) \ar[r]^-{R^A_{v, S^{(k)}}} \ar_{\simeq}[d] & \cF_A(U_{S^{(k)}}) \ar^{\simeq}[d] \\
\cF_B(U_v) \ar[r]_-{R^B_{v, S^{(k)}}} & \cF_B(U_{S^{(k)}}).
}
\end{gathered}
\end{align*}
\end{theorem}
\begin{proof}
On each vertex
$v \in \Pi_\Sigma$
we have the extended Nadler's equivalence
\begin{align*}
\varphi^v_{d+2}
\colon
\Fuk(H_v)_{\bZ_2}
\to
\IndCoh(Y^{d+2}_{v, d+1})_{\bZ_2}
\to
\MF^\infty(\bA^{d+2}, y_{e^v_1} \cdots y_{e^v_{d+1}} y_{e^v_{d+2}}),
\end{align*}
which is the composition of $\bZ_2$-folding of the equivalence from
\pref{lem:lift1}
with that
\pref{eq:Orlov}
for
$i = d+1$.
Composing with the canonical equivalence from Corollary
\pref{cor:A-local},
we obtain
\begin{align*}
\varphi_v
\colon
\cF_A(U_v)
\to
\MF^\infty(\bA^{d+2}, y_{e^v_1} \cdots y_{e^v_{d+1}} y_{e^v_{d+2}}).
\end{align*}
Recall that
$\MF^\infty(\bA^{d+2}, y_{e^v_1} \cdots y_{e^v_{d+1}} y_{e^v_{d+2}})$
is split-generated by
$\{ \underline{\scrO}^i_{d+1} \}^{d+2}_{i=1}$
and
$\underline{\scrO}^{d+2}_{d+1}$
belongs to the stable envelope of the others.
Recall also that
the equivalence from Corollary
\pref{cor:A-local}
preserves the Lagrangian cocore planes of
$H_v$
which generates
$\Fuk(H_v)$.
Since
$\varphi^v_{d+2}$
is determined by the images of these Lagrangian cocore planes,
the composition
$\varphi_v$
is a canonical lift of
$\varphi^v_{d+2}$.
By
\pref{lem:A-local-restriction}
the restriction
$R^A_{v, S^{(k)}}$
is the quotient by cocores not in
$H_{S^{(k)}}$.
Via the combinatorial duality incorporated into the definition of
$\cF_B$,
such cocores correspond to some of
$\{ \underline{\scrO}^i_{d+1} \}^{d+2}_{i=1}$
which are killed by
$R^B_{v, S^{(k)}}$.
Hence
$\varphi_v$
is compatible with restrictions.
\end{proof}

\subsection{Gluing equivalences}
Finally,
we glue the local equivalence from
\pref{thm:lift2}
on each vertex
$v \in \Pi_\Sigma$
to obtain a global equivalence compatible with restrictions.

\begin{theorem} \label{thm:main1}
There is an equivalence
\begin{align*}
\cF_A(\Pi_\Sigma) \simeq \cF_B(\Pi_\Sigma)
\end{align*}
compatible with restrictions.
\end{theorem}
\begin{proof}
We will construct an equivalence of the sections of
$\cF_A$
and
$\cF_B$
over
$\Pi_k(v^0_1)$
compatible with restrictions.
The base case is
\pref{thm:lift2}.
Suppose that
there is an equivalence
\begin{align*}
\varphi_{k-1}
\colon
\cF_A(U_{v^0_1} \cup \bigcup^{l_1}_{i_1 = 1} U_{v^1_{i_1}} \cup \cdots \cup \bigcup^{l_{k-1}}_{i_{k-1} = 1} U_{v^{k-1}_{i_{k-1}}})
\to
\cF_B(U_{v^0_1} \cup \bigcup^{l_1}_{i_1 = 1} U_{v^1_{i_1}} \cup \cdots \cup \bigcup^{l_{k-1}}_{i_{k-1} = 1} U_{v^{k-1}_{i_{k-1}}})
\end{align*}
compatible with restrictions
and
that restricts to
$\varphi_v$
on each vertex
$v \in \Pi_{k-1}(v^0_1)$.
When
$v^k_i$
is connected to only one vertex
$v^{k-1}_j$,
by
assumption
and
\pref{thm:lift2}
we have the commutative diagram
\begin{align*}
\begin{gathered}
\xymatrix{
\cF_A(U_{v^0_1} \cup \bigcup^{l_1}_{i_1 = 1} U_{v^1_{i_1}} \cup \cdots \cup \bigcup^{l_{k-1}}_{i_{k-1} = 1} U_{v^{k-1}_{i_{k-1}}}) \ar[r]^-{\varphi_{k-1}} \ar^{R^A_{k-1, e^k_{ji}}}[d]
&
\cF_B(U_{v^0_1} \cup \bigcup^{l_1}_{i_1 = 1} U_{v^1_{i_1}} \cup \cdots \cup \bigcup^{l_{k-1}}_{i_{k-1} = 1} U_{v^{k-1}_{i_{k-1}}}) \ar^{R^B_{k-1, e^k_{ji}}}[d] \\
\cF_A (U_{e^k_{ji}}) \ar[r]^-{\varphi_{k-1} |_{e^k_{ji}} = \varphi_{v^{k-1}_j} |_{e^k_{ji}}} \ar[d]^-{\simeq}
&
\cF_B (U_{e^k_{ji}}) \ar[d]^-{\simeq} \\
\cF_A (U_{e^k_{ji}}) \ar[r]^-{\varphi_{v^k_i} |_{e^k_{ji}}}
&
\cF_B (U_{e^k_{ji}}) \\
\cF_A (U_{v^k_i}) \ar[r]^-{\varphi_{v^k_i}} \ar_{R^A_{v^k_i, e^k_{ji}}}[u]
&
\cF_B (U_{v^k_i}) \ar_{R^B_{v^k_i, e^k_{ji}}}[u]
}
\end{gathered}
\end{align*}
where the middle vertical arrows are the canonical equivalences induced by the gluing of coordinates.
Note that
one can take a collection of isomorphisms from
$\tilde{P}_d$
to
$H_v$
for all vertices
$v \in \Pi_\Sigma$
compatible with the fixed pants decomposition.
Via the combinatorial duality incorporated into the definition of
$\cF_B$,
it corresponds to a collection of isomorphisms from
$\bA^{d+2}$
to
$\overline{Y}_v$
for all vertices
$v \in \Pi_\Sigma$
compatible with the open cover of
$Y$.
Hence
$\varphi_{k-1}, \varphi_{v^k_i}$
glue to yield an equivalence
\begin{align*}
\cF_A(U_{v^0_1} \cup \bigcup^{l_1}_{i_1 = 1} U_{v^1_{i_1}} \cup \cdots \cup \bigcup^{l_{k-1}}_{i_{k-1} = 1} U_{v^{k-1}_{i_{k-1}}} \cup U_{v^k_i})
\to
\cF_B(U_{v^0_1} \cup \bigcup^{l_1}_{i_1 = 1} U_{v^1_{i_1}} \cup \cdots \cup \bigcup^{l_{k-1}}_{i_{k-1} = 1} U_{v^{k-1}_{i_{k-1}}} \cup U_{v^k_i})
\end{align*}
compatible with restrictions
and
that restricts to
$\varphi_v$
on each vertex
$v \in \Pi_{k-1}(v^0_1)$
and
$v^k_i$.
The same argument also works when
$v^k_i$
is connected to more than one vertices
$v^{k-1}_{j_1}, \ldots, v^{k-1}_{j_l}$
by single edges.
Iteratively,
we obtain an equivalence
\begin{align*}
\varphi_k
\colon
\cF_A(U_{v^0_1} \cup \bigcup^{l_1}_{i_1 = 1} U_{v^1_{i_1}} \cup \cdots \cup \bigcup^{l_k}_{i_k = 1} U_{v^k_{i_k}})
\to
\cF_B(U_{v^0_1} \cup \bigcup^{l_1}_{i_1 = 1} U_{v^1_{i_1}} \cup \cdots \cup \bigcup^{l_k}_{i_k = 1} U_{v^k_{i_k}})
\end{align*}
compatible with restrictions
and
that restricts to
$\varphi_v$
on each vertex
$v \in \Pi_k(v^0_1)$.
\end{proof}

\begin{remark}
The equivalence
$\varphi_1$
can be constructed by gluing Nadler's equivalences without lifts to sections of
$\cF_A$.
As explained in Section
$9$,
in special settings such as
\cite{GS1},
the same argument works to yield a global equivalence.
However,
such gluing requires us to choose
$i$
in
\pref{eq:Orlov}
for all vertices of
$\Pi_\Sigma$
and 
it is in general impossible to find a collection of compatible choices.
\end{remark}

\begin{corollary}
There is an equivalence
\begin{align*}
\Fuk(H)_{\bZ_2}
\simeq
\MF^\infty(Y, W_Y).
\end{align*}
\end{corollary}

\section{Critical loci of Landau--Ginzburg models for complete intersections}
In this section,
following
\cite[Section 10]{AAK},
we realize the mirror pair for a complete intersection of very affine hypersufaces as critical loci of associated Landau--Ginzburg models.
They give rise to fibrations over the complete intersection of the tropical hypersurfaces
equipped with the canonical stratification.

\subsection{Landau--Ginzburg $A$-model for complete intersections}
Let
$H_1, \ldots, H_r \subset \bT^\vee_\bC$
be very affine hypersurfaces in general position defined by Laurent polynomials
\begin{align} \label{eq:Laurents} 
W_i \colon \bT^\vee_\bC \to \bC, \
x
\mapsto
\sum_{\alpha^i \in A_i} c_{\alpha^i} t^{-\rho_i(\alpha^i)}x^{\alpha^i}, \
i = 1, \ldots, r.
\end{align}
Here,
\begin{itemize}
\item
$c_{\alpha^i} \in \bC^*$
are arbitrary constants,
\item
$t$
is a sufficiently large tropical parameter,
\item
$\rho_i$
are convex piecewise linear functions on convex lattice polytopes
$\Delta^\vee_i$
whose corner loci give adapted unimodular triangulations
$\cT_i$
of
$\Delta^\vee_i$, 
and
\item
$A_i \subset M^\vee$
are the set of vertices of
$\cT_i$.
\end{itemize}
Throughout the paper,
we assume 
$H_1, \ldots, H_r$
to be in sufficiently general position for their tropical hypersurfaces
$\Pi_{\Sigma_1}, \ldots, \Pi_{\Sigma_r}$
to intersect transversely. 
We denote by
$\bfH$
the complete intersection
$H_1 \cap \cdots \cap H_r \subset \bT^\vee_\bC$.
For
$\bfX = \bT^\vee_\bC \times \bC^r$
with coordinates
$(x, u) = (x_1, \ldots, x_{d+1}, u_1, \ldots, u_r)$,
consider a Laurent polynomial
\begin{align*}
W_\bfX \colon \bfX \to \bC, \
(x, u)
\mapsto
\sum^{r}_{i = 1} u_i W_i(x).
\end{align*}

\begin{definition}
Let
$\bfH \subset \bT^\vee_\bC$
be the complete intersection of the very affine hypersurfaces
$H_1, \ldots, H_r$
defined by the Laurent polynomials
$W_1, \ldots, W_r$
from
\pref{eq:Laurents}.
We call the pair
$(\bfX, W_\bfX)$
the
\emph{Landau--Ginzburg $A$-model}
for
$\bfH$.
\end{definition}

\begin{definition}
The
\emph{Newton polytope}
$\Delta^\vee_\bfX$
of
$W_\bfX$
is the convex hull
\begin{align*}
\Conv(0, -\Delta^\vee_1 \times e_1, \ldots, -\Delta^\vee_r \times e_r)
\subset
M^\vee_\bR \times \bR^r
\end{align*}
where
$e_1, \ldots, e_r \in \bR^r$
are the standard basis vectors.
\end{definition}

\begin{remark} \label{rmk:triangulation}
The polytope
$\Delta^\vee_\bfX$
admits an adapted star-shaped triangulation
$\bf{T}$
canonically induced by
$\rho_1, \ldots, \rho_r$.
However,
it might not be unimodular.
\end{remark}

\begin{lemma} \label{lem:cA-critical}
The critical locus
$\Crit(W_\bfX)$
is given by
$\bigcap^r_{i=1} \{ u_i = 0 \}
\cap
\bigcap^r_{i=1}\{ W_i = 0 \}
\subset
\bfX$.
\end{lemma}
\begin{proof}
Express the tangent map
$d W_\bfX$
of
$W_\bfX$
as a $(1, 2r)$-matrix
\begin{align*}
(u_1 dW_1, \ldots, u_r dW_r, W_1, \ldots, W_r).
\end{align*}
Since by assumption
$H_1, \ldots, H_r$
intersect transversely,
$dW_i$
nowhere vanish.
Hence
$\rank (dW_\bfX) = 0$
if and only if 
$u_1 = \cdots = u_r = 0$
and
$W_1 = \cdots = W_r = 0$.
\end{proof}

\begin{remark} \label{rmk:cA-fibration}
By
\pref{lem:cA-critical}
the projection
$\pr_1 \colon \bfX = \bT^\vee_\bC \times \bC^r \to \bT^\vee_\bC$
preserves
$\Crit(W_\bfX)$
and
the inclusions
$\bfH \subset H_i \hookrightarrow \Crit(W_\bfX)$.
Since by assumption
$\Pi_{\Sigma_1}, \ldots, \Pi_{\Sigma_r}$
intersect transversely,
we may assume further that
$H_1, \ldots, H_r$
intersect along their legs. 
Let
$\ret_i \colon \Pi_i \to \Pi_{\Sigma_i}$
be the continuous maps induced by the retractions.
Then the composition
\begin{align} \label{eq:cA-fibration}
\bff
\colon
\bfH
\cong
\Crit(W_\bfX)
\hookrightarrow
\bfX
\xrightarrow{\Log_{d+1} \circ \pr_1}
\bigcap^r_{i=1} \Pi_i
\xrightarrow{\ret_r \circ \cdots \circ \ret_1}
\bigcap^r_{i=1} \Pi_{\Sigma_i}
\end{align}
gives a stratified fibration.
Away from lower dimensional strata,
the fiber over a point in the intersection of
$r$
top dimensional strata one from each
$\Pi_{\Sigma_i}$
is a real $(d - r +1)$-torus.
\end{remark}

\subsection{Landau--Ginzburg $B$-model for complete intersections}
Recall that
$\cT_i$
are the chosen adapted unimodular triangulations of
$\Delta^\vee_i \subset M^\vee_\bR$
obtained as the corner loci of the convex piecewise linear functions
$\rho_i \colon \Delta^\vee_i \to \bR$.
Recall also from Remark
\pref{rmk:triangulation}
that
$\bfT$
is the adapted star-shaped triangulation of
$\Delta^\vee_\bfX \subset M^\vee_\bR \times \bR^r$
canonically induced by
$\rho_1, \ldots, \rho_r$.
Let
$\Sigma_\bfY \subset M^\vee_\bR \times \bR^r$
be the fan corresponding to
$\bfT$
and
$\bfY$
the noncompact $(d+r+1)$-dimensional toric variety associated with
$\Sigma_\bfY$.
The primitive ray generators of
$\Sigma_\bfY$
are the vectors of the form
$(- \alpha^{i,j}, e_i)$
with
$\alpha^{i,j} \in A_i$.
Such vectors span a smooth cone of
$\Sigma_\bfY$
if
$-\alpha^{i,j}$
span a cell of
$\cT_i$
for fixed
$i$.

\begin{remark}
Unlike the case
$r = 1$,
there might be nonsmooth cones of
$\Sigma_\bfY$
as
$\bfT$
is not necessarily unimodular.
Indeed,
consider the case
where
$d = 2, r=2$
and
the two defining polynomials are
$x_1 + x_2 + x_3, x^2_1 x_2 + x_3$.
Then the Newton polytope
$\Delta^\vee_\bfX$
is a $5$-dimensional simplex
which has twice volume of the unit simplex.
As
$x_1 + x_2 + x_3, x^2_1 x_2 + x_3$
cannot be further divided,
there is no room for subdivision of
$\Delta^\vee_\bfX$.
Possible nonsmooth cones would contain at least two rays belonging to distinct subfans of the form
\begin{align*}
\Sigma_i
=
\bR_{\geq 0} \cdot (-\cT_i \times \{ e_i \})
\subset
\Sigma_\bfY.
\end{align*}
\end{remark}

Dually,
$\bfY$
is associated with the noncompact moment polytope
\begin{align*}
\Delta_\bfY
=
\{ (m, u_1, \ldots, u_r)
\in
M_\bR \times \bR^r \ | \ u_i \geq \varphi_i (m), \ 1 \leq i \leq r \}.
\end{align*}
The facets of
$\Delta_\bfY$
correspond to the maximal domains of linearity of
$\varphi_1, \ldots, \varphi_r$.
We denote by
$\bfA$
the set of connected components of
$M_\bR \setminus \bigcup^r_{i=1} \Pi_{\Sigma_i}$
and
index each component by the tuple
\begin{align*}
\vec{\alpha}
=
(\alpha^{1,j_1}, \ldots, \alpha^{r, j_r})
\in
M^\vee \times \cdots \times M^\vee
\end{align*}
of vertices.

\begin{remark}
The noncompact polytope
$\Delta_\bfY$
is homeomorphic to the image of
$\bfY$
under the composition
\begin{align} \label{eq:cmoment}
\bfY \to (\bfY)_{\geq 0} \to M_\bR \times \bR^r
\end{align}
of
the map induced by retraction to the nonnegative real points
with
the restriction of the negated algebraic moment map
\cite[Section 12.2]{CLS}.
\end{remark}

\begin{lemma} \label{lem:cprojection}
Let
$\bfq \colon M_\bR \times \bR^r \to M_\bR$
be the natural projection.
Then under
$\bfq$
the union of intersections of
$r$
facets of
$\Delta_\bfY$
one from each
$\{ u_i \geq \varphi_i (m) \}$
homeomorphically maps to
$M_\bR$.
Moreover,
the union of intersections of
$r$
codimension
$2$
faces of
$\Delta_\bfY$
one from each
$\{ u_i \geq \varphi_i (m) \}$
homeomorphically maps to
$\bigcap^r_{i=1} \Pi_{\Sigma_i}$.
\end{lemma}
\begin{proof}
By construction of
$\Sigma_\bfY$
under
$\bfq$
the intersection of
$r$
facets of
$\Delta_\bfY$
one from each
$\{ u_i \geq \varphi_i (m) \}$
homeomorphically maps to the intersection of
$r$
maximal domains of linearity
one from each
$\varphi_i$
corresponding to the same
$\alpha^{i,j_i} \in A_i$.
When
$\vec{\alpha}
=
(\alpha^{1,j_1}, \ldots, \alpha^{r, j_r})$
runs through
$\bfA$,
the closure of the latter covers
$M_\bR$.
Then the second statement follows from the same argument as in the proof of
\pref{lem:projection}.
\end{proof}

For each
$\vec{\alpha} \in \bfA$
let
$\bfY_{\vec{\alpha}}
=
\bT_\bC \times \bC^r$
with coordinates
$y_{\vec{\alpha}}
=
(y_{\vec{\alpha}_,1}, \ldots, y_{\vec{\alpha}, d+1}, v_{\vec{\alpha}, 1}, \ldots, v_{\vec{\alpha}, r})$,
where
$y_{\vec{\alpha}_,1}, \ldots, y_{\vec{\alpha}, d+1}$
are the monomials with weights
\begin{align*}
\eta_1 = (-1, \ldots, 0, -\alpha^{1, j_1}_1, \ldots, -\alpha^{r, j_r}_1),
\ldots,
\eta_{d+1} = (0, \ldots, -1, -\alpha^{1, j_1}_{d+1}, \ldots, -\alpha^{r, j_r}_{d+1})
\in
M \times \bZ^r
\end{align*}
and
$v_{\vec{\alpha}, 1}, \ldots, v_{\vec{\alpha}, r}$
are the monomials with weights
\begin{align*}
\eta_{d+2} = (0, \ldots, 0, 1, \ldots, 0),
\ldots,
\eta_{d+r+1} = (0, \ldots, 0, 0, \ldots, 1)
\in
M \times \bZ^r.
\end{align*}
Pairing of the former monomials with the monomials with weight
\begin{align*}
u_{\xi_1} = (-\alpha^{1, j_1}, e_1),
\ldots,
u_{\xi_r} =(-\alpha^{r, j_r}, e_r)
\in
M^\vee \times \bZ^r
\end{align*}
yield
$0$
while that of the latter monomials yield
$1$.

\begin{lemma} \label{lem:cchart}
The complex algebraic variety
$\bfY_{\vec{\alpha}}$
is the affine open subset of
$\bfY$
associated with the cone spanned by
$u_{\xi_1}, \ldots, u_{\xi_r} \subset M^\vee \times \bZ^r$.
\end{lemma}
\begin{proof}
Suppose that
$\sigma \in \Sigma_\bfY(r)$
is the cone associated with the affine open subset
$\bfY_{\vec{\alpha}} \subset \bfY$.
We have
\begin{align*}
\ddiv (y^{\pm 1}_{\vec{\alpha}, i})
=
\sum_{\xi \in \sigma(1)}
\langle \pm \eta_i, u_\xi \rangle D_\xi, \
\ddiv (v_{\vec{\alpha}, j})
=
\sum_{\xi \in \sigma(1)}
\langle \eta_{d+j+1}, u_\xi \rangle D_\xi
\end{align*}
for
$1 \leq i \leq d+1$
and
$1 \leq j \leq r$,
where
$u_\xi$
are the primitive ray generators of
$\xi$
and
$D_\xi = \overline{O(\xi)}$
are the closures of the orbits corresponding to
$\xi$.
Since
$y^{\pm 1}_{\vec{\alpha}, 1}, \ldots, y^{\pm 1}_{\vec{\alpha}, d+1}$
never vanish on
$\bfY_{\vec{\alpha}}$,
pairing of
$\eta_i$
with the the primitive ray generators in
$\sigma$
must yield
$0$
for
$1 \leq i \leq d+1$.
On the other hand,
pairing of
$\eta_{d+j+1}$
with
$u_{\xi_k}$
must yield
$\delta_{jk}$
for
$1 \leq j, k \leq d+1$.
\end{proof}

Due to the above lemma,
$\bfY_{\vec{\alpha}}$
covers the open stratum of
$\bfY$
and
the open strata of the irreducible toric divisors corresponding to
$\alpha^{1, j_1}, \ldots, \alpha^{r, j_r}$.
If
$\alpha^{i, j_i}, \beta^{i, k_i}$
are connected by an edge in
$\cT_i$
for some
$1 \leq i \leq d+1$,
then we glue
$\bfY_{\vec{\alpha}}$
to
$\bfY_{\vec{\beta}}$
with the coordinate transformations
\begin{align*}
y_{\vec{\alpha}, l}
=
v^{\beta^{i, k_i}_l - \alpha^{i, j_i}_l}_{\vec{\beta}, i} y_{\vec{\beta}, l}, \
v_{\vec{\alpha}, i}
=
v_{\vec{\beta}, i}, \
1 \leq l \leq d+1.
\end{align*}
Thus the coordinate charts
$\{ \bfY_{\vec{\alpha}} \}_{\vec{\alpha} \in \bfA}$
cover the complement in
$\bfY$
of the codimension more than
$1$
strata.

We may write
$v_i$
for
$v_{\vec{\alpha}, i}$
as they do not depend on the choice of
$\vec{\alpha} \in \bfA$. 
Since the weights
\begin{align*}
(0, \ldots, 0, 1, \ldots, 0),
\ldots,
(0, \ldots, 0, 0, \ldots, 1)
\in
M \times \bZ^r
\end{align*}
pair nonnegatively with the primitive ray generators of
$\Sigma_\bfY$,
the polynomial
$v_1 + \cdots + v_r$
defines a regular function on
$\bfY$,
which we denote by
$W_\bfY$.

\begin{definition}
Let
$\bfH \subset \bT^\vee_\bC$
be the complete intersection of very affine hypersurfaces
$H_1, \ldots, H_r$
defined by the Laurent polynomials
$W_1, \ldots, W_r$
from
\pref{eq:Laurents}.
We call the pair
$(\bfY, W_\bfY)$
the
\emph{Landau--Ginzburg $B$-model}
for
$\bfH$.
\end{definition}

\begin{remark}
The pair
$(\bfY, W_\bfY)$
is a conjectural SYZ mirror to
$\bfH$
\cite[Theorem 1.6]{AAK}.
\end{remark}

\begin{lemma} \label{lem:cB-critical}
The critical locus
$\Crit(W_\bfY)$
is given by
$\bigcap^r_{i = 1} \Crit(v_i)$.
\end{lemma}
\begin{proof}
Since we have
$\bigcap^r_{i = 1} \Crit(v_i)
\subset
\Crit(W_\bfY)$,
it remains to show the opposite inclusion.
For
$y \in \Crit(W_\bfY)$
there are
$r$
rays
$\xi_{i, j_i} = \Cone(-\alpha^{i, j_i} \times e_i) \in \Sigma_\bfY$
one from each
$\Sigma_i$
such that
$y
\in
\bigcap^r_{i = 1} D_{\xi_{i, j_i}}$
where
$D_{\xi_{i, j_i}}
=
\overline{O(\xi_{i, j_i})}$
are the closures of the orbits corresponding to
$\xi_{i, j_i}$.
Indeed,
we have
\begin{align*}
\Crit(W_\bfY)
\subset
\bfY \setminus (\bT_\bC \times (\bC^*)^r)
=
\bigcup_{\xi \in \Sigma_\bfY(1)} D_\xi
=
\bigcup^r_{i=1} \bigcup_{\xi_i \in \Sigma_i (1)} D_{\xi_i}.
\end{align*}
If
$y
\notin
\bigcup_{\xi_j \in \Sigma_j (1) } D_{\xi_j}$
for some
$1 \leq j \leq r$,
then there is a neighborhood
$y \in U \subset \bfY$
such that
$U
\cap
\bigcup_{\xi_j \in \Sigma_j (1) } D_{\xi_j}
=
\emptyset$
and
$\Crit(W_\bfY |_U) = \emptyset$,
as
$v_j$
never vanishes on
$U$.
Hence
$y$
belongs to at least one
$D_{\xi_{i, j_i}}$
for each
$i$
and
we obtain
\begin{align*}
\Crit(W_\bfY)
\subset
\bigcup_{\vec{\alpha} \in \bfA} \bigcap^r_{i=1} D_{\xi_{i, j_i}}.
\end{align*}

From the proof of
\pref{lem:cprojection}
it follows that
under
$\bfq$
the intersection of
$r$
facets of
$\Delta_\bfY$
corresponding to
$D_{\xi_{1, j_1}}, \ldots, D_{\xi_{r, j_r}}$
maps to the closure
$\bar{C}_{\vec{\alpha}}$
of the connected component
$C_{\vec{\alpha}} \subset M_\bR \setminus \bigcup^r_{i=1} \Pi_{\Sigma_i}$
indexed by
$\vec{\alpha} = (\alpha^{1, j_1}, \ldots, \alpha^{r, j_r})$.
By
\pref{lem:cchart}
under
\pref{eq:cmoment}
$\bfY_{\vec{\alpha}}$
maps to the intersection of
$r$
facets of
$\Delta_\bfY$
corresponding to
$D_{\xi_{1, j_1}}, \ldots, D_{\xi_{r, j_r}}$.
As explained above,
on
$\bfY_{\vec{\alpha}}$
we are given the coordinates
\begin{align*}
(w^{-1}_1 v^{-\alpha^{1, j_1}_1}_1 \cdots v^{-\alpha^{r, j_r}_1}_r, 
\ldots,
w^{-1}_r v^{-\alpha^{1, j_1}_{d+1}}_1 \cdots v^{-\alpha^{r, j_r}_{d+1}}_r,
v_1, \ldots, v_r)
\end{align*}
which implies
$\Crit(W_\bfY) |_{\bfY_{\vec{\alpha}}} = \emptyset$.
Hence we obtain
\begin{align*}
\Crit(W_\bfY)
\subset
\bigcup_{\vec{\alpha} \in \bfA}
\left(
\bigcap^r_{i=1}
D_{\xi_{i, j_i}}
\setminus
\bfY_{\vec{\alpha}}
\right)
=
\bigcap^r_{i=1} \Crit(v_i)
\end{align*}
as we have
$\bigcup_{\alpha^{i, j_i} \in A_i}
(D_{\xi_{i, j_i}}
\setminus
(Y_{\alpha^i} \times \bC^{r-1}))
=
\Crit(v_i)$
by
\pref{lem:B-critical}.
\end{proof}

\begin{remark} \label{rmk:cB-fibration}
Since the map
\pref{eq:cmoment}
sends each $k$-th intersection of
$D_{\xi_{i, j_i}}, \alpha^{i, j_i} \in \bfA$
to a codimension
$k$
face of
$\Delta_\bfY$,
by
\pref{lem:cB-critical}
it sends
$\Crit(W_\bfY)$
to the union of codimension
$2r$
faces.
On the other hand,
by
\pref{lem:cprojection}
the map
$\bfq \colon M_\bR \times \bR^r \to M_\bR$
homeomorphically sends the union of intersections of
$r$
codimension
$2$
faces of
$\Delta_\bfY$
one from each
$\{ u_i \geq \varphi_i (m) \}$
to
$\bigcap^r_{i=1} \Pi_{\Sigma_i}$.
Hence the composition
\begin{align} \label{eq:cB-fibration}
\bfg
\colon
\Crit(W_\bfY)
\hookrightarrow
\bfY
\xrightarrow{\bfq \circ \pref{eq:cmoment}}
\bigcap^r_{i=1} \Pi_{\Sigma_i}
\end{align}
gives a stratified fibration.
The fiber over a point in the intersection of
$r$
top dimensional strata
one from each
$\Pi_{\Sigma_i}$
is a real $(d - r +1)$-torus
\cite[Prop 12.2.3(b)]{CLS}. 
\end{remark}

\section{Equivariantization and de-equivariantization}
In this section,
following
\cite[Section 4]{She22},
we review the last piece of our proof,
i.e.,
equivariantization
and
de-equivaiantization
of presentable dg categories with certain group actions.
The fact that,
they give mutually inverse equivalences of the categories
we will consider,
enables us to deduce our main result for nonunimodular case from unimodular case.

\subsection{Equivariantization}
For simplicity we keep working over
$\bC$.
Let
$G \subset (\bC^*)^N$
be any subgroup.
Assume that
$G$
acts on a presentable dg category
$\scrC$.
Namely,
there is a monoidal functor
$G \to \End(\scrC)$.
Then
$\scrC$
becomes a module over the monoidal category
$(\Qcoh(G), \star)$,
where
$\star$
is the convolution product induced by the multiplication on
$G$.
Let
$(\Qcoh(BG), \otimes)$
be the monoidal category of $G$-representations.
Taking $G$-invariants defines a functor
\begin{align*}
\scrC
\mapsto
\scrC^G
=
\Hom_{\Qcoh(G)}(\Mod(\bC), \scrC)
\end{align*}
from the category of
$(\Qcoh(G), \star)$-modules
to the category of
$(\Qcoh(BG), \otimes)$-modules,
called
\emph{$G$-equivariantization}.
Here,
the action on
$\Mod(\bC)$
is trivial.

\subsection{De-equivariantization}
We denote by
$G^\vee$
the character group
$\Hom(G, \bC^*)$
of
$G$.
Assume that
$G^\vee$
acts on a presentable dg category
$\scrD$.
This is the same as an action of the monoidal category of $G^\vee$-graded $\bC$-modules,
which in turn is equivalent to
$(\Qcoh(BG), \otimes)$.
Hence 
$\scrD$
becomes a module over
$(\Qcoh(BG), \otimes)$.
Taking $G$-coinvariants defines a functor
\begin{align*}
\scrD
\mapsto
\scrD_{BG}
=
\Mod(\bC) \otimes_{\Qcoh(BG)} \scrD
\end{align*}
from the category of
$(\Qcoh(BG), \otimes)$-modules
to the category of
$(\Qcoh(G), \star)$-modules,
called
\emph{$G$-de-equivariantization}.
Also here,
the action on
$\Mod(\bC)$
is trivial.

\subsection{Mutually inverse equivalences}
For a presentable dg category
$\scrD$
with a $G^\vee$-action,
taking its $G^\vee$-invariants
is equivalent to
taking its $G$-coinvariants.
Since
equivariantization
and
de-equivariantization
give mutually inverse equivalences,
one obtains

\begin{lemma}[{\cite[Lemma 8]{She}}] \label{lem:Shende}
Let
$G \subset (\bC^*)^N$
be any subgroup
and
$G^\vee = \Hom(G, \bC^*)$ . 
Then
$G$-equivariantization
and
$G^\vee$-equivariantization
give mutually inverse equivalences between
the category of presentable dg categories with a $G$-action
and
that with a $G^\vee$-action.
\end{lemma}

\subsection{Quotient construction of toric varieties}
Let
$Y_\Sigma$
be the toric variety associated with a fan
$\Sigma \subset M^\vee_\bR$.
Assume that
$Y_\Sigma$
has no torus factors,
i.e.,
$M^\vee_\bR$
is spanned by primitive ray generators
$u_\rho$
for all
$\rho \in \Sigma(1)$
\cite[Proposition 3.3.9]{CLS}. 
Then
\cite[Theorem 4.1.3]{CLS}
gives the short exact sequence
\begin{align*}
0
\to
M
\to
\bZ^{\Sigma(1)}
=
\bigoplus_{\rho \in \Sigma(1)} \bZ D_\rho
\to
\Cl(Y_\Sigma)
\to
0,
\end{align*}
where
$m \subset M$
maps to
$\ddiv(\chi^m) = \Sigma_{\rho \in \Sigma(1)} \langle m, u_\rho \rangle D_\rho$
and
$\Cl(Y_\Sigma)$
is the divisor class group.
Applying
$\Hom(-, \bC^*)$,
one obtains another short exact sequence
\begin{align*}
1
\to
G = \Hom(\Cl(Y_\Sigma) , \bC^*)
\to
\Hom(\bZ^{\Sigma(1)} , \bC^*)
\cong
(\bC^*)^{\Sigma(1)}
\to
\Hom(M , \bC^*)
\cong
M^\vee_\bC
\to
1.
\end{align*}

\begin{lemma}[{\cite[Lemma 5.1.1]{CLS}}] \label{lem:G}
The subgroup
$G \subset (\bC^*)^{\Sigma(1)}$
is isomorphic to a product of
an algebraic torus
and
a finite group.
More explicitly,
we have
\begin{align} \label{eq:G}
G
=
\{ (t_\rho) \in (\bC^*)^{\Sigma(1)} | \prod_{\rho \in \Sigma(1)} t^{\langle m, u_\rho \rangle}_\rho = 1 \ \text{for all} \ m \subset M \}.
\end{align}  
\end{lemma}

Let
$S = \bC [y_\rho \ | \ \rho \in \Sigma(1)]$
be the total coordinate ring of
$Y_\Sigma$.
Then we have
$\bC^{\Sigma(1)} = \Spec S$
and
the irrelevant ideal is defined as
\begin{align*}
B(\Sigma)
=
\langle y^{\hat{\sigma}} \ | \ \sigma \in \Sigma \rangle
=
\langle y^{\hat{\sigma}} \ | \ \sigma \in \Sigma_{\max} \rangle
\subset S, \
y^{\hat{\sigma}} = \prod_{\rho \notin \sigma(1)} y_\rho.
\end{align*}
We denote by
$Z(\Sigma)$
the zero locus
$V(B(\Sigma)) \subset \bC^{\Sigma(1)}$
of
$B(\Sigma)$.
Via inclusion
$G \subset (\bC^*)^{\Sigma(1)}$
the canonical action of
$(\bC^*)^{\Sigma(1)}$
on
$\bC^{\Sigma(1)}$
induces a $G$-action on
$\bC^{\Sigma(1)} \setminus Z(\Sigma)$.

\begin{lemma}[{\cite[Proposition 5.1.9, Theorem 5.1.11]{CLS}}] \label{lem:quotient}
There is a toric morphism
\begin{align*}
\pi
\colon
\bC^{\Sigma(1)} \setminus Z(\Sigma)
\to
Y_\Sigma
\end{align*}
which is constant on $G$-orbit
and
gives an isomorphism
\begin{align*}
Y_\Sigma
\cong 
\bC^{\Sigma(1)} \setminus Z(\Sigma) / G.
\end{align*}
Namely,
$\pi$
is an almost geometric quotient for the $G$-action.
It is a geometric quotient
if and only if
$\Sigma$
is simplicial.
\end{lemma}

Now,
we drop the assumption that
$Y_\Sigma$
has no torus factors.
Then the primitive ray generators
$u_\rho, \rho \in \Sigma(1)$
span a proper subspace
$(M^\vee_\bR)^\prime \subsetneq M^\vee_\bR$.
Pick the complement
$(M^\vee)^{\prime \prime}$
of
$(M^{\vee})^\prime = (M^{\vee}_\bR)^\prime \cap M^\vee$
so that
$M^\vee = (M^{\vee})^\prime \oplus (M^\vee)^{\prime \prime}$.
The cones of
$\Sigma$
defines a fan
$\Sigma^\prime \subset (M^\vee_\bR)^\prime$.
Note that
we have
$\Sigma^\prime(1) = \Sigma(1)$
and
$B(\Sigma^\prime) = B(\Sigma) \subset S$.
We denote by
$G^\prime$
the subgroup
$\Hom(\Cl(Y_{\Sigma^\prime}), \bC^*)
\subset
(\bC^*)^{\Sigma^\prime(1)}$.
Since
$Y_{\Sigma^\prime}$
has no torus factors,
from the above argument it follows
\begin{align*}
Y_\Sigma
\cong
Y_{\Sigma^\prime} \times (M^\vee_\bC)^{\prime \prime}
\cong
(\bC^{\Sigma^\prime(1)} \setminus Z(\Sigma^\prime) / G^\prime)
\times
(\bC^*)^r.
\end{align*}

\subsection{Model case}
Let
$Y_\Sigma$
be a simplicial affine toric variety associated with a fan
$\Sigma \subset M^\vee_\bR$.
Then
$Y_\Sigma$
has no torus factors.
Since
$\Sigma$
contains only one top dimensional cone,
by definition of
$B(\Sigma)$
and
\pref{lem:quotient}
we obtain
$Z(\Sigma) = \emptyset$
and
\begin{align*}
Y_\Sigma \cong \bA^{d+r+1} / G, \
d+r+1 = \# \Sigma(1).
\end{align*}
If also the rank of
$M$
is
$d+r+1$,
then the inclusion
$M \to \bZ^{\Sigma(1)}$
induces a finite cover
\begin{align*}
h^\vee
\colon
T^* (\bR^{d+r+1} / M)
\to
T^* (\bR^{d+r+1} / \bZ^{d+r+1}).
\end{align*}
Consider the tailored pants
$\tilde{P}_{d+r}$
in the target induced by the pants
\begin{align*}
\{ 1+ \Sigma^{d+r}_{i=1} x_{i}x_{d+r+1} = 0 \}
=
\{ x^{-1}_{d+r+1} + \Sigma^{d+r}_{i=1} x_{i} = 0 \}
\subset
(\bC^*)^{d+r+1}.
\end{align*}
Let
$\tilde{P}_{d+r-1}
\subset
\tilde{P}_{d+r}$
be its closed subset induced by setting
$x_{d+r+1} = 1$.
From
$G^\vee \cong \bZ^{d+r+1} / M$
it follows
\begin{align*}
\tilde{P}_{d+r-1}
\cong
(h^\vee)^{-1}(\tilde{P}_{d+r-1}) / G^\vee.
\end{align*}
Hence we obtain
\begin{align*}
\Fuk(\tilde{P}_{d+r-1})
\simeq
\Fuk((h^\vee)^{-1}(\tilde{P}_{d+r-1}))^{G^\vee}
\simeq
\Fuk((h^\vee)^{-1}(\tilde{P}_{d+r-1}))_{BG}.
\end{align*} 
Taking $G$-invariants,
we obtain
\begin{align*}
\Fuk(\tilde{P}_{d+r-1})^G
\simeq
(\Fuk((h^\vee)^{-1}(\tilde{P}_{d+r-1}))_{BG})^G
\simeq
\Fuk((h^\vee)^{-1}(\tilde{P}_{d+r-1})).
\end{align*} 
Via HMS for pairs of pants $\bZ_2$-folding of the left most term is equivalent to
\begin{align*}
\MF^\infty(\bA^{d+r+1}, y_1 \cdots y_{d+r+1})^G
\simeq
\MF^\infty([\bA^{d+r+1} / G], y_1 \cdots y_{d+r+1}).
\end{align*}
Note that
by
\pref{lem:G}
the product
$y_1 \cdots y_{d+r+1}$
is invariant under the $G$-action.
Thus we obtain
\begin{align*}
\MF^\infty([\bA^{d+r+1} / G], y_1 \cdots y_{d+r+1})
\simeq
\Fuk((h^\vee)^{-1}(\tilde{P}_{d+r-1}))_{\bZ_2}.
\end{align*}

\section{Intersections and categories}
As explained in Remark
\pref{rmk:cA-fibration},
\pref{rmk:cB-fibration},
the critical loci
$\Crit(W_\bfX), \Crit(W_\bfY)$
dually project onto
$\bigcap^r_{i = 1} \Pi_{\Sigma_i}$
under
$\bff, \bfg$. 
We introduce a topology on
$\bigcap^r_{i = 1} \Pi_{\Sigma_i}$
induced by that on
$\Pi_{\Sigma_i}$
defined as in Definition
\pref{dfn:topology}.
Then
$\bigcap^r_{i = 1} \Pi_{\Sigma_i}$
admits an open cover by the intersections
$\bigcap^r_{i = 1} U_{S^{(k_i)}_i}$
for $k_i$-strata
$S^{(k_i)}_i$
of
$\Pi_{\Sigma_i}$.
In this section,
we establish equivalences of corresponding categories over
$\bigcap^r_{i = 1} U_{S^{(k_i)}_i}$
and
glue them to yield HMS for complete intersections of very affine hypersurfaces.
 
\subsection{Covering complete intersections}
By
\pref{lem:cA-critical}
we have
\begin{align*}
\Crit(W_\bfX)
=
\bigcap^r_{i=1} \{ u_i = 0 \}
\cap
\bigcap^r_{i=1} \{ W_i = 0 \}
\subset
(\bC^*)^{d+1} \times \bC^r.
\end{align*}
Since
$H_1, \ldots, H_r$
are in sufficiently general position,
we may assume that
$\Crit(W_\bfX)$
is the union
\begin{align*}
\bigcup_{S^{(k_1)}_1, \ldots, S^{(k_r)}_r}
\left(
\bigcap^r_{i=1} \{ u_i = 0 \}
\cap
\bigcap^r_{i=1} \left( H_{S^{(k_i)}_i} \times \bC^r \right)
\right).
\end{align*}
Under
$\bff$
it maps to
$\bigcup_{S^{(k_1)}_1, \ldots, S^{(k_r)}_r} \bigcap^r_{i=1} S^{(k_i)}_i$.

We denote by
$\sigma(S^{(k_i)}_i)$
the cones of the subfans
$\Sigma_i \subset \Sigma_\bfY \subset \bR^{d+r+1}$
corresponding to
$S^{(k_i)}_i$
and
by
$\xi_1(S^{(k_i)}_i), \ldots, \xi_{d+2-k_i}(S^{(k_i)}_i) \in \Sigma_i(1)$
the rays spanning
$\sigma(S^{(k_i)}_i)$.
Since
$\Pi_{\Sigma_1}, \ldots, \Pi_{\Sigma_r}$
intersect transversely,
the rays
\begin{align} \label{eq:rays}
\xi_1(S^{(k_1)}_1), \ldots, \xi_{d+2-k_1}(S^{(k_1)}_1),
\ldots,
\xi_1(S^{(k_r)}_r), \ldots, \xi_{d+2-k_r}(S^{(k_r)}_r)
\in
\Sigma_\bfY(1)
\end{align}
span a $(\Sigma^r_{i=1} d+2-k_i)$-dimensional cone
$\sigma(S^{(k_1)}_1, \ldots, S^{(k_r)}_r) \in \Sigma_\bfY$.
Note that
we have
\begin{align*}
\sum^r_{i=1} d + 2 - k_i
\leq
d + r + 1.
\end{align*}
Under
$\bfg$
the union of intersections of
\begin{align*}
U(S^{(k_1)}_1, \ldots, S^{(k_r)}_r)
=
\Spec \bC[\sigma^\vee(S^{(k_1)}_1, \ldots, S^{(k_r)}_r) \cap (M \times \bZ^r)]
\end{align*}
and
$\Crit(W_\bfY)$
projects onto
$\bigcup_{S^{(k_1)}_1, \ldots, S^{(k_r)}_r} \bigcap^r_{i=1} S^{(k_i)}_i$.

\subsection{Local $A$-side categories for complete intersections}
Let
$\vec{\beta}$
be a $\Sigma^r_{i=1} l_i$-tuple of vertices
\begin{align*}
\beta^{1, 1}, \ldots, \beta^{1, l_1} \in A_1,
\ldots,
\beta^{r, 1}, \ldots, \beta^{r, l_r} \in A_r
\end{align*}
which together with the origin define rays
\begin{align*}
\xi_{d + 3 - k_1}(S^{(k_1)}_1), \ldots, \xi_{d + 2 - k_1 + l_1}(S^{(k_1)}_1)
\in
\Sigma_1(1),
\ldots,
\xi_{d + 3 - k_r}(S^{(k_r)}_r), \ldots, \xi_{d + 2 - k_r + l_r}(S^{(k_r)}_r)
\in
\Sigma_r(1)
\end{align*}
spanning a top dimensional cone
$\sigma(S^{(k_1)}_1, \ldots, S^{(k_r)}_r, \vec{\beta})
\in
\Sigma_{\bfY, \max}$
together with the rays
\pref{eq:rays}.
We denote by
$\Delta^\vee(S^{(k_1)}_1, \ldots, S^{(k_r)}_r, \vec{\beta})$
the $(d+r+1)$-simplex
\begin{align*}
\Conv(0, (-\alpha^{1,1}, e_1), \ldots, (-\alpha^{1, d + 2 - k_1 + l_1}, e_1), \ldots, (-\alpha^{r, 1}, e_r), \ldots, (-\alpha^{r, d + 2 - k_r + l_r}, e_r))
\end{align*}
with
$\xi_{j}(S^{(k_i)}_i)
=
\Cone(-\alpha^{i, j}, e_i)$
and
$\alpha^{i, d + 2 - k_i + j_i} = \beta^{i, j_i}$.

Recall that
the tailored
$\Delta^\vee(S^{(k_1)}_1, \ldots, S^{(k_r)}_r, \vec{\beta})$-pants
$\tilde{P}(S^{(k_1)}_1, \ldots, S^{(k_r)}_r, \vec{\beta})$
is the inverse image of
$\tilde{P}_{d+r}$
under the map
\begin{align*}
h^\vee(S^{(k_1)}_1, \ldots, S^{(k_r)}_r, \vec{\beta}) \colon (\bC^*)^{d+r+1} \to (\bC^*)^{d+r+1},
\end{align*}
whose restriction gives a finite cover of
$\tilde{P}_{d+r}$.
Here,
$h^\vee(S^{(k_1)}_1, \ldots, S^{(k_r)}_r, \vec{\beta})$
is induced by a homomorphism
\begin{align*}
h(S^{(k_1)}_1, \ldots, S^{(k_r)}_r, \vec{\beta}) \colon \bZ^{d+r+1} \to \bZ^{d+r+1}
\end{align*}
of lattices
which sends
$\Delta^\vee_{d+r+1}$
to
$\Delta^\vee(S^{(k_1)}_1, \ldots, S^{(k_r)}_r, \vec{\beta})$.
Now,
assume that
$\Delta^\vee(S^{(k_1)}_1, \ldots, S^{(k_r)}_r, \vec{\beta})$
is unimodular.
Then
$h^\vee(S^{(k_1)}_1, \ldots, S^{(k_r)}_r, \vec{\beta})$
becomes an isomorphism
and
the monomials
\begin{align*}
x^{-\alpha^{1,1}, e_1}, \ldots, x^{-\alpha^{1, d + 2 - k_1 + l_1}, e_1},
\ldots, 
x^{-\alpha^{r,1}, e_r}, \ldots, x^{-\alpha^{r, d + 2 - k_r + l_r}, e_r}
\end{align*}
give coordinates on the target
$(\bC^*)^{d+r+1}$.

Consider the product
\begin{align*}
\tilde{H}(S^{(k_1)}_1, \vec{\beta})
\times
\cdots
\times
\tilde{H}(S^{(k_r)}_r, \vec{\beta})
\subset
(\bC^*)^{d + 2 - k_1 + l_1}
\times
\cdots
\times
(\bC^*)^{d + 2 - k_r + l_r}
=
(\bC^*)^{d + r + 1}
\end{align*}
of very affine hypersurfaces
\begin{align*}
\tilde{H}(S^{(k_i)}_i, \vec{\beta})
=
\{ \sum^{d + 2 - k_i + l_i}_{j_i = 1} c_{\alpha^{i,j_i}} t^{-\rho_i(\alpha^{i,j_i})} x^{-\alpha^{i,j_i}, e_i} = 0 \}
\subset
(\bC^*)^{d + 2 - k_i + l_i}
\subset
(\bC^*)^{d + r + 1}.
\end{align*}
As explained in Section
$3.1$,
without loss of generality we will assume
$c_{\alpha^{i,j_i}} = 1$
for all
$\alpha^{i,j_i}$.
We denote by
$H(S^{(k_i)}_i, \vec{\beta})$
the quotients of
$\tilde{H}(S^{(k_i)}_i, \vec{\beta})$
by the $\bC^*_{u_i}$-action
\begin{align*}
(u_i, x^{-\alpha^{i,1}, e_i}, \ldots, x^{-\alpha^{i, d + 2 - k_i + l_i}, e_i})
\mapsto
(u_i x^{-\alpha^{i,1}, e_i}, \ldots, u_i x^{-\alpha^{i, d + 2 - k_i + l_i}, e_i}),
\end{align*}
which are isomorphic to $(d-k_i+l_i)$-dimensional tailored pants up to $\frakS_{(d+2-k_i+l_i)}$-equivariant Hamiltonian isotopy of symplectic submanifolds of
$(\bC^*)^{d + 1 - k_i + l_i}$.
Locally,
$\bfH$
is given by a product of such 
$r$
lower dimensional tailored pants.

\begin{lemma} \label{lem:cA-description}
Assume that
$\Delta^\vee(S^{(k_1)}_1, \ldots, S^{(k_r)}_r, \vec{\beta})$ 
is unimodular.
Then the product
\begin{align*}
H(S^{(k_1)}_1, \vec{\beta})
\times
\cdots
\times
H(S^{(k_r)}_r, \vec{\beta})
\subset
(\bC^*)^{d + 1 - k_1 + l_1}
\times
\cdots
\times
(\bC^*)^{d + 1 - k_r + l_r}
=
(\bC^*)^{d + 1}
\end{align*}
is isomorphic to the intersection
\begin{align*}
\bigcap^r_{i=1}
\{ u_i = 0 \}
\cap
\bigcap^r_{i=1}
\{ \sum^{d + 2 - k_i + l_i}_{j_i = 1} x^{-\alpha^{i, j_i}} = 0 \}
\subset
(\bC^*)^{d+1} \times \bC^r.
\end{align*}
\end{lemma}
\begin{proof}
The quotient of
$\times^r_{i=1} \tilde{H}(S^{(k_i)}_i, \vec{\beta})$
by the $(\bC^*)^r_{u_1, \ldots, u_r}$-action
\begin{align*}
&(u_1, \ldots, u_r, x^{-\alpha^{1,1}, e_1}, \ldots, x^{-\alpha^{1, d + 2 - k_1 + l_1}, e_1},
\ldots, 
x^{-\alpha^{r,1}, e_r}, \ldots, x^{-\alpha^{r, d + 2 - k_r + l_r}, e_r}) \\
\mapsto
&(u_1 x^{-\alpha^{1,1}, e_1}, \ldots, u_1 x^{-\alpha^{1, d + 2 - k_1 + l_1}, e_1},
\ldots, 
u_r x^{-\alpha^{r,1}, e_r}, \ldots, u_r x^{-\alpha^{r, d + 2 - k_r + l_r}, e_r})
\end{align*}
is
$\times^r_{i=1} H(S^{(k_i)}_i, \vec{\beta})$.
Since we have
\begin{align*}
\times^r_{i=1} \tilde{H}(S^{(k_i)}_i, \vec{\beta})
=
\bigcap^r_{i=1}
\left(
\{ \sum^{d + 2 - k_i + l_i}_{j_i = 1} t^{-\rho_i(\alpha^{i,j_i})} x^{-\alpha^{i,j_i}, e_i} = 0 \}
\times
(\bC^*)^{r - 1 + k_i - l_i}
\right),
\end{align*}
the quotient is isomorphic to
\begin{align*}
\bigcap^r_{i=1}
\{ u_i = \epsilon \}
\cap
\bigcap^r_{i=1}
\{ \sum^{d + 2 - k_i + l_i}_{j_i = 1} t^{-\rho_i(\alpha^{i,j_i})} x^{-\alpha^{i,j_i}, e_i} = 0 \}
\subset
(\bC^*)^{d+r+1},
\end{align*}
which in turn is isomorphic to
\begin{align*}
\bigcap^r_{i=1}
\{ u_i = 0 \}
\cap
\bigcap^r_{i=1}
\{ \sum^{d + 2 - k_i + l_i}_{j_i = 1} t^{-\rho_i(\alpha^{i,j_i})} x^{-\alpha^{i, j_i}} = 0 \}
\subset
(\bC^*)^{d+1} \times \bC^r.
\end{align*}
\end{proof}

\begin{corollary} \label{cor:cA-description}
Assume that
$\Delta^\vee(S^{(k_1)}_1, \ldots, S^{(k_r)}_r, \vec{\beta})$ 
is unimodular.
Then there is an equivalence
\begin{align*}
\Fuk(\bigcap^r_{i=1}
\{ u_i = 0 \}
\cap
\bigcap^r_{i=1}
\{ \sum^{d + 2 - k_i + l_i}_{j_i = 1} x^{-\alpha^{i, j_i}} = 0 \})
\simeq
\bigotimes^r_{i=1}
\Fuk(H(S^{(k_i)}_i, \vec{\beta})).
\end{align*}
\end{corollary}
\begin{proof}
As an open submanifold of a closed submanifold
$\bfH$
of
$(\bC^*)^{d+r+1}$,
the intersection  
\begin{align*}
\bigcap^r_{i=1}
\{ u_i = 0 \}
\cap
\bigcap^r_{i=1}
\{ \sum^{d + 2 - k_i + l_i}_{j_i = 1} t^{-\rho_i(\alpha^{i,j_i})} x^{-\alpha^{i, j_i}} = 0 \}
\end{align*}
carries a Stein structure,
which in turn defines a Weinstein structure.
Then the claim follows from
\cite[Theorem 1.5, Corollary 1.18]{GPS2}.
\end{proof}

\begin{remark}\label{rmk:restriction}
Consider the product
\begin{align*}
\tilde{H}(S^{(k_1)}_1)
\times
\cdots
\times
\tilde{H}(S^{(k_r)}_r)
\subset
(\bC^*)^{d + 2 - k_1 + l_1}
\times
\cdots
\times
(\bC^*)^{d + 2 - k_r + l_r}
=
(\bC^*)^{d + r + 1}
\end{align*}
of very affine hypersurfaces
\begin{align*}
\tilde{H}(S^{(k_i)}_i)
=
\{ \sum^{d + 2 - k_i}_{j_i = 1} t^{-\rho_i(\alpha^{i,j_i})} x^{-\alpha^{i,j_i}, e_i} = 0 \}
\subset
(\bC^*)^{d + 2 - k_i + l_i}
\subset
(\bC^*)^{d + r + 1}.
\end{align*}
We denote by
$H(S^{(k_i)}_i)$
the quotients of
$\tilde{H}(S^{(k_i)}_i)$
by the $\bC^*_{u_i}$-action
\begin{align*}
(u_i, x^{-\alpha^{i,1}, e_i}, \ldots, x^{-\alpha^{i, d + 2 - k_i + l_i}, e_i})
\mapsto
(u_i x^{-\alpha^{i,1}, e_i}, \ldots, u_i x^{-\alpha^{i, d + 2 - k_i + l_i}, e_i}),
\end{align*}
which are isomorphic to intersections of $l_i$ legs of $(d - k_i + l_i)$-dimensional tailored pants up to $\frakS_{(d+2-k_i+l_i)}$-equivariant Hamiltonian isotopy of symplectic submanifolds of
$(\bC^*)^{d + 1 - k_i + l_i}$.
From the above proof it follows that
the product
\begin{align*}
H(S^{(k_1)}_1)
\times
\cdots
\times
H(S^{(k_r)}_r)
\subset
(\bC^*)^{d + 1 - k_1 + l_1}
\times
\cdots
\times
(\bC^*)^{d + 1 - k_r + l_r}
=
(\bC^*)^{d + 1}
\end{align*}
is isomorphic to the intersection
\begin{align*}
\bigcap^r_{i=1}
\{ u_i = 0 \}
\cap
\bigcap^r_{i=1}
\{ \sum^{d + 2 - k_i}_{j_i = 1} t^{-\rho_i(\alpha^{i,j_i})} x^{-\alpha^{i, j_i}} = 0 \}
\subset
(\bC^*)^{d+1} \times \bC^r
\end{align*}
and
there is an equivalence
\begin{align*}
\Fuk(\bigcap^r_{i=1}
\{ u_i = 0 \}
\cap
\bigcap^r_{i=1}
\{ \sum^{d + 2 - k_i}_{j_i = 1} t^{-\rho_i(\alpha^{i,j_i})} x^{-\alpha^{i, j_i}} = 0 \})
\simeq
\bigotimes^r_{i=1}
\Fuk(H(S^{(k_i)}_i)).
\end{align*}
\end{remark}

\begin{definition} \label{dfn:topology-CI}
We define a topology on
$\bigcap^r_{i=1} \Pi_{\Sigma_i}$
induced by its canonical stratification.
Namely,
for each vertex
$\bfv \in \bigcap^r_{i=1} \Pi_{\Sigma_i}$,
which is an intersection
$\bigcap^r_{i=1} S^{(k_i)}_i$
with
$\sigma(S^{(k_1)}_1, \ldots, S^{(k_r)}_r)
\in
\Sigma_{\bfY, \max}$,
we define the associated open subset
$U_\bfv$
as the union
$\bigcap^r_{i=1} U_{S^{(k_i)}_i}$
of all strata adjacent to
$\bigcap^r_{i=1} S^{(k_i)}_i$.
For each edge
$\mathbf{e} \subset \bigcap^r_{i=1} \Pi_{\Sigma_i}$
connecting two vertices
$\bfv_1, \bfv_2$
we define the associated open subset
$U_{\mathbf{e}}$
as the intersection
$U_{\bfv_1} \cap U_{\bfv_2}$.
Similarly,
for each $k$-stratum
$\bfS^{(k)} \subset \Pi_\Sigma$
adjacent to
$l$
vertices
$\bfv_1, \ldots, \bfv_l$
we define the associated open subset
$U_{\bfS^{(k)}}$
as the intersection
$U_{\bfv_1} \cap \cdots \cap U_{\bfv_l}$.
\end{definition}

Fix a pants decomposition of each
$H_i \subset \bT^\vee_\bC \cong T^* T^{d+1}$
\cite[Theorem 1']{Mik}.
Equip
$\Pi_{\Sigma_i}$
with the topology from Definition
\pref{dfn:topology}
and
then equip
$\bigcap^r_{i=1} \Pi_{\Sigma_i}$
with the topology from Definition
\pref{dfn:topology-CI}.
The basis of the latter is given by the open subsets
$U_{\bfS^{(k)}}$.

\begin{definition}
The
\emph{$A$-side partially defined presheaf}
$\cF^{pre}_\bfA$
of categories for
$\bfH$
is a collection
\begin{align*}
\{ \cF^{pre}_\bfA(U_{\bfS^{(k)}}), R^\bfA_{\bfS^{(k)}, \bfS^{(l)}} \}
\end{align*}
of
sections
and
restriction functors
defined on the basis
$U_{\bfS^{(k)}} = \bigcap^r_{i=1} U_{S^{(k_i)}_i}$
as follows:
\begin{itemize}
\item
The section over
$U_{\bfS^{(k)}}$
is given by
\begin{align*}
\cF^{pre}_\bfA(U_{\bfS^{(k)}})
=
\bigotimes^r_{i=1} \Fuk(H(S^{(k_i)}_i))_{\bZ_2}
\end{align*}
where we equip
$H(S^{(k_i)}_i)$
with the induced Weinstein structure by the canonical one on
$\tilde{P}_{d - k_i + l_i}$.
\item
Along an inclusion
$U_{\bfS^{(l)}}
\hookrightarrow
U_{\bfS^{(k)}}$
the restriction functor
\begin{align*}
R^\bfA_{\bfS^{(k)}, \bfS^{(l)}}
\colon
\cF^{pre}_\bfA(U_{\bfS^{(k)}})
\to
\cF^{pre}_\bfA(U_{\bfS^{(l)}})
\end{align*}
is induced by termwise restriction functors
$R^A_{S^{(k_i)}_i,S^{(l_i)}_i}$.
\end{itemize}
\end{definition}

Since
$U_{\bfS^{(k)}}$
form a basis of the topology of
$\bigcap^r_{i=1} \Pi_{\Sigma_i}$,
we may pass to the sheafification.

\begin{definition}
The
\emph{$A$-side constructible sheaf of categories}
for
$\bfH$
is the sheafification
\begin{align*}
\cF_\bfA \colon \Open(\bigcap^r_{i=1} \Pi_{\Sigma_i})^{op} \to ^{**}\DG,
\end{align*}
where
$\Open(\bigcap^r_{i=1} \Pi_{\Sigma_i})$
is the category of open subsets of
$\bigcap^r_{i=1} \Pi_{\Sigma_i}$
with respect to the topology from Definition
\pref{dfn:topology-CI}.
\end{definition}

\begin{theorem}
\label{thm:global2}
There is an equivalence
\begin{align*}
\Fuk(\bfH)_{\bZ_2}
\simeq
\cF_\bfA(\bigcap^r_{i=1} \Pi_{\Sigma_i}).
\end{align*}
\end{theorem}
\begin{proof}
By the construction of
$\cF_\bfA$
and
Remark
\pref{rmk:restriction},
the caim follows from the same argument as in the proof of
\pref{thm:global}.
Indeed,
let
$\bfv \in \bigcap^r_{i=1} \Pi_{\Sigma_i}$
be a vertex
$\bigcap^r_{i=1} S^{(k_i)}_i$
with
$\sigma(S^{(k_1)}_1, \ldots, S^{(k_r)}_r)
\in
\Sigma_{\bfY, \max}$
such that
$H(S^{(k_1)}_1)$
is an external $(d-k_i)$-dimensional pants
which makes sense
whenever
$S^{(k_2)}_2, \ldots, S^{(k_r)}_r$
fixed.
Run the same argument as in the proof of
\pref{thm:global}
for all connected components of the first factor.
Fixing
the first factors of the results
and
$S^{(k_3)}_3, \ldots, S^{(k_r)}_r$,
repeat the same process.
Inductively,
we obtain the desired equivalence.
\end{proof}

\subsection{Local $B$-side categories for complete intersections}
Also in this subsection,
we assume
$\Delta^\vee(S^{(k_1)}_1, \ldots, S^{(k_r)}_r, \vec{\beta})$ 
to be unimodular.
Let
\begin{align*}
y_1(S^{(k_1)}_1), \ldots, y_{d + 2- k_1 + l_1}(S^{(k_1)}_1),
\ldots, 
y_1(S^{(k_r)}_r), \ldots, y_{d + 2 - k_r + l_r}(S^{(k_r)}_r)
\end{align*}
be local coordinates for
\begin{align*}
\bA^{d + 2 - k_1 + l_1} \times \cdots \times \bA^{d + 2 - k_r + l_r}
&\cong
U(S^{(k_1)}_1, \ldots, S^{(k_r)}_r, \vec{\beta}) \\
&=
\Spec \bC[\sigma^\vee(S^{(k_1)}_1, \ldots, S^{(k_r)}_r, \vec{\beta}) \cap (M \times \bZ^r)]
\subset
\bfY.
\end{align*}

\begin{lemma} \label{lem:cB-description}
Assume that
$\Delta^\vee(S^{(k_1)}_1, \ldots, S^{(k_r)}_r, \vec{\beta})$ 
is unimodular.
Then there is an equivalence
\begin{align*}
\MF^\infty(U(S^{(k_1)}_1, \ldots, S^{(k_r)}_r, \vec{\beta}), W_\bfY)
\simeq
\bigotimes^r_
{i=1} \MF^\infty(\bA^{d+2-k_i+l_i}, y_1(S^{(k_i)}_i) \cdots y_{d+2-k_i+l_i}(S^{(k_i)}_i)).
\end{align*}
\end{lemma}
\begin{proof}
Since we have
\begin{align*}
W_\bfY |_{U(S^{(k_1)}_1, \ldots, S^{(k_r)}_r, \vec{\beta})}
=
y_1(S^{(k_1)}_1) \cdots y_{d+2-k_1+l_1}(S^{(k_1)}_1)
+ \cdots + 
y_1(S^{(k_r)}_r) \cdots y_{d+2-k_r+l_r}(S^{(k_r)}_r),
\end{align*}
the claim follows from
\pref{lem:B-critical},
\pref{lem:cB-critical}
and
\cite[Theorem 4.1.3]{Pre}.
\end{proof}

Recall from
\cite[Proposition A.3.1]{Pre}
that
$\MF^\infty$
is a sheaf on
$\bfY$. 
\begin{definition}
The
\emph{$B$-side constructible sheaf of categories}
for
$\bfH$
is the composite: 
$$
\cF_\bfB \colon \Open(\Pi_\Sigma)^{op} 
\longrightarrow \Open(\bfY)^{op} \stackrel{\MF^\infty}\longrightarrow \, ^{**}\DG.
$$
\end{definition}

For concreteness, let us describe sections and restrictions of the sheaf $\cF_\bfB$ on the basis
$U_{\bfS^{(k)}} = \bigcap^r_{i=1} U_{S^{(k_i)}_i}$. 
\begin{itemize}
\item
The section of $\cF_\bfB$ over
$U_{\bfS^{(k)}}$
is given by
\begin{align*}
\cF_\bfB(U_{\bfS^{(k)}})
=
\bigotimes^r_
{i=1} \MF^\infty(\bA^{d+2-k_i+l_i}, y_1(S^{(k_i)}_i) \cdots y_{d+2-k_i+l_i}(S^{(k_i)}_i)).
\end{align*}
\item
Along an inclusion
$U_{\bfS^{(l)}}
\hookrightarrow
U_{\bfS^{(k)}}$
the restriction functor
\begin{align*}
R^\bfB_{\bfS^{(k)}, \bfS^{(l)}}
\colon
\cF_\bfB(U_{\bfS^{(k)}})
\to
\cF_\bfB(U_{\bfS^{(l)}})
\end{align*}
is induced by termwise restriction functors
$R^B_{S^{(k_i)}_i,S^{(l_i)}_i}$.
\end{itemize}


\subsection{Gluing equivalences}
\begin{lemma} \label{lem:local}
Assume that
$\Delta^\vee(S^{(k_1)}_1, \ldots, S^{(k_r)}_r, \vec{\beta})$ 
is unimodular.
Then there is an equivalence
\begin{align*}
\Fuk(\bigcap^r_{i=1}
\{ u_i = 0 \}
\cap
\bigcap^r_{i=1}
\{ \sum^{d + 2 - k_i + l_i}_{j_i = 1} x^{-\alpha^{i, j_i}} = 0 \})_{\bZ_2}
\simeq
\MF^\infty(U(S^{(k_1)}_1, \ldots, S^{(k_r)}_r, \vec{\beta}), W_\bfY).
\end{align*}
\end{lemma}
\begin{proof}
Due to Corollary
\pref{cor:cA-description}
and
\pref{lem:cB-description}
it suffices to show an equivalence
\begin{align} \label{eq:loceq}
\Fuk(H(S^{(k_i)}_i, \vec{\beta}))_{\bZ_2}
\simeq
\MF^\infty(\bA^{d+2-k_i+l_i}, y^i_1(S^{(k_i)}_i) \cdots y^i_{d+2-k_i+l_i}(S^{(k_i)}_i)).
\end{align}
This follows from
\pref{thm:lift2},
since by construction
$H(S^{(k_i)}_i, \vec{\beta})$
is isomorphic to $(d-k_i+l_i)$-dimensional tailored pants up to $\frakS_{(d+2-k_i+l_i)}$-equivariant Hamiltonian isotopy of symplectic submanifolds of
$(\bC^*)^{d + 1 - k_i + l_i}$.
\end{proof}

\begin{theorem} \label{thm:main2}
Assume that
$\Delta^\vee(S^{(k_1)}_1, \ldots, S^{(k_r)}_r, \vec{\beta})$ 
are unimodular for all
$S^{(k_1)}_1, \ldots, S^{(k_r)}_r$
such that
$\bigcap^r_{i=1} H_{S^{(k_i)}_i}
\subset
(\bC^*)^{d+1}$
is nonempty.
Then there is an equivalence
\begin{align*}
\cF_\bfA(\bigcap^r_{i=1} \Pi_{\Sigma_i})
\simeq 
\cF_\bfB(\bigcap^r_{i=1} \Pi_{\Sigma_i})
\end{align*}
compatible with restrictions.
\end{theorem}
\begin{proof}
We glue the local equivalences from
\pref{lem:local}.
Again,
due to Corollary
\pref{cor:cA-description}
and
\pref{lem:cB-description}
it suffices to show the compatibility of
\pref{eq:loceq}
with gluing,
which follows from the construction of
$\cF_\bfA, \cF_\bfB$,
Remark
\pref{rmk:restriction}
and
the same argument as in the proof of 
\pref{thm:main1}.
Indeed,
let
$\bfv \in \bigcap^r_{i=1} \Pi_{\Sigma_i}$
be a vertex
$\bigcap^r_{i=1} S^{(k_i)}_i$
with
$\sigma(S^{(k_1)}_1, \ldots, S^{(k_r)}_r)
\in
\Sigma_{\bfY, \max}$
such that
$H(S^{(k_1)}_1)$
is an external $(d-k_i)$-dimensional pants
which makes sense
whenever
$S^{(k_2)}_2, \ldots, S^{(k_r)}_r$
fixed.
Run the same argument as in the proof of
\pref{thm:main1}
for all connected components of the first factor.
Fixing
the first factors of the results
and
$S^{(k_3)}_3, \ldots, S^{(k_r)}_r$,
repeat the same process.
Inductively,
we obtain the desired equivalence.
\end{proof}

\begin{corollary}
Under the same assumption as above,
there is an equivalence
\begin{align*}
\Fuk(\bfH)_{\bZ_2}
\simeq 
\MF^\infty(\bfY, W_\bfY).
\end{align*}
\end{corollary}

\subsection{Nonunimodular case}
Finally,
we drop the assumption on
$\Delta^\vee(S^{(k_1)}_1, \ldots, S^{(k_r)}_r, \vec{\beta})$
to be unimodular.
The simplicial affine toric variety
$U(S^{(k_1)}_1, \ldots, S^{(k_r)}_r, \vec{\beta})$ 
becomes isomorphic to the geometric quotient
\begin{align*}
\bA^{d+r+1} / G(S^{(k_1)}_1, \ldots, S^{(k_r)}_r, \vec{\beta}).
\end{align*}
by the finite subgroup 
$G(S^{(k_1)}_1, \ldots, S^{(k_r)}_r, \vec{\beta})$
of
$(\bC^*)^{d+r+1}$
with canonically induced action.
The inclusion extends to a short exact sequence
\begin{align} \label{eq:SES}
0
\to
G(S^{(k_1)}_1, \ldots, S^{(k_r)}_r, \vec{\beta})
\to
(\bC^*)^{d+r+1}
\to
M^\vee(S^{(k_1)}_1, \ldots, S^{(k_r)}_r, \vec{\beta})_{\bC}
\to
0,
\end{align}
where
$M^\vee(S^{(k_1)}_1, \ldots, S^{(k_r)}_r, \vec{\beta})$
is the cocharacter lattice associated with
$U(S^{(k_1)}_1, \ldots, S^{(k_r)}_r, \vec{\beta})$.
Then we have
\begin{align*}
\MF^\infty([\bA^{d+r+1} / G(S^{(k_1)}_1, \ldots, S^{(k_r)}_r, \vec{\beta})], W_\bfY)
\simeq
\MF^\infty(\bA^{d+r+1}, W_\bfY)^{G(S^{(k_1)}_1, \ldots, S^{(k_r)}_r, \vec{\beta})}.
\end{align*}

On the $A$-side,
$\tilde{P}(S^{(k_1)}_1, \ldots, S^{(k_r)}_r, \vec{\beta})$
becomes a finite cover of
$\tilde{P}_{d+r}$.
Since
$U(S^{(k_1)}_1, \ldots, S^{(k_r)}_r, \vec{\beta})$
has no torus factors,
\pref{eq:SES}
is obtained from the short exact sequence
\begin{align*}
0
\to
M(S^{(k_1)}_1, \ldots, S^{(k_r)}_r, \vec{\beta})
\to
\bZ^{d+r+1}
\to
G^\vee(S^{(k_1)}_1, \ldots, S^{(k_r)}_r, \vec{\beta})
\to
0
\end{align*}
by taking
$\Hom(-, \bC^*)$.
Here,
we use the symbol
$G^\vee(S^{(k_1)}_1, \ldots, S^{(k_r)}_r, \vec{\beta})$
to denote the divisor class group of
$U(S^{(k_1)}_1, \ldots, S^{(k_r)}_r, \vec{\beta})$,
which acts on
\begin{align*}
(\bC^*)^{d+r+1}
\cong
T^* (\bR^{d+r+1} / M(S^{(k_1)}_1, \ldots, S^{(k_r)}_r, \vec{\beta})).
\end{align*}
Then the finite cover is given by
\begin{align*}
\tilde{P}(S^{(k_1)}_1, \ldots, S^{(k_r)}_r, \vec{\beta})
\to
\tilde{P}(S^{(k_1)}_1, \ldots, S^{(k_r)}_r, \vec{\beta})/G^\vee(S^{(k_1)}_1, \ldots, S^{(k_r)}_r, \vec{\beta})
=
\tilde{P}_{d+r}.
\end{align*}
Similarly,
replace
\begin{align*}
\times^r_{i=1} \tilde{H}(S^{(k_i)}_i, \vec{\beta})
=
\bigcap^r_{i=1}
\left(
\{ \sum^{d + 2 - k_i + l_i}_{j_i = 1} x^{-\alpha^{i,j_i}, e_i} = 0 \}
\times
(\bC^*)^{r - 1 + k_i - l_i}
\right)
=
\bigcap^r_{i=1}
\{ \sum^{d + 2 - k_i + l_i}_{j_i = 1} x^{-\alpha^{i,j_i}, e_i} = 0 \}
\end{align*}
and
$\times^r_{i=1} H(S^{(k_i)}_i, \vec{\beta})$
with their inverse images.
Then we have
\begin{align} \label{eq:A-bridge1}
\Fuk((\times^r_{i=1} H(S^{(k_i)}_i, \vec{\beta})) /G^\vee(S^{(k_1)}_1, \ldots, S^{(k_r)}_r, \vec{\beta}))
\simeq
\Fuk(\times^r_{i=1} H(S^{(k_i)}_i, \vec{\beta}))^{G^\vee(S^{(k_1)}_1, \ldots, S^{(k_r)}_r, \vec{\beta})}.
\end{align}

As explained above,
there is an equivalence to the de-equivariantization
\begin{align} \label{eq:A-bridge2}
\Fuk(\times^r_{i=1} H(S^{(k_i)}_i, \vec{\beta}))^{G^\vee(S^{(k_1)}_1, \ldots, S^{(k_r)}_r, \vec{\beta})}
\simeq
\Fuk(\times^r_{i=1} H(S^{(k_i)}_i, \vec{\beta}))_{BG(S^{(k_1)}_1, \ldots, S^{(k_r)}_r, \vec{\beta})}
\end{align}
with respect to
$G(S^{(k_1)}_1, \ldots, S^{(k_r)}_r, \vec{\beta})$.
Combining
\pref{eq:A-bridge1},
\pref{eq:A-bridge2}
and
the equivalence for unimodular case,
we obtain
\begin{align*}
\MF^\infty(\bA^{d+r+1}, W_\bfY)
\simeq
(\Fuk(\times^r_{i=1} H(S^{(k_i)}_i, \vec{\beta}))_{\bZ_2})_{BG(S^{(k_1)}_1, \ldots, S^{(k_r)}_r, \vec{\beta})}.
\end{align*}
Passing to the equivariantization,
we obtain
\begin{align*}
\MF^\infty([\bA^{d+r+1}/G(S^{(k_1)}_1, \ldots, S^{(k_r)}_r, \vec{\beta})], W_\bfY)
\simeq
((\Fuk(\times^r_{i=1} H(S^{(k_i)}_i, \vec{\beta}))_{\bZ_2})_{BG(S^{(k_1)}_1, \ldots, S^{(k_r)}_r, \vec{\beta})})^{G(S^{(k_1)}_1, \ldots, S^{(k_r)}_r, \vec{\beta})}
\end{align*}
which by
\pref{lem:Shende}
implies
\begin{align} \label{eq:loceq2}
\MF^\infty(U(S^{(k_1)}_1, \ldots, S^{(k_r)}_r, \vec{\beta}), W_\bfY)
\simeq
\Fuk((\times^r_{i=1} H(S^{(k_i)}_i, \vec{\beta})))_{\bZ_2}.
\end{align}

We glue the above local equivalences.
It suffices to show the compatibility of
\pref{eq:loceq2}
with gluing,
which follows from the same argument as the proof of
\pref{thm:main2}
extended in a straight forward way.
Indeed,
the compatibility of the actions of
$G(S^{(k_1)}_1, \ldots, S^{(k_r)}_r, \vec{\beta})$
and
$G^\vee(S^{(k_1)}_1, \ldots, S^{(k_r)}_r, \vec{\beta})$
with gluing follows from
\pref{lem:G}
and
the combinatorial duality.

\section{HMS for very affine hypersurfaces revisited}
In this section,
we recover HMS for very affine hypersurfaces established by Gammage--Shende,
whose $B$-side category is the derived category of coherent sheaves on the toric boundary divisor.
Under their assumption,
our computations carried out in Section
$4, 5$
become simpler
and
yields the original statement as a byproduct.
We stress that
there could be more than one approach to comparing our results in this paper to the results of Gammage–Shende.
If the goal was only to show that
one can derive from our results an HMS statement
where the mirror category is the category of coherent sheaves on the toric boundary divisor (as in Gammage–Shende),
then one could try to prove this by arguing only on the B-side. Namely, one would need a generalization of the Orlov dictionary between MF and coherent sheaves (i.e.
\pref{lem:Coh-MF}
above) to the global stacky toric setting.

In the following we shall give a different argument,
for two reasons.
First,
the generalization of Orlov’s dictionary we would need is not well-documented in the literature.
Second, this kind of proof would not clarify the relationship between our methods and Gammage–Shende.
We argue instead by showing that
in the setting considered by Gammage–Shende the local skeleta
we build in Section
$4$
actually globalize.
This,
together with an iterated application of
\pref{lem:Coh-MF},
yields the desired comparison.

\subsection{$A$-side category}
As explained above,
the equivalence
\begin{align*}
\Fuk(H)_{\bZ_2}
\simeq
\cF_A(\Pi_\Sigma)
\simeq
\cF_B(\Pi_\Sigma)
=
\MF^\infty(Y, W_Y)
\end{align*}
from
\pref{thm:main1}
can be obtained
whenever we choose an adapted unimodular triangulation
$\cT$
of the lattice polytope
$\Delta^\vee \subset M^\vee_\bR$.
Now,
we assume further that
$\cT$
is star-shaped.
Then
the set
$A$
of vertices of
$\cT$
must contain
$0$
and
all
$v \in \VVert(\Pi_\Sigma)$
belong to
$\del \Pi^0_\Sigma$,
where
$\Pi^0_\Sigma$
is the closure of the connected component of
$\bR^{d+1} \setminus \Pi_\Sigma$
corresponding to
$0 \in A$.

\begin{remark}
Gammage--Shende also assumed that
the associated fan
$\Sigma \subset M^\vee_\bR$
with
$\cT$
is perfectly centered,
i.e.,
$\Pi^0_\Sigma$
is perfectly centered
\cite[Definition 6.2.1]{GS1}.
Namely,
for each nonempty face of
$\Pi^0_\Sigma$
its relative interior intersects its normal cone transported to
$M^\vee_\bR$
via the inner product
$M^\vee_\bR \cong M_\bR$.
This additional assumption was used
when they constructed a global skeleton of
$H$
by gluing its open cover
\cite[Theorem 6.2.4]{GS1}.
More precisely,
the assumption allows one to define an open cover of
$H$
which restricts to the open cover of the skeleton.
\end{remark}

Now,
one can modify the argument in the proof of
\pref{thm:global}
as follows.
Suppose that
the canonical functor
\begin{align*}
\Fuk(H_{v^0_1} \cup \bigcup^{l_1}_{i_1 = 1} H_{v^1_{i_1}} \cup \cdots \cup \bigcup^{l_{k-1}}_{i_{k-1} = 1} H_{v^{k-1}_{i_{k-1}}})_{\bZ_2}
\to
\cF_A(U_{v^0_1} \cup \bigcup^{l_1}_{i_1 = 1} U_{v^1_{i_1}} \cup \cdots \cup \bigcup^{l_{k-1}}_{i_{k-1} = 1} U_{v^{k-1}_{i_{k-1}}})
\end{align*}
is an equivalence
and
$v^k_i \in \VVert(\Pi_\Sigma)$
is connected by an edge to at least one
$v^{k-1}_j$.
Consider Nadler's Weinstein structures transported to
$H_{v^0_1}, H_{v^1_{i_1}}, \ldots, H_{v^{k-1}_{i_{k-1}}}, H_{v^k_i}$
whose final legs correspond to the edges dual to boundary facets.
Since
$\cT$
is star-shaped,
these final legs are free
and
the transported Nadler's Weinstein structures glue to
yield a global Weinstein structure without modification.
By Corollary
\pref{cor:key}
we obtain a canonically induced equivalence
\begin{align*}
\Fuk(H_{v^0_1} \cup \bigcup^{l_1}_{i_1 = 1} H_{v^1_{i_1}} \cup \cdots \cup \bigcup^{l_{k-1}}_{i_{k-1} = 1} H_{v^{k-1}_{i_{k-1}}} \cup H_{v^k_i} )_{\bZ_2}
\to
\cF_A(U_{v^0_1} \cup \bigcup^{l_1}_{i_1 = 1} U_{v^1_{i_1}} \cup \cdots \cup \bigcup^{l_{k-1}}_{i_{k-1} = 1} U_{v^{k-1}_{i_{k-1}}} \cup U_{v^k_i}).
\end{align*} 
Iteratively,
we obtain an equivalence 
$\Fuk(H)_{\bZ_2}
\simeq
\cF_A(\Pi_\Sigma)$.

The above argument shows that
also the skeleta of the transported Nadler's Weinstein structures glue to yield a global skeleton
$\Core(H)$
of
$H$.
From
\pref{lem:Symp-Symp2}
it follows that
$\Core(H)$
can be identified with the spherical projectivization
$- \del_\infty \bL_\Sigma \subset S^* \bT^\vee$
of the negated FLTZ skeleton.
Then by
\cite[Theorem 1.4]{GPS3}
there is a canonical equivalence
$\Fuk(H) = \mSh^\diamondsuit_{\del_\infty \bL_\Sigma}(\del_\infty \bL_\Sigma)$.
Here,
we obtain
$\Core(H)$
by gluing its closed cover,
while
in
\cite[Theorem 6.2.4]{GS1}
it was obtained by gluing its open cover.

Summarizing,
we obtain

\begin{lemma} \label{lem:GS-A}
If the chosen adapted unimodular triangulation
$\cT$
is star-shaped,
then there is an equivalence
\begin{align*}
\cF_A(\Pi_\Sigma)
\simeq
\Fuk(H)_{\bZ_2}
=
\mSh^\diamondsuit_{\del_\infty \bL_\Sigma}(\del_\infty \bL_\Sigma)_{\bZ_2}.
\end{align*}
\end{lemma}

\begin{remark}
Here,
the Weinstein structure on
$H$
can be obtained by restricting the canonical one on
$T^* \bT^\vee$.
By
\cite[Corollary 6.2.6]{GS1}
the pair
$(T^* \bT^\vee, H)$
defines a Liouville pair
whose relative skeleton is
$-\bL_\Sigma$.
We denote by
$\Fuk(T^*\bT^\vee, H)$
the ind-completion of its partially wrapped Fukaya category. 
Applying
\cite[Theorem 1.1, Corollary 7.22]{GPS3},
one obtains a commutative diagram
\begin{align*}
\begin{gathered}
\xymatrix{
\mSh^\diamondsuit_{\del_\infty \bL_\Sigma}(\del_\infty \bL_\Sigma) \ar[r]^-{} \ar@{=}[d]
& \Sh^\diamondsuit_{\bL_\Sigma}(\bT^\vee) \ar@{=}[d] \\
\Fuk(H) \ar[r]^-{}
& \Fuk(T^*\bT^\vee, H).
}
\end{gathered}
\end{align*}
\end{remark}

\subsection{$B$-side category}
Under the additional assumption,
the diagram
\begin{align*}
\prod_{v \in \VVert(\Pi_\Sigma)} \MF^\infty(Y_v, y_{e^v_1} \cdots y_{e^v_{d+2}}) \rightarrow
\prod_{e \in \Edge(\Pi_\Sigma)} \MF^\infty(Y_v |_{\{ y_e \neq 0 \}}, y_{e^v_1} \cdots y_{e^v_{d+2}}) \rightarrow \cdots
\end{align*}
whose limit is
$\cF_B(\Pi_\Sigma) = \MF^\infty(Y, W_Y)$
has an alternative description.
Since
$\cT$
is star-shaped,
all the top dimensional cones of
$\Sigma_Y$
share the ray
$\Cone(0, \ldots, 0, 1) \subset M^\vee_\bR \times \bR$.
For each
$v \in \VVert(\Pi_\Sigma)$
let
$e^v_{d+2}$
stand for the edge dual to the boundary facet
$\sigma_{e^v_{d+2}}$
of
$\sigma_v$.
By
\pref{lem:Coh-MF}
we have the equivalence
\begin{align*}
\IndCoh(Y^{d+2}_{v, d+1})_{\bZ_2}
\simeq
\MF^\infty(y_{e^v_1} \cdots y_{e^v_{d+2}}).
\end{align*}
If
$e$
is an edge connecting two vertices
$v, v^\prime$,
then the associated restriction functors become
\begin{align*}
\begin{gathered}
\bar{R}^B_{v, e}
\colon
\IndCoh(Y^{d+2}_{v, d+1})_{\bZ_2}
\to
\IndCoh(Y^{d+2}_{v, d+1} |_{\{ y_{e^v_i} \neq 0 \}})_{\bZ_2}, \\
\bar{R}^B_{v^\prime, e}
\colon
\IndCoh(Y^{d+2}_{v^\prime, d+1})_{\bZ_2}
\to
\IndCoh(Y^{d+2}_{v^\prime, d+1} |_{\{ y_{e^{v^\prime}_j}\neq 0 \}})_{\bZ_2}.
\end{gathered}
\end{align*}

Recall that
$\Pi_\Sigma$
encodes all the combinatorial information to recover both
$H$
and
$Y$
from pieces.
In particular,
it gives a coordinate transformation
\begin{align} \label{eq:transformation4}
(y_{e^{v^\prime}_1}, \cdots, y_{e^{v^\prime}_{d+2}})
\to
(y_{e^v_1}, \cdots, y_{e^v_{d+2}})
\end{align}
on
$Y_e$.
We denote by
$T_\Sigma$
the closed toric subvariety of
$Y$
associated with the quotient
$\Sigma$
of
$\Sigma_Y = \Cone(-\cT \times \{ 1 \})$
by the ray
$\Cone(0, \ldots, 0, 1)$.
Then
\pref{eq:transformation4}
induces a coordinate transformation
\begin{align} \label{eq:transformation5}
(\tilde{y}_{e^{v^\prime}_1}, \cdots, \tilde{y}_{e^{v^\prime}_{d+1}})
\to
(\tilde{y}_{e^v_1}, \cdots, \tilde{y}_{e^v_{d+1}})
\end{align}
on the affine piece
$\tilde{Y}_e
=
\tilde{Y}_v \cap \tilde{Y}_{v^\prime}
\cong
Y^{d+2}_{v, d+1}
\cap
Y^{d+2}_{v^\prime, d+1}$
of
$T_\Sigma$
corresponding to
$Y_e$,
which defines a gluing datum
\begin{align*}
\IndCoh(Y^{d+2}_{v, d+1} |_{\{ y_{e^v_i} \neq 0 \}})_{\bZ_2}
\simeq
\IndCoh(Y^{d+2}_{v^\prime, d+1} |_{\{ y_{e^{v^\prime}_j}\neq 0 \}})_{\bZ_2}.
\end{align*}
As
$\tilde{Y}_v$
form an open cover of
$T_\Sigma$,
such a gluing datum is compatible with further restrictions.
Hence
$\cF_B(\Pi_\Sigma)$
is also the limit of the diagram
\begin{align*}
\prod_{v \in \VVert(\Pi_\Sigma)} \IndCoh(Y^{d+2}_{v, d+1})_{\bZ_2}
\rightarrow
\prod_{e \in \Edge(\Pi_\Sigma)} \IndCoh(Y^{d+2}_{v, d+1} |_{\{ y_{e^v_i} \neq 0 \}})_{\bZ_2}
\rightarrow
\cdots.
\end{align*}
Here,
the compatible equivalences on pieces glue to yield a canonical equivalence
\begin{align*}
\MF^\infty(Y, W_Y)
\simeq
\IndCoh(\del T_\Sigma)_{\bZ_2}
\end{align*}
for the toric boundary divisor
$\del T_\Sigma$,
as
$\IndCoh$
satisfies Zariski descent
\cite[Proposition 4.2.1]{Gai}.

Summarizing,
we obtain

\begin{lemma} \label{lem:GS-B}
If the chosen adapted unimodular triangulation
$\cT$
is star-shaped,
then there is a canonical equivalence
\begin{align*}
\cF_B(\Pi_\Sigma)
=
\MF^\infty(Y, W_Y)
\simeq
\IndCoh(\del T_\Sigma)_{\bZ_2}.
\end{align*}
\end{lemma}

\begin{remark}
In general,
for a smooth quasiprojective stacky fan
$\Sigma$
the canonical morphism
\begin{align*}
\colim_{0 \neq \sigma \in \Sigma}\overline{O(\sigma)}
\to
\del T_\Sigma
\end{align*}
of algebraic stacks is an isomorphism
\cite[Lemma 3.4.1]{GS1}.
Namely,
as an algebraic stack
$\del T_\Sigma$
is the union of the nontrivial orbit closures.
Here,
we do not need this result,
since by definition the toric boundary divisors of affine open pieces glue to yield
$\del T_\Sigma$.
\end{remark}

\subsection{Homological mirror symmetry}
Recall that the equivalence
$\cF_A(\Pi_\Sigma) \simeq \cF_B(\Pi_\Sigma)$
from
\pref{thm:main1}
is obtained by gluing the local equivalences
$\cF_A(U_v) \simeq \cF_B(U_v)$
for all
$v \in \VVert(\Pi_\Sigma)$,
each of which comes from the extended Nadler's equivalence
\begin{align*}
\varphi^v_{d+2}
\colon
\Fuk(H_v)_{\bZ_2}
=
\mSh^\diamondsuit(\Lambda^\infty_{d+1})_{\bZ_2}
\to
\IndCoh(Y^{d+2}_{v, d+1})_{\bZ_2}
\to
\MF^\infty(\bA^{d+2}, y_{e^v_1} \cdots y_{e^v_{d+1}} y_{e^v_{d+2}})
\end{align*}
from
\pref{thm:lift2}.
Now,
fix a Nadler's equivalence
$\varphi^v_{d+2}$
for each
$v \in \VVert(\Pi_\Sigma)$.
Suppose that
$v$
is connected to another vertex
$v^\prime$
by an edge
$e$.
Then
$\Sigma_Y$
gives the correspondences
\begin{align} \label{eq:transformation6}
(y_{e^v_1}, \ldots, y_{e^v_{d+2}})
\leftrightarrow
(y_{e^{v^\prime}_1}, \ldots, y_{e^{v^\prime}_{d+2}})
\end{align}
of the coordinates on
$Y_e$
inducing that
\begin{align} \label{eq:transformation7}
(\tilde{y}_{e^v_1}, \ldots, \tilde{y}_{e^v_{d+1}})
\leftrightarrow
(\tilde{y}_{e^{v^\prime}_1}, \ldots, \tilde{y}_{e^{v^\prime}_{d+1}})
\end{align}
on
$\tilde{Y}_e$.
Since the final legs are free,
by construction of Nadler's equivalences
$\varphi^{v^\prime}_{d+2}$
is compatible with
$\varphi^v_{d+2}$
and
we have the commutative diagram
\begin{align} \label{eq:commute2}
\begin{gathered}
\xymatrix{
\varphi^{v^\prime}_{d+2} \colon
\Fuk(H_{v^\prime})_{\bZ_2} \ar[r]^-{} \ar[d]
&
\IndCoh(Y^{d+2}_{v^\prime, d+1})_{\bZ_2} \ar[r]^-{} \ar^{\tilde{R}^B_{v^\prime, e}}[d]
&
\MF^\infty(\bA^{d+2}, y_{e^{v^\prime}_1}, \ldots, y_{e^{v^\prime}_{d+1}} y_{e^{v^\prime}_{d+2}}) \ar^{R^B_{v^\prime, e}}[d] \\
\varphi^{v^\prime}_{d+2} |_{e^{v^\prime}_{i^\prime}} \colon
\Fuk(H_e)_{\bZ_2} \ar[r]^-{} \ar^{\simeq}[d]
&
\IndCoh(Y^{d+2}_{v^\prime, d+1} |_{\{ y_e \neq 0 \}})_{\bZ_2} \ar[r]^-{} \ar^{\simeq}[d]
&
\MF^\infty(y_{e^{v^\prime}_1} \ldots y_{e^v_{d+1}} y_{e^{v^\prime}_{d+2}} |_{\{ y_e \neq 0 \}}) \ar^{\simeq}[d] \\
\varphi^v_{d+2} |_{e^v_i} \colon
\Fuk(H_e)_{\bZ_2} \ar[r]^-{}
&
\IndCoh(Y^{d+2}_{v^\prime, d+1} |_{\{ y_e \neq 0 \}})_{\bZ_2} \ar[r]^-{}
&
\MF^\infty(y_{e^v_1} \cdots y_{e^v_{d+1}} y_{e^v_{d+2}} |_{\{ y_e \neq 0 \}}) \\
\varphi^v_{d+2} \colon
\Fuk(H_v)_{\bZ_2} \ar[r]^-{} \ar[u]
&
\IndCoh(Y^{d+2}_{v^\prime, d+1})_{\bZ_2} \ar[r]^-{} \ar_{\tilde{R}^B_{v, e}}[u]
&
\MF^\infty(\bA^{d+2}, y_{e^v_1} \cdots y_{e^v_{d+1}} y_{e^v_{d+2}}) \ar_{R^B_{v, e}}[u]
}
\end{gathered}
\end{align}
where the $B$-side middle vertical arrows are the canonical equivalences induced by
\pref{eq:transformation6},
\pref{eq:transformation7}.
Clearly,
it yields the commutative diagrams for further restrictions.

Suppose further that
$v^\prime$
is connected to another vertex
$v^{\prime \prime}$
by an edge
$e^{\prime}$.
Then
$\Sigma_Y$
gives the correspondences
\begin{align*}
(y_{e^{v^\prime}_1}, \ldots, y_{e^{v^\prime}_{d+2}})
\leftrightarrow
(y_{e^{v^{\prime \prime}}_1}, \ldots, y_{e^{v^{\prime \prime}}_{d+2}})
\end{align*}
of the coordinates on
$Y_{e^\prime}$
inducing that
\begin{align} \label{eq:transformation8}
(\tilde{y}_{e^{v^\prime}_1}, \ldots, \tilde{y}_{e^{v^\prime}_{d+1}})
\leftrightarrow
(\tilde{y}_{e^{v^{\prime \prime}}_1}, \ldots, \tilde{y}_{e^{v^{\prime \prime}}_{d+1}})
\end{align}
on
$\tilde{Y}_{e^\prime}$.
Since
$\varphi^{v^{\prime \prime}}_{d+2}$
is compatible with
$\varphi^{v^\prime}_{d+2}$,
we have the same commutative diagram as
\pref{eq:commute2}.
On the intersection
$Y_v \cap Y_{v^\prime} \cap Y_{v^{\prime \prime}}$,
the fan
$\Sigma_Y$
also gives the correspondences
\begin{align*}
(y_{e^v_1}, \ldots, y_{e^v_{d+2}})
\leftrightarrow
(y_{e^{v^{\prime \prime}}_1}, \ldots, y_{e^{v^{\prime \prime}}_{d+2}})
\end{align*}
of the coordinates inducing that
\begin{align} \label{eq:transformation9}
(\tilde{y}_{e^v_1}, \ldots, \tilde{y}_{e^v_{d+1}})
\leftrightarrow
(\tilde{y}_{e^{v^{\prime \prime}}_1}, \ldots, \tilde{y}_{e^{v^{\prime \prime}}_{d+1}})
\end{align}
on
$\tilde{Y}_v
\cap
\tilde{Y}_{v^\prime}
\cap
\tilde{Y}_{v^{\prime \prime}}$.
Since
$\varphi^{v^{\prime \prime}}_{d+2}$
is compatible with
$\varphi^v_{d+2}$
and
the affine pieces
$\tilde{Y}_v, \tilde{Y}_{v^\prime}, \tilde{Y}_{v^{\prime \prime}}$
glue to yield an open subvariety of
$\del T_\Sigma$,
the correspondences
\pref{eq:transformation9}
are compatible with
\pref{eq:transformation7}
and
\pref{eq:transformation8}.
Hence by Corollary
\pref{cor:key}
the equivalences
$\varphi^v_{d+2}, \varphi^{v^\prime}_{d+2}, \varphi^{v^{\prime \prime}}_{d+2}$
glue to yield an equivalence
\begin{align*}
\varphi^v_{d+2}
\cup
\varphi^{v^\prime}_{d+2}
\cup
\varphi^{v^{\prime \prime}}_{d+2}
\colon
\Fuk(H_v \cup H_{v^\prime} \cup H_{v^{\prime \prime}})_{\bZ_2}
\to
\cF_B(U_v \cup U_{v^\prime} \cup U_{v^{\prime \prime}})
\end{align*}
factorizing through
$\IndCoh(\tilde{Y}_v \cup \tilde{Y}_{v^\prime} \cup \tilde{Y}_{v^{\prime \prime}})_{\bZ_2}$.
Iteratively,
we obtain a compatible system
$\{ \varphi^{v}_{d+2} \}_{v \in \VVert(\Pi_\Sigma)}$
of Nadler's equivalences
whose gluing
\begin{align*}
\bigcup_{v \in \VVert(\Pi_\Sigma)} \varphi^{v}_{d+2}
\colon
\cF_A(\Pi_\Sigma)
\to
\cF_B(\Pi_\Sigma)
\end{align*}
gives the equivalence from
\pref{thm:main1}
factorizing through
$\IndCoh(\del T_\Sigma)_{\bZ_2}$. 

Summarizing,
from
\pref{lem:GS-A}
and
\pref{lem:GS-B}
we obtain

\begin{lemma} \label{lem:GS-HMS}
If the chosen adapted unimodular triangulation
$\cT$
is star-shaped,
then there is an equivalence
\begin{align*}
\cF_A(\Pi_\Sigma)
\simeq
\Fuk(H)_{\bZ_2}
=
\mSh^\diamondsuit_{\del_\infty \bL_\Sigma}(\del_\infty \bL_\Sigma)_{\bZ_2}
\simeq
\IndCoh(\del T_\Sigma)_{\bZ_2}
\simeq
\MF^\infty(Y, W_Y)
=
\cF_B(\Pi_\Sigma).
\end{align*}
\end{lemma}

\subsection{Nonunimodular case}
Finally,
we drop the assumption on
$\cT$
to be unimodular.
Note that
$\cT$
is still
adapted
and
star-shaped.
Then by
\cite[Lemma 2.5.3]{GS1}
the polytope
$\Delta^\vee_X, \Delta^\vee$
are the convex hulls of the stacky primitives of smooth quasiprojective stacky fans
$\Sigma_Y, \Sigma$.
Roughly speaking,
a stacky fan is a simplicial fan together with a choice of integer point along each ray.
See
\cite[Definition 2.4]{GS15}
for the definition. 
The homomorphism
$\beta \colon L \to N$
there corresponds to
$\bZ^{d+2} \to M^\vee \times \bZ,
\bZ^{d+1} \to M^\vee$
whose cokernels are the quotients by the stacky primitives.
Since
$\Coker \beta$
is finite,
$\beta$
induces a surjection of tori 
\begin{align*}
T_\beta
\colon
T_L
=
\Hom(\Hom(L, \bZ), \bC^*)
\to
T_N
=
\Hom(\Hom(N, \bZ), \bC^*).
\end{align*}
Let
$G(\Delta^\vee_X) \subset (\bC^*)^{\Sigma_Y(1)}, G(\Delta^\vee) \subset (\bC^*)^{\Sigma(1)}$
be the subgroups corresponding to
$\Ker \beta$.
Then the associated toric stacks are
$[\bA^{\Sigma_Y(1)} / G(\Delta^\vee_X)], [\bA^{\Sigma(1)} / G(\Delta^\vee)]$.

For each top dimensional cone
$\sigma \in \Sigma_Y$,
consider open substacks
$[\bA^{\sigma(1)} / G(\Delta^\vee_X)],
[\bA^{\bar{\sigma}(1)} / G(\Delta^\vee)]$
where
$\bA^{\sigma(1)}, \bA^{\bar{\sigma}(1)}$
are the open subvarieties determined by the rays
whose images span
$\sigma, \bar{\sigma}$.
Since the induced homomorphisms
$\bZ^{\sigma(1)}
\to
M^\vee \times \bZ,
\bZ^{\bar{\sigma}(1)}
\to
M^\vee$
have finite cokernels,
their dual yield the short exact sequences
\begin{align*}
0
\to
M \times \bZ
\to
\bZ^{\sigma(1)}
\to
G^\vee(\sigma)
\to
0, \
0
\to
M
\to
\bZ^{\bar{\sigma}(1)}
\to
G^\vee(\sigma)
\to
0.
\end{align*}
Taking
$\Hom(-, \bC^*)$,
we obtain the short exact sequences
\begin{align*}
0
\to
G(\sigma)
\to
(\bC^*)^{\sigma(1)}
\to
M^\vee_\bC \times \bC^*
\to
0, \
0
\to
G(\sigma)
\to
(\bC^*)^{\bar{\sigma}(1)}
\to
M^\vee_\bC
\to
0
\end{align*}
where
$G(\sigma)$
is the restriction
$G(\Delta^\vee_X) \cap (\bC^*)^{\sigma(1)}
=
G(\Delta^\vee) \cap (\bC^*)^{\bar{\sigma}(1)}$.
The inclusion
$M \times \bZ
\hookrightarrow
\bZ^{\sigma(1)},
M
\to
\bZ^{\bar{\sigma}(1)}$
induce finite covers
\begin{align*}
h^\vee(\sigma)
\colon
T^* (\bR^{d+2} / \bZ^{\sigma(1)})
\to
T^* (\bR^{d+2} / M \times \bZ), \
h^\vee(\bar{\sigma})
\colon
T^* (\bR^{d+1} / \bZ^{\bar{\sigma}(1)})
\to
T^* (\bR^{d+1} / M),
\end{align*}
which are the quotients by
$G^\vee(\sigma)$.

Applying the same argument as in Section
$8.5$,
we obtain
\begin{align} \label{eq:preGS}
\Fuk(H(\Delta^\vee))_{\bZ_2}
\simeq
\MF^\infty([\bA^{\Sigma_Y(1)} / G(\Delta^\vee_X)], W_Y),
\end{align}
where
$H(\Delta^\vee)$
is the gluing of the inverse images of
$\tilde{P}_d$
under
$h^\vee(\bar{\sigma})$.
Note that
the $G(\Delta^\vee_X)$-action preserves
$W_Y$
as it is defined by the height function.
Since the argument in the previous section is compatible with the actions of
$G(\sigma), G(\bar{\sigma})$,
the equivalence
\pref{eq:preGS}
factors through
$\IndCoh(\del T_\Sigma(\Delta^\vee))_{\bZ_2}$,
where
$\del T_\Sigma(\Delta^\vee)$
is the toric boundary divisor of
$[\bA^{\Sigma(1)} / G(\Delta^\vee)]$.

Summarizing,
we obtain

\begin{theorem}[cf. {\cite[Theorem 1.0.1]{GS1}}] 
Let
$\cT$
be an adapted star-shaped triangulation of
$\Delta^\vee$.
Then there is an equivalence
\begin{align*}
\Fuk(H(\Delta^\vee))_{\bZ_2}
\simeq
\mSh^\diamondsuit_{\del_\infty \bL_\Sigma(\Delta^\vee)}(\del_\infty \bL_\Sigma(\Delta^\vee))_{\bZ_2}
\simeq
\IndCoh(\del T_\Sigma(\Delta^\vee))_{\bZ_2}
\simeq
\MF^\infty([\bA^{\Sigma_Y(1)} / G(\Delta^\vee_X)], W_Y),
\end{align*}
where
$\del_\infty \bL_\Sigma(\Delta^\vee)$
is the corresponding finite cover of
$\del_\infty \bL_\Sigma$.
\end{theorem}

\begin{remark}
In
\cite{GS1}
the authors also showed the compatibility of the equivalence with coherent--constructible correspondence
\cite{Kuw}
\begin{align*}
\Fuk(T^* (\bR^{d+1} / \bZ^{\bar{\sigma}(1)}), H(\Delta^\vee))
=
\Sh_{\bL_\Sigma(\Delta^\vee)}(\bR^{d+1} / \bZ^{\bar{\sigma}(1)})
\simeq
\Coh(T_\Sigma(\Delta^\vee)).
\end{align*}
\end{remark}

\appendix
\section{Local-to-Global on A-side}
The main result of the Appendix is a proof of the locality of the Fukaya category in the complete intersection setting. More precisely, if
$Z_t$
is a tropical family of smooth complete intersections in
$(\bC^*)^n$,
we will compute its Fukaya category as the global sections of a sheaf of Fukaya categories over the tropical 
base
$Z^{trop}$.
This statement has been already proved in the main text as Theorem
\ref{thm:global}
and
\ref{thm:global2},
albeit with a different and somewhat more ad hoc argument,
so it is useful to clarify in
which way this Appendix complements that material. 

First of all, the proofs of Theorem
\ref{thm:global}
and
\ref{thm:global2}
actually depend on the results of this Appendix.
Indeed,
they rely on the fact
that
the separatedness assumption in
\pref{lem:key2}
and
Corollary
\pref{cor:key}
hold in our setting.
At least in the hypersurface case,
this follows from the discussion in
\pref{eg:key},
which describes the geometry of the Reeb flow
when we glue  together two pairs-of-pants.
We believe that
similar arguments can be made to work in the complete intersection setting, however we do not pursue this line of argument in the current article.
Instead,
in this Appendix,
we shall show that
the separatedness assumption holds using the technique of semi-tropicalization,
which was introduced in a previous paper by the third author
\cite{Zho}. 

Additionally,
in this Appendix we will explain a general approach to construct Weinstein structures adapted to a gluing of pants (or more generally covers of products of pants). 
Namely,
we shall show how to explicitly engineer global Weinstein structures
that have the required local properties
by considering appropriate potentials in the tropical base.
This approach is more flexible than the one explained in the proof of 
\pref{thm:global},
where we work with the standard Weinstein structure in the ambient torus (up to translation).
We believe that
this clarifies our argument,
and
that
it might be of independent interest
as it should be applicable in more general contexts than the one considered in the present paper.

\subsection{Setup}
We recall the setup from before. 
Fix positive integers
$n$
and
$r \leq n$.
Consider
$r$
subsets
$A_1, \cdots, A_r \subset \bZ^n$.
Fix the following functions
$$ \rho_i \colon A_i \to \bR. $$
For any fixed
$t>1$
we consider the following Laurent polynomials on
$(\bC^*)^n$
$$ W_{i,t}(z) = \sum_{a \in A_i}   c_i(a) t^{-\rho_i(a)} z^a $$
where
$z^a = z_1^{a_1} \cdots z_n^{a_n}$
is a monomial on
$(\bC^*)^n$
and
$c_i(a) \in \bC^*$.
Let
$Z_t = \bigcap_{i=1}^r H_{i,t}$
be the complete intersection of
$H_{i,t} = W^{-1}_{i, t}(0)$. 

We also define the tropical version of the above data. 
Recall the piecewise linear function
$$ W_i^{trop} \colon \bR^n \to \bR, \
x \mapsto \max_{a \in A_i} \{ \langle x, a \rangle - \rho_i(a) \}. $$
Let 
$Z^{trop} = \bigcap_i H_i^{trop}$
be the complete intersection of the corner loci
$H_i^{trop}$
of
$W_i^{trop}$. 

Consider the rescaled log-modulus map
$$ \Log_t \colon \bC^* \to \bR, \quad z \mapsto \frac{\log |z|} {\log t}, $$
and
extend naturally to
$(\bC^*)^n \to \bR^n$
componentwise. 
We have Gromov-Hausdorff convergence of subsets in
$\bR^n$
\cite{Mik}
$$ \lim_{t \to \infty} \Log_t (H_{i,t}) = H_i^{trop}, \quad \lim_{t \to \infty} \Log_t (Z_{t}) = Z^{trop}. $$
Since
$H_i^{trop}$
intersect transversely,
$H_{i,t}$
intersect transversely
and
$Z_t$
is smooth for all
$t \gg 1$. 

We will often consider a subset
$P \subset \bR^n$
and
the preimage
$\Log_t^{-1}(P)$.
We denote
$$ Z_t |_P \coloneqq \Log_t^{-1}(P) \cap Z_t. $$

For each
$W_i^{trop}$
we have a piecewise constant function
$$  \bR^n \to \{\text{subsets of $A_i$}\},\quad x \mapsto A_i(x) = \{ a\in A_i \mid W_i^{trop}(x) = \langle x, a \rangle - \rho_i(a) \}. $$
The level set of such a function defines a stratification
$\cS_i$
of
$\bR^n$,
where for each strata
$\tau_i$
we write
$A_{i,\tau_i}$
for the value of the function. 

The intersection of
$\cS_i$
is a common refinement stratification
$\cS$,
and
a strata $ \tau \in \cS$
is denoted as
$(\tau_1, \cdots, \tau_r)$
for
$\tau_i \in \cS_i$.
The subset
$Z^{trop} \subset \bR^n$
has the induced stratification
$\cS_Z$
of
$\cS$.

\begin{definition}
For each
$\tau_i \in \cS_i$
we have a truncated local defining function
$$ W_{i,t,\tau_i}(z) = \sum_{a \in A_{i,\tau_i}} c_i(a) t^{- \rho_i(a)} z^a $$
and
call
$H_{i,t,\tau_i} = W_{i,t,\tau_i}^{-1}(0)$
the truncated hypersurface. 
\end{definition}

For any strata
$\tau =(\tau_1, \cdots, \tau_r) \in \cS_Z$
we have
$$ Z_{t,\tau} = \bigcap_{i=1}^r H_{i,t,\tau_i}. $$

\subsection{Semi-Tropicalization}
We first discuss the semi-tropicalization of a hypersurface. 
Let
$H_t \subset (\bC^*)^n$
be the zero set of a tropicalized Laurent polynomial
$$ W_t(z) = \sum_{a \in A} c(a) t^{-\rho(a)} z^a$$
with
monomial set
$A \subset \bZ^n$
and
height function
$\rho$.
Assume that
for each
$a \in A$
there is a region
$C_a \subset \bR^n$
where
$W^{trop}(x) = \langle x, a \rangle - \rho(a)$
holds. 

Let
$\cS$
be the natural stratification on
$\bR^n$
given by
$W^{trop}$.
For each strata
$\tau \in \cS$
let 
$$ A_\tau = \{a \in A \mid W^{trop}(x) = \langle x, a \rangle - \rho(a) \}.$$  

Consider a canonical distance function from a point
$x$
to the cell
$C_a$
defined as
$$ d(x,C_a) = W^{trop}(x) - ( \langle x, a \rangle - \rho(a) ). $$

Let
$\epsilon_t  = 1 / \log t$.
Then as 
$t \to \infty$,
we have
$\epsilon_t \to 0$. 

\begin{remark}
The ``thickness'' of the rescaled amoeba
$\Log_t(H_{t})$
depends on the order of
$\epsilon_t$.
We choose the cut-off function on the scale
$\epsilon^{1/2}_t \gg \epsilon_t$,
which will guarantee that
the semi-tropicalized hypersurface sufficiently close to the original one.
See
\cite{Zho}
for more details about the semi-tropicalization procedure. 
\end{remark}

Let
$\chi(r)$
be a smooth cut-off function on
$\bR_{\geq 0}$
with
$\chi'(r) \leq 0$,
$\chi|_{[0,1/2]}=1$
and
$\chi_{[1,\infty)}=0$.
Define
$$\chi_{a,t}(x) = \chi\left(\frac{d(x,C_a)}{\epsilon_t^{1/2}}\right).$$

For any
$\delta \in (0, 1)$
and
$a \in A$
let 
$$ U^\delta_{a,t} = \{d (x, C_a) < \epsilon_t^\delta\}.$$
For any strata
$\tau \in \cS$
define
$$ U^\delta_{\tau,t}
=
\bigcap_{a \in A_\tau} U^\delta_{a,t}, \quad
V^\delta_{\tau,t}
=
U^\delta_{\tau,t}
\cap
\left(\bigcap_{a \in A \setminus A_\tau} (U^\delta_{a,t})^{c}\right). $$
If
$\delta = 1/2$,
then we drop the superscript
$\delta$. 

The subset
$U_{\tau,t} \subset \bR^n$
is an open neighborhood of
$\tau$.
And
$V_{\tau,t}$
forms a partition of
$\bR^n$
in the sense that 
$$ \bR^n = \bigsqcup_{\tau \in \cS} V_{\tau,t}.$$

\begin{definition}
We define the
\emph{semi-tropicalized function}
$$ W^{\chi}_t(z)
=
\sum_{a \in A} \chi_{a,t}(\Log_t(z))  c(a) t^{-\rho(a)} z^a. $$
and call 
$H^\chi_t = (W^{\chi}_t)^{-1}(0)$
the semi-tropicalized hypersurface. 
\end{definition}

There is a linear interpolation between
the original function
and
the semi-tropicalization
$$ W^{\chi,s}_t(z) = (1-s) W_t(z) + s W^{\chi}_t(z). $$.
We denote by
$H^{\chi,s}_t$
the hypersurface
$(W^{\chi,s}_t)^{-1}(0)$. 

Finally,
we return to complete intersection.
We define
$$ Z^\chi_t = \bigcap_{i=1}^r H_{i,t}^\chi. $$
For a strata
$\tau=(\tau_1,\cdots, \tau_r) \in \cS$
the subsets
$U_{\tau,t}, V_{\tau,t}$
are defined by intersections of
$U_{\tau_i,t}, V_{\tau_i,t}$
in the obvious way. 

\subsection{Convex domain and Convex function}
By a convex potential on
$\bR^n$,
we will mean a  strictly convex smooth proper function 
$$ \Phi \colon \bR^n \to \bR.$$ 
Fix
any convex potential
$\Phi$
and
any
$t \gg 1$.
The pullback
$$ \Log_t^*(\Phi) \colon (\bC^*)^n \to \bR$$
defines a K\"{a}hler potential,
with K\"{a}hler form
$\omega_{\Phi,t}$
and
Liouville 1-form
$\lambda_{\Phi,t}$
on
$(\bC^*)^n$
respectively given by 
$$ \omega_{\Phi,t} = dd^c \Log_t^*(\Phi), \quad \lambda_{\Phi,t} = d^c \Log_t^*(\Phi).$$ 
Then for any
$c \in \bR$
the subset 
$$ ((\bC^*)^n)_{t,\Phi \leq c}
=
\Log_t^{-1}((\bR^n)_{\Phi \leq c}), \quad \text{where} \ \ (\bR^n)_{\Phi \leq c}
=
\{x \in \bR^n \mid \Phi(x) \leq c \} $$
is a Weinstein subdomain of
$((\bC^*)^n, \omega_{\Phi,t}, \lambda_{\Phi,t})$. 

The K\"{a}hler structure on
$(\bC^*)^n$
restricts to any complex submanifold,
in particular to the complete intersection
$Z_t$.
For any
$c \in \bR$
the sublevel set 
$$ Z_{t, \Phi \leq c} \coloneqq Z_t \cap ((\bC^*)^n)_{t,\Phi \leq c} $$
is also a Weinstein subdomain of
$Z_t$. 

If the level set
$\{\Phi = c\}$
intersects
$Z^{trop}$
transversely,
then for large enough
$t$,
the semi-tropicalized complete intersection
$$ Z^\chi_{t, \Phi \leq c} \coloneqq Z^\chi_t \cap ((\bC^*)^n)_{t,\Phi \leq c} $$
is also a Weinstein subdomain of
$Z^\chi_t$,
and
is homotopic to
$Z_{t, \Phi \leq c}$.

\subsection{Local Tropical Moment map}
Recall that we have the stratification
$\cS_i$
of
$\bR^n$
coming from each tropoical hypersurface
$H_i^{trop}$,
and
we have their common refinement
$\cS$.
The tropical intersection
$Z^{trop} \subset \bR^n$
inherits a stratification
$\cS_Z$. 

Let
$\Phi \colon \bR^n \to \bR$
be a convex potential.
We call its Legendre transofrmation
$\cL_\Phi$
the
\emph{tropical moment map},
which is defined as  
$$ \cL_\Phi \colon \bR^n \to (\bR^n)^\vee, \quad x \mapsto d\Phi|_x. $$
Since
$\Phi$
is strictly convex,
$\cL_\Phi$
is injective.
The moment map of the Hamiltonian action of
$U(1)^n$
on
$(\bC^*)^n$
with symplectic form
$\omega_{\Phi,t}$
is the composition (up to scaling by a constant) 
$$ \cL_\Phi \circ \Log_t \colon (\bC^*)^n \to (\bR^n)^\vee. $$

For any linear subspace
$V \hookrightarrow  \bR^n$,
we have
$\pi_V \colon (\bR^n)^\vee \twoheadrightarrow V^\vee$.
We define 
$$ \mu_V = \pi_V \circ \cL_\Phi \colon \bR^n \to V^\vee. $$

For any strata
$\tau$
let
$(\bR^n)_\tau \subset \bR^n$
be the linear subspace tangent to
$\tau$,
and
let 
$$ (\bZ^n)_\tau = (\bR^n)_\tau \cap \bZ^n.$$ 
Define 
$$ \mu_\tau
\coloneqq
\mu_{V_\tau}
\colon
\bR^n
\to
(\bR^n)^\vee_\tau. $$

\begin{lemma}
The restriction
$\mu_\tau|_\tau \colon \tau \to (\bR^n)_\tau^\vee$
is injective.
\end{lemma}
\begin{proof}
The restriction
$\mu_\tau|_\tau$
is the Legendre transformation of
$\Phi|_\tau \colon \tau \to \bR$,
which is strictly convex. 
\end{proof}

For a strata
$\tau=(\tau_1,\cdots,\tau_r) \in \cS$
we have the star open neighborhood 
$$ U_{\geq \tau} = \bigcup_{\sigma \geq \tau} \sigma. $$
Here we introduce a slightly shrunk version of the star neighborhood
$$ V_{\geq \tau, t} = \bigcup_{\sigma \geq \tau} V_{\sigma,t}. $$
It has the property that \footnote{Given a family
$\{ V_t \}_{(R, \infty)}$
of sets,
we define
$$ \limsup_{t \to \infty} V_t \coloneqq \bigcap_{s > R} \bigcup_{s < t} V_t, \quad \liminf_{t \to \infty} V_t \coloneqq \bigcup_{s > R} \bigcap_{s < t} V_t. $$
The limit exists if $\liminf_{t \to \infty} V_t =\limsup_{t \to \infty} V_t$. }
$$ V_{\geq \tau, t} \subset U_{\geq \tau}, \quad \lim_{t \to \infty} V_{\geq \tau, t} = U_{\geq \tau}.$$

\begin{proposition}\label{prop:lmom}
For any strata
$\tau \in \cS$
the local semi-tropicalized complete intersection
$$ Z^\chi_{t, \geq \tau} \coloneqq Z^\chi_t|_{V_{\geq \tau, t}}$$
is invariant under the Hamiltonian torus action by 
$$ U(1)_\tau^n \coloneqq (\bZ^n)_\tau \otimes_\bZ U(1), $$
and the moment map is
$$ \mu^\chi_{t, \geq \tau}
\colon
Z^\chi_{t, \geq \tau}
\to 
(\bR^n)_\tau^\vee, \quad z
\mapsto
\mu_\tau(\Log_t(z)). $$
\end{proposition}
\begin{proof}
It suffices to show the hypersurface case,
as we have
$U(1)^n_\tau = \bigcap_i U(1)^n_{\tau_i}$
and
$V_{\geq \tau,t} = \bigcap_i V_{\geq \tau_i, t}$.
Note that the defining equation of the hypersurface over
$V_{\geq \tau_i, t}$
is
$$ \sum_{a \in A_{i,\tau_i}} \chi_{a,t}(\Log_t(z)) c_i(a) t^{-\rho_i(a)} z^a = 0 $$
Under the action of
$U(1)_{\tau_i}^n$,
$\Log_t(z)$
and
hence the cut-off factor
$\chi_{a,t}(\Log_t(z))$
is invariant.
The monomial
$z^a$
will change by a factor of
$\lambda^a$
for
$\lambda \in U(1)_{\tau_i}^n \subset U(1)^n \subset (\bC^*)^n$.
However,
for
$a, b \in A_{i,\tau_i} \subset \bZ^n$
we have
$a-b \in \tau_i^\perp$
and
$\lambda^a = \lambda^b$.
Hence all the terms in the local defining equation will acquire the same phase under the action by
$\lambda \in U(1)_{\tau_i}^n$,
which implies the invariance of the equation. 
\end{proof}

\subsection{Transverse Foliation and locality of Hamiltonian flow}
Let
$\Phi \colon \bR^n \to \bR$
be a convex potential.
Let
$\cS$
be a piecewise linear stratification of
$\bR^n$.
For each strata
$\tau$
let
$\mu_\tau \colon \bR^n \to (\bR^n)^\vee_\tau$
be the local tropical moment map.
Recall that
$V_{\tau,t}$
is a ``truncated version'' of a tubular neighborhood of
$\tau$
and
$V_{\geq \tau, t} = \cup_{\sigma \geq \tau} V_{\sigma,t}$
is a shrunk version of the star neighborhood of
$\tau$. 

For each point
$x \in \tau$
consider the following subsets
$$ S_x = \mu_\tau^{-1}(\mu_\tau(x)), \quad
S_{x,\tau, t} = S_x \cap U_{\tau,t} , \quad
S_{x, \geq \tau, t} = S_x \cap V_{\geq \tau,t}. $$

\begin{definition}
A
\emph{system of transverse foliation}
is a collection
$\{ (B_{\tau, t}, T_{\tau, t}, \pi_{\tau, t}) \}_{\tau \in \cS}$
of triples
where for each
$\tau$
the triple consists of
an open subset
$B_{\tau, t} \subset \tau$,
a tubular neighborhood
$T_{\tau, t}$
of
$B_{\tau, t}$
and
a map
$\pi_{\tau, t} \colon T_{\tau, t} \to B_{\tau, t}$.
These data must satisfy the following conditions:
\begin{enumerate}
\item
For any
$x \in B_{\tau, t}$ 
\begin{align*}
S_{x, \tau,t}
\subset
T_{x, \tau, t}
\coloneqq
\pi_{\tau, t}^{-1}(x)
\subset
S_{x, \geq \tau, t}.
\end{align*}
\item
For each
$\tau \in \cS$
we have
$\bar{\tau}
\subset 
T_{\leq \tau, t}
\coloneqq
\bigcup_{\sigma \leq \tau} T_{\sigma, t}$. 
\item
If
${T_{x, \tau, t}} \cap T_{x^\prime, \tau^\prime, t} \neq \emptyset$
for
$\tau^\prime > \tau$
and
$x \in B_{\tau, t}, x^\prime \in B_{\tau^\prime, t}$,
then
$T_{x^\prime, \tau^\prime, t} \subset {T_{x, \tau, t}}$.  
\item
$\lim_{t \to \infty} T_{\tau, t} = \tau$. 
\end{enumerate}
\end{definition}

\begin{proposition}
For any
$t \gg 0$
and
convex potential
$\Phi \colon \bR^n \to \bR$
there exists a system
$\{ (B_{\tau, t}, T_{\tau, t}, \pi_{\tau, t}) \}_{\tau \in \cS}$
of transverse foliation.
\end{proposition}
\begin{proof}
Fix 
$0 < \delta_1 < \delta_2 < 1/2$.
Then we have
\begin{align*}
U^{\delta_1}_{\tau, t}
\supset
U^{\delta_2}_{\tau, t}
\supset
U^{1/2}_{\tau, t}
=
U_{\tau, t}.
\end{align*}
For each
$\tau \in \cS$
we set
\begin{align*}
B_{\tau, t}
\coloneqq
\tau \setminus \bigcup_{\sigma < \tau} \bar{U}^{\delta_2}_{\sigma, t}.
\end{align*}
We claim that
for all
$\tau \in \cS, x \in B_{\tau, t}$
we have
\begin{align*}
S_{x, \tau, t}
\subset
V_{\tau, t}
\end{align*}
when
$t$
is sufficiently large.
Indeed,
for any fixed Euclidean distance on
$\bR^n$
we have
\begin{align*}
\epsilon^{1/2}_t \ll \epsilon^{\delta_2}_t, \
\sup_{x \in \tau} diam(S_{x, \tau, t})
=
O(\epsilon^{1/2}_t), \
dist(U_{\tau, t} \setminus V_{\tau, t}, B_{\tau, t})
=
O(\epsilon^{\delta_2}_t)
\end{align*}
as
$t \to \infty$
and
$\epsilon_t = 1 / \log t \to 0$.
This implies
\begin{align*}
S_{x, \tau, t} \cap (U_{\tau, t} \setminus V_{\tau, t})
=
\emptyset
\end{align*}
for all
$\tau \in \cS, x \in B_{\tau, t}$.
From
$S_{x, \tau, t} \subset U_{\tau, t}$
and
$V_{\tau, t} \subset U_{\tau, t}$
it follows
$S_{x, \tau, t} \subset V_{\tau, t}$.

Assume that
$t$
is sufficiently large for the above condition.
Then we define
\begin{align*}
T^{pre}_{\tau, t}
\coloneqq
V_{\geq \tau, t}
\cap
U^{\delta_1}_{\tau, t}
\cap 
\mu^{-1}_\tau(\mu_\tau(B_{\tau, t})), \
\pi^{pre}_{\tau, t}
\coloneqq
(\mu_\tau |_{B_{\tau, t}})^{-1} \circ \mu_\tau
\colon
T^{pre}_{\tau, t}
\to
B_{\tau, t}.
\end{align*}
By
construction
and
$S_{x, \tau, t} \subset V_{\tau, t}$
we have
\begin{align*}
S_{x, \tau, t}
\subset
T_{x, \tau, t}
=
(\pi^{pre}_{\tau, t})^{-1}(x)
=
V_{\geq \tau, t} \cap U^{\delta_1}_{\tau, t} \cap S_x 
\subset
S_{x, \geq \tau, t}
\end{align*}
for all
$\tau \in \cS, x \in B_{\tau, t}$.
From
$\epsilon^{\delta_2}_t \ll \epsilon^{\delta_1}_t$
it follows
\begin{align*}
\bar{\tau} \subset T^{pre}_{\leq \tau, t}
=
\bigcup_{\sigma \leq \tau} T^{pre}_{\sigma, t}.
\end{align*}
Hence
$\{ (B_{\tau, t}, T^{pre}_{\tau, t}, \pi^{pre}_{\tau, t} ) \}_{\tau \in \cS}$
satisfies condition
(1)
and
(2).
As for condition
(4),
we have
\begin{align*}
B_{\tau, t} \subset T^{pre}_{\tau, t} \subset U^{\delta_1}_{\tau, t}, \
\lim_{t \to \infty} B_{\tau, t}
=
\lim_{t \to \infty} U^{\delta_1}_{\tau, t}
=
\tau.
\end{align*}

Now,
we correct
$\{ (B_{\tau, t}, T^{pre}_{\tau, t}, \pi^{pre}_{\tau, t} ) \}_{\tau \in \cS}$
for condition
(3).
If
$\tau$
is top dimensional,
then condition
(3)
is null
and
we set
$T_{x, \tau, t} \coloneqq T^{pre}_{x, \tau, t}$.
Suppose that
for all
$\sigma \in \cS$
with
$\dim \sigma > \dim \tau$
we have chosen
$\pi_{\sigma, t} \colon T_{\sigma, t} \to B_{\sigma, t}$
satisfying
\begin{align*}
T^{pre}_{x, \sigma, t}
\subset
T_{x, \sigma, t}
\subset
S_{x, \geq \sigma, t}
\end{align*}
and
condition
(3).
Then we set
\begin{align*}
T_{x, \tau, t}
\coloneqq
T^{pre}_{x, \tau, t}
\cup
\bigcup_{\sigma > \tau} \pi^{-1}_{\sigma, t} (\pi_{\sigma, t}(T^{pre}_{x, \tau, t} \cap T_{\sigma, t})).
\end{align*}
Clearly,
$T_{x, \tau, t}$
satisfies condition
(2), (3)
and
(4).
From
$V_{\geq \sigma, t} \subset V_{\geq \tau, t}$
and
$\pi^{pre}_{\tau, t} \circ \pi_{\sigma, t} = \pi^{pre}_{\tau, t}$
on
$V_{\geq \sigma, t}$
we obtain
\begin{align*}
\pi^{-1}_{\sigma, t} (\pi_{\sigma, t}(T^{pre}_{x, \tau, t} \cap T_{\sigma, t}))
\subset
S_{x, \tau, t}.
\end{align*}
Hence the corrected fiber
$T_{x, \tau, t}$
still satisfies condition
(1).
\end{proof}

By a
\emph{base-Hamiltonian}
we will mean any smooth function
$g \colon \bR^n \to \bR$.
It induces
a Hamiltonian function
\begin{align*}
G^\chi_t
=
\Log^*_t(g) |_{Z^\chi_t}
\colon
Z^\chi_t
\to
\bR
\end{align*}
and
Hamiltonian vector field
$X_{G^\chi_t}$
with respect to
$\omega_{\Phi, t}$.
Recall from Proposition
\pref{prop:lmom}
that
$Z^\chi_t$
locally respects the Hamiltonian torus actions of
$U(1)^n_\tau$.
There actions preserve
$G^\chi_t$
and
$X_{G^\chi_t}$
is compatible with the local moment maps
$\mu_{t, \geq \tau}$.
We write
$z(u)$
for the time
$u \in \bR$
trajectory of any point
$z = z(0) \in Z^\chi_t$
under the Hamiltonian flow.

\begin{proposition} \label{prop:trap}
For any
$t \gg 0$
and
convex potential
$\Phi$
let
$\{ (B_{\tau, t}, T_{\tau, t}, \pi_{\tau, t}) \}_{\tau \in \cS}$
be a system of transverse foliation. 
Let
$g: \bR^n \to \bR$
be a base-Hamiltonian,
$G^\chi_t$
the induced function on
$Z^\chi_t$
and
$X_{G^\chi_t}$
the Hamiltonian vector field with respect to
$\omega_{\Phi,t}$.
Then the Hamiltonian trajectory respects the transverse foliation in the sense that,
if
$\Log_t(z)  \in T_{x, \tau, t}$
for some
$\tau \in \cS_Z$,
then
$\Log_t(z(u)) \in T_{x, \tau, t}$
for all
$u \in \bR$.
\end{proposition}
\begin{proof}
If
$\tau \in \cS_Z$
is top dimensional,
then
$Z^\chi_t |_{T_{\tau, t}}$
is a real $(n-r)$-dimensional torus fibration over
$B_{\tau, t}$.
Then the trajectory of
$z$
under the Hamiltonian flow is contained in the torus fiber.
Suppose that
the claim holds for all strata with dimension more than
$\dim \tau$.
The Hamiltonian flow of
$X_{G^\chi_t}$
is tangent to the fiber of the local moment map, hence the flow trajectory cannot exit the leaf 
at any interior point.
Since the boundary
$\del T_{m, \tau, t}$
of a leaf
$T_{m, \tau, t}$
is foliated by leaves transverse to strata
$\sigma > \tau$,
by induction hypothesis the boundary is closed under the Hamiltonian flow.
Thus the Hamiltonian trajectory respects the transverse foliation.
\end{proof}

\subsection{Localization of Fukaya category}
Let
$\Open(\bR^n)$
be the category of open subsets of
$\bR^n$.
We denote by
$\Open_{conv}(\bR^n)$
the full subcategory of convex open subsets.
Recall the category
$\Open(\cS_Z)$
of open subsets of
$\cS_Z$
with respect to the topology from Definition
\pref{dfn:topology}.
There is an obvious ``saturation'' functor
\begin{align*}
\lceil - \rceil
\colon
\Open(\bR^n)
\to
\Open(\cS_Z), \
P
\mapsto
\{ \tau \in \cS_Z \ | \ \tau \cap P \neq \emptyset\}.
\end{align*}
We denote by
$\Open_{conv}(\cS_Z)$
its image of
$\Open_{conv}(\bR^n)$.
For
$P_1, P_2 \in \Open_{conv}(\bR^n)$
we write
$P_1 \sim P_2$
if
$\lceil P_1 \rceil = \lceil P_2 \rceil$.
Clearly,
this defines an equivalence relation.
Let
$| \lceil P \rceil |$
be the set of open subsets of
$\bR^n$
equivalent to
$P$.
We equip
$| \lceil P \rceil |$
with the natural topology coming from the Gromov--Hausdorff metric.

\begin{lemma} \label{lem:contractible}
For any
$P \in \Open_{conv}(\bR^n)$
the space
$| \lceil P \rceil |$
is contractible.
If
$P_1, P_2 \in | \lceil P \rceil |$
then there is a finite sequence
\begin{align*}
P_1
\hookleftarrow
Q_1
\hookrightarrow
Q_{12}
\hookleftarrow
Q_2
\hookrightarrow
P_2
\end{align*}
of zig-zag inclusions in
$| \lceil P \rceil |$.
\end{lemma}
\begin{proof}
We construct a deformation retract from
$| \lceil P \rceil |$
to its contractible subset.
Let
$\Pi_{\lceil P \rceil} = \Pi_{\tau \in \min(\lceil P \rceil)}$
where
$\min(\lceil P \rceil) \subset \cS_Z\cap P$
is the subset of minimal elements.
Fix a Euclidean metric on
$\bR^n$.
For each point
$p = (p_\tau)_\tau \in \Pi_{\lceil P \rceil}$
we
form the convex hull
$\Delta_p = \Conv(\{ p_\tau \}_\tau)$
and 
choose sufficiently small
$\epsilon(p)$
so that
for all
$0 < r \leq \epsilon(p)$
we have
\begin{align*}
\lceil P \rceil
=
\lceil U_r(\Delta_p) \rceil, \
U_r(\Delta_p)
=
\{ x \subset \bR^n \ | \ dist(x, \Delta_p) < r \}
\in
\Open_{conv}(\bR^n).
\end{align*}
One can choose
$\epsilon(p)$
to depend on
$p$
continuously.
Hence there is an inclusion
\begin{align*}
\widetilde{\Pi}_{\lceil P \rceil}
=
\Pi_{\lceil P \rceil} \times \bR_{0 < r \leq \epsilon(p)}
\to
| \lceil P \rceil |, \
(p, r)
\mapsto
U_r(\Delta_p)
\end{align*} 
of a contractible set
$\widetilde{\Pi}_{\lceil P \rceil}$.
For any
$P \in | \lceil P \rceil |$
and
$\tau \in \min(\lceil P \rceil)$
let
$b_{P, \tau}$
be the barycenter of
$\overline{\tau \cap P}$.
We denote by
$p(P)$
the point
$(b_{\tau, P})_\tau \in \Pi_{\lceil P \rceil}$.
Let
\begin{align*}
r(P)
=
\sup \{ r \in (0, \epsilon(p(P))] \ | \ U_r(\Delta_{p(P)}) \subset P \}.
\end{align*}
Then there is a canonical isotopy from
$P$
to
$U_{r(P)}(\Delta_{p(P)})$,
which defines a deformation retract from
$\widetilde{\Pi}_{\lceil P \rceil}$.
Finally,
the desired zig-zag sequence is obtained by setting
\begin{align*}
Q_1
\coloneqq
U_r(\Delta_{p(P_1)})
\subset
P_1, \
Q_2
\coloneqq
U_r(\Delta_{p(P_2)})
\subset
P_2, \
Q_{12}
\coloneqq
U_r(\Conv(\Delta_{p(P_1)}, \Delta_{p(P_2)}))
\in
| \lceil P \rceil |
\end{align*}
for sufficiently small
$r$.
\end{proof}

For any
$P \in \Open_{conv}(\bR^n)$
let
$\Phi$
be a convex potential on
$P$
whose limit on the boundary
$\del P$
is zero.
Take
$\epsilon > 0$
sufficiently large
so that
$P_\epsilon \sim P$
for
$P_\epsilon
=
\{ x \subset \bR^n \ | \ \Phi(x) < - \epsilon \}$
is equivalent to
$P$.
Fix
$t_0 \in \bR$
such that
\begin{align*}
(H_t |_{P_\epsilon} = H_t \cap \Log^{-1}_t(P_\epsilon), \
\lambda_{\Phi, \epsilon, t} = d^c (\Log^*_t(\Phi) |_{H_t \cap \Log^{-1}_t(P_\epsilon)})
\end{align*}
are Liouville homotopic for all
$t > t_0$.
We set
\begin{align*}
\Fuk(P, \Phi, \epsilon, t)
\coloneqq
\Fuk(H_t |_{P_\epsilon}, \lambda_{\Phi, \epsilon, t}).
\end{align*}
the ind-completion of the wrapped Fukaya category. 
Up to canonical equivalence,
$\Fuk(P, \Phi, \epsilon, t)$
is independent of the admissible choices
$\Phi, \epsilon, t$
and
the space of such choices is contractible.
Hence we will fix a choice
and
write
$\Fuk(P)$
for short,
with the understanding that
one can switch choices at the cost of a canonical equivalence.
Note that
the Viterbo restriction
\begin{align*}
\Fuk(P^\prime) \to \Fuk(P)
\end{align*}
along a morphism
$P \to P^\prime$
in
$\Open_{conv}(\bR^n)$
becomes an equivalence
when
$P \sim P^\prime$.
By
\pref{lem:contractible},
up to canonical equivalence,
$\Fuk(P)$
only depends on
$\lceil P \rceil$.
Hence we may write
$\Fuk(\lceil P \rceil)$
for
$\Fuk(P)$.

\begin{remark}
If we fix a holomorphic volume form
$\Omega_{(\bC^*)^n}$,
e.g.
$\wedge_i d \log z_i$
on
$(\bC^*)^n$,
then we have a choice of holomorphic volume form on
$Z_t$
as
$$ \Omega_{Z_t}=\frac{\Omega_{(\bC^*)^n} }{\wedge_{i=1}^r dW_{i,t}}. $$
This will give a $\bZ$-grading on
$\Fuk(Z_t), \Fuk(Z^\chi_t)$
and the corresponding various
$\Fuk(P)$. 
\end{remark}

\begin{definition}
We define the presheaf
$\cF^{pre, conv}_A$
as the functor
\begin{align*}
\Open_{conv}(\cS_Z)^{op}
\to
^{**}\DG, \
\lceil P \rceil
\mapsto
\Fuk(\lceil P \rceil)
\end{align*}
which sends each inclusion to the associated Viterbo restriction.
\end{definition}

Let 
$\cF_{A, \bZ}$
be the sheafification of
$\cF^{pre, conv}_A$.
Note that
the sheafification does not change the value of the stalks of
$\cF^{pre, conv}_A$.
\begin{remark}
Up to $\bZ_2$-folding,
$\cF_{A, \bZ}$
defined here coincides with
$\cF_A$
from Definition
\pref{dfn:cF_A}.
\end{remark}

\begin{theorem} \label{thm:main1-Peng}
For any
$\lceil P \rceil \in \Open_{conv}(\cS_Z)$
there is an equivalence
\begin{align*}
\Fuk(\lceil P \rceil)
\simeq
\cF_{A, \bZ}(\lceil P \rceil).
\end{align*}
\end{theorem}
\begin{proof}
We proceed by induction on the poset
$\Open_{conv}(\cS_Z)$.
Since
$\cF_{A, \bZ}$
has the same stalks as
$\cF^{pre, conv}_A$,
the claim holds for minimal elements,
i.e.,
those
which satisfy
$\lceil P \rceil = U_\tau$
for some
$\tau \in \cS_Z$.
Suppose that
the claim holds for all
$\lceil P^\prime \rceil$
with
$\lceil P^\prime \rceil < \lceil P \rceil$.
We may assume that
$P \in \Open_{conv}(\bR^n)$
is covered by some
$P_1, P_2 \in \Open_{conv}(\bR^n)$
with
$\lceil P_1 \rceil, \lceil P_2 \rceil \lneq \lceil P \rceil$.
Indeed,
choose
$p_\tau$
for each
$\tau \in \min(\lceil P \rceil)$
and
cut
$\Conv(\{ p_\tau \}_{\tau \in \min(\lceil P \rceil)})$
into two convex polytopes
$\Delta_{p_1}, \Delta_{p_2}$
by a general hyperplane
which does not contain any
$p_\tau$.
Let
\begin{align*}
P
=
U_r(\Conv(\{ p_\tau \}_{\tau \in \min(\lceil P \rceil)})
=
\Delta_{p_1} \cup \Delta_{p_2}), \
P_1
=
U_r(\Delta_{p_1}), \
P_2
=
U_r(\Delta_{p_2})
\end{align*}
for sufficiently small
$r$.
One can shrink
$P_1, P_2$
so that
\begin{align*}
\overline{P_1 \setminus P_2}
\cap
\overline{P_2 \setminus P_1}
\neq 
\emptyset
\end{align*}
preserving
$\lceil P_1 \rceil, \lceil P_2 \rceil$
and
$ \lceil P_1 \cap P_2 \rceil $.

Choose
a convex potential
$\Phi \colon P \to \bR$
and
sufficiently large
$t$
such that
all the semi-tropicalized hypersurfaces
$H^\chi_t |_P,
H^\chi_t |_{P_1},
H^\chi_t |_{P_2},
H^\chi_t |_{P_1 \cap P_2}$
define Liouville manifolds with respect to
$\lambda_{\Phi, t}$.
Then for any exact Lagrangians
$L_{1, t} \subset H^\chi_t |_{P_1 \setminus P_2},
L_{2, t} \subset H^\chi_t |_{P_2 \setminus P_1}$
we have
\begin{align*}
\Hom_{\Fuk(\lceil P \rceil)}(L_{1, t}, L_{2, t})
=
\Hom_{\Fuk(\lceil P \rceil)}(L_{2, t}, L_{1, t})
=
0,
\end{align*} 
since by Proposition
\pref{prop:trap}
there is no leaf
$T_{x, \tau, t}$
which intersects both
$L_{1, t}, L_{2, t}$.
Note that
the absence of such leaves
$T_{x, \tau, t}$
implies that of Reeb chords connecting the ideal boundaries of
$L_{1, t}, L_{2, t}$. 
Applying Corollary
\pref{cor:key}
to
$P = P_1 \cup P_2$,
we obtain a fiber product
\begin{align*}
\begin{gathered}
\xymatrix{
\Fuk(\lceil P \rceil) \ar[d]_{} \ar[r]^{} & \Fuk(\lceil P_1 \rceil) \ar[d]_{}\\
\Fuk(\lceil P_2 \rceil)  \ar[r]^{} & \Fuk(\lceil P_1 \cap P_2 \rceil).
}
\end{gathered}
\end{align*}
whose arrows are Viterbo restrictions. 
\end{proof}


\end{document}